\documentclass[11pt,a4paper,reqno]{amsart}
\usepackage{fancyhdr}
\usepackage{ifpdf}
\usepackage{amsmath,amsfonts,amsthm,amsopn,amssymb}
\usepackage{verbatim,wasysym,mathrsfs}
\usepackage[left=2.6cm,right=2.6cm,top=3.1cm,bottom=3.1cm]{geometry}
\usepackage{cite,marginnote}
\usepackage{color,enumitem,graphicx}
\usepackage[colorlinks=true,urlcolor=blue, citecolor=red,linkcolor=blue,
linktocpage,pdfpagelabels, bookmarksnumbered,bookmarksopen]{hyperref}
\usepackage[english]{babel}
\usepackage[T1]{fontenc}
\usepackage{lmodern}
\newtheorem{Thm}{Theorem}[section]
\newtheorem{Lem}[Thm]{Lemma}

\newtheorem{Cor}[Thm]{Corollary}
\newtheorem{Prop}[Thm]{Proposition}

\theoremstyle{definition}

\newtheorem{Rem}{Remark}[section]

\newtheorem*{ackn}{Acknowledegments}

\newcommand{\R}{{\mathbb R}}
\newcommand{\weakto}{{\rightharpoonup}}

\numberwithin{equation}{section}

\makeindex

\makeatletter
\def\@makefnmark{}
\makeatother

\allowdisplaybreaks

\begin{document}

\title[a slightly supercritical Choquard equation]{Nonexistence of single-bubble solutions for a slightly supercritical Choquard equation}

\author[]{Jinkai Gao}
\address[J.~Gao]{\newline\indent School of Mathematical Sciences
\newline\indent Nankai University
\newline\indent Tianjin, 300071, PRC.}
    \email{\href{mailto:jinkaigao@mail.nankai.edu.cn}{jinkaigao@mail.nankai.edu.cn}}




\begin{abstract}
In this paper, we consider the existence of positive solutions to the following slightly supercritical Choquard equation
\begin{equation*}
    \begin{cases}
        -\Delta u=\displaystyle\Big(\int\limits_{\Omega}\frac{u^{2^*_{\alpha}+\varepsilon}(y)}{|x-y|^\alpha}dy\Big)u^{2^*_{\alpha}-1+\varepsilon},\quad u>0\ \  &\mbox{in}\ \Omega,\\
 \quad \ \ u=0 \ \  &\mbox{on}\ \partial \Omega,
    \end{cases}
\end{equation*}
where $N\geq 3$, $\Omega$ is a smooth bounded domain in $\mathbb{R}^{N}$, $\alpha\in (0,N)$, $2^*_{\alpha}:=\frac{2N-\alpha}{N-2}$ is the upper critical exponent in the sense of Hardy-Littlewood-Sobolev inequality and $\varepsilon>0$ is a small parameter. In contrast with the slightly subcritical Choquard equation studied by Chen and Wang (Calculus of Variations and Partial Differential Equations, 63:235, 2024), we find that there is no chance to construct a family of single-bubble solutions as $\varepsilon\to 0^{+}$.
\end{abstract}

\maketitle


\section{Introduction and Main result}
In this paper, we are interested in the existence of positive solutions to the following near critical Choquard equation under zero Dirichlet boundary condition
\begin{equation}\label{slightly supercritical choquard equation}
    \begin{cases}
        -\Delta u=\displaystyle\Big(\int\limits_{\Omega}\frac{u^{2^*_{\alpha}+\varepsilon}(y)}{|x-y|^\alpha}dy\Big)u^{2^*_{\alpha}-1+\varepsilon},\quad u>0\ \  &\mbox{in}\ \Omega,\\
 \quad \ \ u=0 \ \  &\mbox{on}\ \partial \Omega,
    \end{cases}
\end{equation}
where $N\geq 3$, $\Omega$ is a smooth bounded domain in $\mathbb{R}^{N}$, $\alpha\in (0,N)$, $2^*_{\alpha}:=\frac{2N-\alpha}{N-2}$ is the upper critical exponent in the sense of Hardy-Littlewood-Sobolev inequality (see Remark \ref{remark of HLS} below) and $\varepsilon$ is a small parameter.

The Choquard equation, first introduced in the pioneering work of Fr\"{o}hlich \cite{FrohlichHerbert} and Pekar \cite{Pekar}, has several physical origins such as quantum theory \cite{Penrose1998QuantumCE,Meeron1966PhysicsOM} and Hartree-Fock theory \cite{Lieb1977,Lieb1976SAM}. Apart from the physical motivations, Choquard equation has been broadly investigated from a mathematical point of view due to the existence of the nonlocal term. We refer to  Moroz-Van Schaftingen \cite{Moroz2017JFPTA} and references therein for a broad survey.

We now proceed to describe some previous work related to the problem \eqref{slightly supercritical choquard equation}. In the critical case when $\varepsilon=0$. The problem \eqref{slightly supercritical choquard equation} reduces to the upper critical Choquard equation
\begin{equation}\label{upper critical Choquard equation}
\begin{cases}
    -\Delta u=\left(\displaystyle{\int_{\Omega}}\frac{u^{2^{*}_{\alpha}}(y)}{|x-y|^{\alpha}}dy\right){u}^{2^{*}_{\alpha}-1},\quad u>0\ \ &\text{in}\ \Omega,\\
    \quad \ \ u=0 \ \  &\mbox{on}\ \partial \Omega,
\end{cases}
\end{equation}
which arises as the Euler-Lagrange equation of the variational problem
\begin{equation}\label{definition of SHL}
    S_{HL}(\Omega):=\inf_{u\in H^{1}_{0}(\Omega)\setminus\{0\}}\frac{\int_{\Omega}|\nabla u|^{2}dx}{\left(\int_{\Omega}\int_{\Omega
    }\frac{|u(x)|^{2^*_{\alpha}}|u(y)|^{2^*_{\alpha}}}{|x-y|^{\alpha}})dxdy\right)^{\frac{1}{2^*_{\alpha}}}}.
\end{equation}
The existence of solutions to \eqref{upper critical Choquard equation} is strongly influenced by the geometry and topology of the domain $\Omega$. Indeed, when $\Omega$ is a star-shaped domain, Poho\v{z}aev identity shows that problem \eqref{slightly supercritical choquard equation} admits no nontrivial solutions. However, for non-star-shaped domains such as annular domain, Goel, R\u{a}dulescu and Sreenadh \cite{Goel2020} proved the existence of positive, high-energy solutions. Using the reduction method, Ghimenti, Huang and Pistoia \cite{Ghimenti2025DCDS} recently constructed a single-bubble solution that blows up at the origin in a pierced domain. 

Particularly, when $\Omega=\mathbb{R}^{N}$, the results in \cite{Lei2018,Du2019,Guo2019,Gao2018} establish that the constant $S_{HL}(\mathbb{R}^{N}):=S_{HL}$ is achieved if and only if the solution takes the form $\bar{U}_{\xi,\lambda}(x)$, where
\begin{equation}\label{solution of choquard equation}
    \bar{U}_{\xi,\lambda}(x) = S^{\frac{(N-\alpha)(2-N)}{4(N-\alpha+2)}} (C_{N,\alpha})^{\frac{2-N}{2(N-\alpha+2)}} [N(N-2)]^{\frac{N-2}{4}} U_{\xi,\lambda}(x) := \bar{C}_{N,\alpha} U_{\xi,\lambda}(x).
\end{equation}
Here, 
\begin{equation}
    S = \pi N(N-2)\left(\frac{\Gamma(N/2)}{\Gamma(N)}\right)^{2/N}
\end{equation}
denotes the sharp Sobolev constant, where $\Gamma(\cdot)$ is the Gamma function, $C_{N,\alpha}$ is the sharp constant in the Hardy-Littlewood-Sobolev (HLS) inequality as defined in \eqref{definition of C N alpha}, and $U_{\xi,\lambda}(x)$ represents the Aubin-Talenti bubble given by
\begin{equation}\label{Aubin-Talenti bubble}
    U_{\xi,\lambda}(x) = \frac{\lambda^{\frac{N-2}{2}}}{(1+\lambda^{2}|x-\xi|^{2})^{\frac{N-2}{2}}}, \quad \lambda \in \mathbb{R}^{+}, \ x, \xi \in \mathbb{R}^{N}.
\end{equation}
Furthermore, $\bar{U}_{\xi,\lambda}(x)$ constitutes the unique family of solutions to \eqref{upper critical Choquard equation} with $\Omega=\mathbb{R}^{N}$ and the sharp constant 
\begin{equation}\label{eq SHL with S anf C-N-alpha}
    S_{HL} = S (C_{N,\alpha})^{-\frac{1}{2^{*}_{\alpha}}}.
\end{equation}
Direct computation shows that $U_{\xi,\lambda}(x)$ satisfies
\begin{equation}\label{eq talenti bubble for choquard equation}
    -\Delta U_{\xi,\lambda} = \bar{C}_{N,\alpha}^{22^*_{\alpha}-2} \left( \int_{\mathbb{R}^{N}} \frac{U_{\xi,\lambda}^{2^{*}_{\alpha}}(y)}{|x-y|^{\alpha}} dy \right) U_{\xi,\lambda}^{2^{*}_{\alpha}-1} \quad \text{in } \mathbb{R}^N,
\end{equation}
and
\begin{equation}\label{eq equation for U-xi-Lambda}
    \int_{\mathbb{R}^{N}} \frac{U_{\xi,\lambda}^{2^{*}_{\alpha}}(y)}{|x-y|^{\alpha}} dy = \frac{N(N-2)}{\bar{C}_{N,\alpha}^{22^*_{\alpha}-2}} U_{\xi,\lambda}^{2^*-2^*_{\alpha}}(x).
\end{equation}
However, it is known that while $S_{HL}(\Omega) = S_{HL}$ for general domains, $S_{HL}(\Omega)$ is never achieved except when $\Omega=\mathbb{R}^{N}$, as shown in \cite[Lemma 1.3]{Gao2018}.

On the other hand, considerable interest has developed around the following Choquard type Br\'{e}zis-Nirenberg problem 
\begin{equation}\label{critical Choquard equation}
    \begin{cases}
        -\Delta u=\left(\int_{\Omega}\frac{u^{2^*_{\alpha}}(y)}{|x-y|^\alpha}dy\right)u^{2^*_{\alpha
        }-1}+\varepsilon u,\quad u>0,\ \  &\mbox{in}\ \Omega,\\
  \quad\ \  u=0, \ \  &\mbox{on}\ \partial \Omega,
    \end{cases}
\end{equation}
Gao and Yang \cite{Gao2018} employed variational methods to establish the existence, nonexistence and multiplicity of solutions for \eqref{critical Choquard equation}, extending the celebrated results of Br\'{e}zis-Nirenberg \cite{Brezis1983CPAM} to nonlocal cases. Furthermore, using the reduction method, the authors in \cite{Yang2023BlowUp,Yang2023ExistenceOC,chen2024blowingupsolutionschoquard} proved the existence and characterized the blow-up behavior of single-bubble solutions for \eqref{critical Choquard equation}. Regarding local uniqueness and eigenvalue problems for \eqref{critical Choquard equation}, see \cite{Squassina-Yang-Zhao2023,pan2024qualitativeanalysiseigenvalueproblem}, while sign-changing solutions were also constructed in \cite{LIU2025128726} recently.

In the subcritical case when $\varepsilon<0$, the problem \eqref{slightly supercritical choquard equation} is always solvable, since a least energy solution can be found by solving the variational problem
\begin{equation}\label{minimizing problem of least energy solution}
    \inf_{u\in H^{1}_{0}(\Omega)\setminus\{0\}}\frac{\int_{\Omega}|\nabla u|^{2}dx}{\left(\int_{\Omega}\int_{\Omega}\frac{|u(x)|^{2^*_{\alpha}+\varepsilon}|u(y)|^{2^*_{\alpha}+\varepsilon}}{|x-y|^{\alpha}}dxdy\right)^{\frac{1}{2^*_{\alpha}+\varepsilon}}}.
\end{equation}
Recently, Chen and Wang \cite{chen2024blowingupsolutionschoquardtype} investigated the existence and blow-up behavior of single-bubble solutions for the slightly subcritical Choquard equation \eqref{slightly supercritical choquard equation} as $\varepsilon\to 0^{-}$. To state their results more precisely, we introduce the following notations.

We begin by recalling the definitions of Green function and Robin function for the domain $\Omega$. The Green function $G(x,\cdot)$ for the negative Laplacian on $\Omega$ satisfies
\begin{equation}
\begin{cases}
-\Delta G(x,\cdot)= \delta_x &{\text{in}~\Omega}, \\
\quad \  \ G(x,\cdot)=0 &{\text{on}~\partial \Omega},
\end{cases}
\end{equation}
where $\delta_x$ denotes the Dirac function at $x\in\Omega$. The Green function admits the decomposition
\begin{equation*}
G(x,y)=S(x,y)-H(x,y), ~~(x,y)\in \Omega\times \Omega,
\end{equation*}
where 
\begin{equation*}
S(x,y)=\frac{1}{(N-2)\omega_{N}|x-y|^{N-2}}
\end{equation*}
is the singular part, which is also the fundamental solution to the negative Laplace equation on $\R^{N}$, $\omega_{N}=\frac{2\pi^{N/2}}{\Gamma(N/2)}$ is the measure of the unit sphere on $\R^{N}$ and $H(x,y)$
is the regular part of $G(x,y)$ satisfying 
\begin{equation}
 \begin{cases}
-\Delta H(x,\cdot)=0 &{\text{in}~\Omega}, \\
\quad \  \ H(x,\cdot)=S(x,\cdot) &{\text{on}~\partial \Omega}.
\end{cases}   
\end{equation}
Furthermore, we denote the leading term of $H$ as
\begin{equation}\label{definition of Robin function}
    R(x):=H(x,x), ~~x\in \Omega,
\end{equation}
 which is called the Robin function of domain $\Omega$ at point $x$. Next, for any $(\xi,\lambda)\in \Omega\times(0,\infty)$, let $P\bar{U}_{\xi,\lambda}$ be the projection of $\bar{U}_{\xi,\lambda}$ from $H^{1}(\Omega)$ onto $H^{1}_{0}(\Omega)$, i.e.,
 \begin{equation}\label{definition of PbarU}
     \begin{cases}
         \Delta P\bar{U}_{\xi,\lambda}=\Delta \bar{U}_{\xi,\lambda}&{\text{in}~\Omega},\\
         \quad P\bar{U}_{\xi,\lambda}=0 &{\text{on}~\partial \Omega},
     \end{cases}
 \end{equation}
and $\bar{\varphi}_{\xi_{\varepsilon},\lambda_{\varepsilon}}:=\bar{U}_{\xi,\lambda}-P\bar{U}_{\xi,\lambda}$. Finally, we define the subspace
\begin{equation*}
\begin{split}
E_{\xi,\lambda}=\Big\{v\in H^1_0(\Omega):  \Big\langle P\bar{U}_{\xi,\lambda},v \Big\rangle=\Big\langle \frac{\partial P\bar{U}_{\xi,\lambda}}{\partial \lambda},v \Big\rangle=\Big\langle \frac{\partial P\bar{U}_{\xi,\lambda}}{\partial \xi_i},v\Big\rangle=0,~\mbox{for}~i=1,\cdots,N\Big\},
\end{split}
\end{equation*}
where $\langle\cdot,\cdot\rangle$ denotes the inner product in the Sobolev space $H^1_0(\Omega)$. 

Chen and Wang \cite{chen2024blowingupsolutionschoquardtype} established a complete characterization of single-bubble solutions for slightly subcritical Choquard equation. On one hand, they proved that if $u_{\varepsilon}$ is a family of solutions to \eqref{slightly supercritical choquard equation} with $\varepsilon<0$ and satisfies 
\begin{equation}
    |\nabla u_{\varepsilon}|^{2}\weakto S_{HL}^{\frac{2^*_{\alpha}}{2^*_{\alpha}-1}}\delta_{\xi_{0}} \text{~~in the the sense of measure~},
\end{equation}
then $\xi_{0}\in\Omega$ and $\xi_{0}$ is a critical point of the Robin function $R(\cdot)$. On the other hand, through reduction methods, they proved that for any stable critical point $\xi_{0}$ (see Definition 2.1 in \cite{chen2024blowingupsolutionschoquardtype}) of the Robin function $R(\cdot)$, there exists a family of solutions with the form
 \begin{equation}\label{single-bubble}
        u_{\varepsilon}=P\bar{U}_{\xi_{\varepsilon},\lambda_{\varepsilon}}+w_{\varepsilon},
    \end{equation}
where $\xi_{\varepsilon}\to\xi_{0} $, $\lambda_{\varepsilon}\to+\infty$, $w_{\varepsilon}\in E_{\xi_{\varepsilon},\lambda_{\varepsilon}}$ and $w_{\varepsilon}\to 0$ in $H^{1}_{0}(\Omega)$ as $\varepsilon\to0^{-}$.

In the supercritical case when $\varepsilon>0$.  As far as we know, the problem \eqref{slightly supercritical choquard equation} has not been studied in literature yet. Since the standard variational approach is no longer applicable, a natural question arises: 
Can reduction methods be employed to construct single-bubble solutions of the form \eqref{single-bubble} for \eqref{slightly supercritical choquard equation}? Surprisingly, our main theorem gives a negative answer to this question, as stated below.
\begin{Thm}\label{thm-1}
Assume that $\alpha\in(0,\min\{4,N\})$, then problem \eqref{slightly supercritical choquard equation} has no solution $u_{\varepsilon}$ such that
    \begin{equation}\label{decomposition-1}
        u_{\varepsilon}=P\bar{U}_{\xi_{\varepsilon},\lambda_{\varepsilon}}+w_{\varepsilon}
    \end{equation}
    with $\xi_{\varepsilon}\in\Omega$, $\lambda_{\varepsilon}\in\R^{+}$, $\lambda_{\varepsilon}d(\xi_{\varepsilon},\partial\Omega)\to+\infty$ and $w_{\varepsilon}\to 0$ in $H^{1}_{0}(\Omega)$ as $\varepsilon\to0^{+}$.
\end{Thm}
\begin{Rem}
    \begin{enumerate}
        \item The condition $\alpha<4$ is necessary for controlling the error estimate of the remainder term $w_{\varepsilon}$, see Lemma \ref{lemma-5}.
        \item Our results reveal a fundamental dichotomy between subcritical and supercritical cases.
    \end{enumerate}
\end{Rem}
\begin{Rem}
    \begin{enumerate}
        \item The result in Theorem \ref{thm-1} generalizes the earlier result for the local problem in \cite{Ben2003CCM}. The main difficulty arises from the nonlocal term and some new estimates need to be established. We would like to point out that the symmetry property of double integrals and the application of Hardy-Littlewood-Sobolev inequality play a crucial role in the computation.
        \item In this paper, we only focus on the existence of single-bubble solutions and from \cite{delpino2003CV,delPino2002JDE,delPino2003BLMS,Teresa2016}, one may ask that, does \eqref{slightly supercritical choquard equation} possesses multi-bubble solutions in a domain with some small holes ? This will be addressed in future research.
    \end{enumerate}
\end{Rem}
The rest of the paper is organized as follows. After introducing some notations, we recall some preliminaries in Section \ref{section-Preliminaries}. Section \ref{section-Someusefulestimates} is devoted to establish some useful estimates on $u_{\varepsilon}$ and $w_{\varepsilon}$. Finally, the proof of Theorem \ref{thm-1} via contradiction is presented in Section \ref{section-prooftheorem}.

\smallskip

\noindent\textbf{Notations.}
Throughout this paper, we use the following notations.
\begin{enumerate}
    \item We use $\mathcal{D}^{1,2}(\R^{N}):=\left\{u\in L^{2^*}(\R^{N}):~\nabla u\in L^{2}(\R^{N})\right\}$ to denote the homogeneous Sobolev space.
    In addition, we use $\|u\|_{H^{1}_{0}(\Omega)}=\left(\int_{\Omega}|\nabla u|^{2}dx\right)^{1/2}$ to denote the norm in $H^{1}_{0}(\Omega)$ and $\langle\cdot,\cdot\rangle$ means the corresponding inner product.  
    \item We use $C$ to denote various positive constant and use $C_{1}=o(\varepsilon)$ and $C_{2}=O(\varepsilon)$ to denote $C_{1}/\varepsilon\to0$ and $|C_{2}/\varepsilon|\leq C $ as $\varepsilon\to0$ respectively.
     \item  Let $f,g: X\to \R^{+}$ be two nonnegative function defined on some set $X$. we write $f\lesssim g$ or $g\gtrsim f$, if there exists a constant $C>0$ independent on $x$ such that $f(x)\leq C g(x)$ for any $x\in X$ and $f\sim g$  means that $f\lesssim g$ and $g\lesssim f$.
\end{enumerate}

\section{Preliminaries}\label{section-Preliminaries}

This section presents necessary preliminaries. We begin with the Hardy-Littlewood-Sobolev (HLS) inequality:
\smallskip

\noindent\textbf{Theorem A.} \cite[Theorem 4.3]{Lieb2001} \label{lema 2.1} \emph{Suppose $\alpha\in(0,N)$ and $\theta,\,r>1$ with $\frac{1}{\theta}+\frac{1}{r}+\frac{\alpha}{N}=2$. Let $f\in L^{\theta}(\R^N)$ and $g\in L^{r}(\R^N)$, there exists a sharp constant $C(\theta,r,\alpha,N)$, independent of $f$ and $g$, such that
\begin{align}\label{HLS}
\displaystyle{\int_{\R^N}}\displaystyle{\int_{\R^N}}\frac{f(x)g(y)}{|x-y|^{\alpha}}dxdy\leq C(\theta,r,\alpha,N)\|f\|_{L^{\theta}(\R^N)}\|g\|_{L^{r}(\R^N)}.
\end{align}
If $\theta=r=\frac{2N}{2N-\alpha}$, then
\begin{equation}\label{definition of C N alpha}
    C(\theta,r,\alpha,N)=C_{N,\alpha}:=\pi^{\frac{\alpha}{2}}\frac{\Gamma\left(\frac{N-\alpha}{2}\right)}{\Gamma\left(N-\frac{\alpha}{2}\right)}\left(\frac{\Gamma(N)}{\Gamma\left(\frac{N}{2}\right)}\right)^{\frac{N-\alpha}{N}}.
\end{equation}
In this case, the equality in \eqref{HLS} holds if and only if $f\equiv (const.)\, g$, where
$$g(x)=A\left(\frac{1}{\gamma^{2}+|x-a|^{2}}\right)^{\frac{2N-\alpha}{2}},\quad \text{for some $A\in \mathbb{C}$, $0\neq\gamma\in\R$ and $a\in\R^N$.}$$}

\begin{Rem}\label{remark of HLS}
\begin{enumerate}
    \item By using HLS inequality and Sobolev inequality, we have
\begin{equation}
\begin{split}
   \left(\displaystyle{\int_{\R^N}}\displaystyle{\int_{\R^N}}\frac{|u(x)|^{2^*_{\alpha}}|u(y)|^{2^*_{\alpha}}}{|x-y|^{\alpha}}dxdy\right)^{\frac{1}{2^*_{\alpha}}}&\leq \left(C_{N,\alpha}\right)^{\frac{1}{2^*_{\alpha}}}\left(\displaystyle{\int_{\R^N}}|u(x)|^{2^{*}}dx\right)^{\frac{2}{2^{*}}}\\
   &\leq S^{-1}\left(C_{N,\alpha}\right)^{\frac{1}{2^*_{\alpha}}}\int_{\R^{N}}|\nabla u(x)|^{2}dx,
\end{split} 
\end{equation}
for any given $u\in \mathcal{D}^{1,2}(\R^{N})$. 
    \item From HLS inequality, the integral
$$\displaystyle{\int_{\R^N}}\displaystyle{\int_{\R^N}}\frac{|u(x)|^{q}|u(y)|^{q}}{|x-y|^{\alpha}}dxdy$$
is well-defined in $H^{1}(\R^N)\times H^{1}(\R^N)$ if $\frac{2N-\alpha}{N}\leq q\leq\frac{2N-\alpha}{N-2}$. Hence, it's natural to call $2_{\alpha}:=\frac{2N-\alpha}{N}$ the lower Hardy-Littlewood-Sobolev critical exponent and $2^*_{\alpha}:=\frac{2N-\alpha}{N-2}$ the upper Hardy-Littlewood-Sobolev critical exponent. 
\end{enumerate} 
\end{Rem}

Next, for any $\xi\in\Omega$ and $\lambda\in \R^{+}$, we define $PU_{\xi,\lambda}$ is the projection of $U_{\xi,\lambda}$ onto $H^{1}_{0}(\Omega)$, i.e. 
$PU_{\xi,\lambda}:=U_{\xi,\lambda}-\varphi_{\xi,\lambda}\in H^{1}_{0}(\Omega)$, where $\varphi_{\xi,\lambda}$ is the harmonic extension of $U_{\xi,\lambda}|_{\partial\Omega}$ to $\Omega$
\begin{equation}
  \begin{cases}
-\Delta \varphi_{\xi,\lambda}=0, \quad{\text{~in~}\Omega},\\
\ \varphi_{\xi,\lambda}|_{\partial\Omega}=U_{\xi,\lambda}|_{\partial\Omega}.  
\end{cases}
\end{equation}
From \eqref{definition of PbarU}, it is easy to see that $P\bar{U}_{\xi,\lambda}=\bar{C}_{N,\alpha}P{U}_{\xi,\lambda}$ and $\bar{\varphi}_{\xi,\lambda}=\bar{C}_{N,\alpha}{\varphi}_{\xi,\lambda}$. By strong maximum principle, we have
 \begin{equation}
  \varphi_{\xi,\lambda}(x)\leq U_{\xi,\lambda}(x):=\left(\frac{\lambda}{1+\lambda^{2}|x-\xi|^{2}}\right)^{\frac{N-2}{2}}.
\end{equation}
Moreover, we have the following estimates, see \cite[Appendix A and Appendix B]{Rey1990}.
\begin{Lem}\label{estimate of U-lambda-a and psi-lambda-a 1}
Assume that $\xi\in\Omega$, $\lambda\in \R^{+}$ and $j=1,\cdots,N$, we have   \begin{equation*}\begin{aligned}
&\frac{\partial U_{\xi,\lambda}(x)}{\partial \xi_j} =-(N-2)\lambda^{\frac{N+2}{2}}\frac{x_{j}-\xi_{j}}{\left(1+\lambda^{2}|x-\xi|^{2}\right)^{\frac{N}{2}}}=O\big(\lambda U_{\xi,\lambda}\big), \\
&\frac{\partial U_{\xi,\lambda}(x)}{\partial\lambda} =\frac{N-2}2\lambda^{\frac{N-4}2}\frac{1-\lambda^2|x-\xi|^2}{(1+\lambda^2|x-\xi|^2)^{\frac N2}}=O\Big(\frac{U_{\xi,\lambda}}\lambda\Big).
\end{aligned}\end{equation*}
\end{Lem}

\begin{Lem}\label{estimate of U-lambda-a and psi-lambda-a 2}
Assume that $\xi\in\Omega$ and $\lambda\in \R^{+}$, we have 
    \begin{equation*}
        \begin{aligned}
\int_{\R^{N}}|\nabla U_{\xi,\lambda}|^{2}&=(N(N-2))^{-\frac{N-2}{2}}S^{\frac{N}{2}},\\
\int_{\R^{N}}|U_{\xi,\lambda}|^{2^*}&=(N(N-2))^{-\frac{N}{2}}S^{\frac{N}{2}},\\
\int_{\Omega}|\nabla U_{\xi,\lambda}|^{2}&=(N(N-2))^{-\frac{N-2}{2}}S^{\frac{N}{2}}+O\left(\frac{1}{(\lambda d)^{N-2}}\right),\\
\int_{\Omega}|\nabla PU_{\xi,\lambda}|^{2}&=(N(N-2))^{-\frac{N-2}{2}}S^{\frac{N}{2}}+O\left(\frac{1}{(\lambda d)^{N-2}}\right),\\
\int_{\R^{N}\setminus\Omega}|\nabla U_{\xi,\lambda}|^{2}&=O\left(\frac{1}{(\lambda d)^{N-2}}\right),\\
\int_{\R^{N}\setminus\Omega}|U_{\xi,\lambda}|^{2^*}&=O\left(\frac{1}{(\lambda d)^{N}}\right),\\
        \end{aligned}
    \end{equation*}
where $d=\text{dist}(\xi,\partial\Omega)$ is the distance between $\xi$ and boundary $\partial\Omega$.
\end{Lem}

\begin{Lem}\label{estimate of U-lambda-a and psi-lambda-a 4}
Assume that $\xi\in\Omega$ and $\lambda\in \R^{+}$, we have
\begin{equation*}
 \begin{aligned}
\varphi_{\xi,\lambda}(x)& =\frac{(N-2)\omega_N}{\lambda^{\frac{N-2}2}}H(\xi,x)+O\Big(\frac1{\lambda^{\frac{N+2}2}d^N}\Big), \\
\frac{\partial\varphi_{\xi,\lambda}(x)}{\partial\lambda}& =-\frac{(N-2)^{2}\omega_N}{2\lambda^{\frac N2}}H(\xi,x)+O\Big(\frac1{\lambda^{\frac{N+4}2}d^N}\Big), \\
\frac{\partial\varphi_{\xi,\lambda}(x)}{\partial \xi_j}& =\frac{(N-2)\omega_N}{\lambda^{\frac{N-2}2}}\frac{\partial H(\xi,x)}{\partial \xi_j}+O\Big(\frac{1}{\lambda^{\frac{N+2}2}d^{N+1}}\Big),
\end{aligned}      
\end{equation*}
and 
\begin{equation*}
\begin{aligned}
    &\parallel\varphi_{\xi,\lambda}\parallel_{L^{2^{*}}}=O\Big(\frac{1}{(\lambda d)^{\frac{N-2}{2}}}\Big),\|\frac{\partial\varphi_{\xi,\lambda}}{\partial\xi_{j}}\|_{L^{2^{*}}} =O\left(\frac1{\lambda^{\frac{N-2}2}d^{\frac{N}{2}}}\right), \|\frac{\partial\varphi_{\xi,\lambda}}{\partial\lambda}\|_{L^{2^{*}}} =O\left(\frac1{\lambda^{\frac{N}2}d^{\frac{N-2}{2}}}\right),\\
&\|\varphi_{\xi,\lambda}\|_{L^{\infty}} =O\left(\frac1{\lambda^{\frac{N-2}2}d^{N-2}}\right), \|\frac{\partial\varphi_{\xi,\lambda}}{\partial\xi_{j}}\|_{L^{\infty}} =O\left(\frac1{\lambda^{\frac{N-2}2}d^{N-1}}\right), \|\frac{\partial\varphi_{\xi,\lambda}}{\partial\lambda}\|_{L^{\infty}} =O\left(\frac1{\lambda^{\frac{N}2}d^{N-2}}\right).
\end{aligned}
\end{equation*}
Moreover
\begin{equation*}
    \int_{\Omega}|\nabla \varphi_{\xi,\lambda}|^{2}dx=O\left(\frac{1}{(\lambda d)^{N-2}}\right).
\end{equation*}
 where $d=\text{dist}(\xi,\partial\Omega)$ is the distance between $\xi$ and boundary $\partial\Omega$.
\end{Lem}
Moreover, we recall the following elementary inequality, see \cite[Lemma 2.2]{Iacopetti-2016-CCM}.
\begin{Lem}\label{lem inequality 2}
Let $q$ be a positive real number. There exists a positive constant $c$, depending only on $q$, such that for any $a, b\in \mathbb R$
\begin{equation}\label{lem inequality 2-1}
||a+b|^q-|a|^q| \leq
\begin{cases}
c(q) \min\{|b|^q, |a|^{q-1}|b|\}, &\ \hbox{if}\ 0<q<1,\\
c(q) (|a|^{q-1}|b|+|b|^q), & \ \hbox{if}\ q\geq1.
\end{cases}
\end{equation}
Moreover if $q>2$ then
\begin{equation}\label{lem inequality 2-2}
\left||a+b|^q-|a|^q-q |a|^{q-2}ab\right|\leq c(q)\left(|a|^{q-2}|b|^2+|b|^q\right).
\end{equation}
\end{Lem}

Finally, the non-degeneracy of the solution $\bar{U}_{\xi,\lambda}$ to \eqref{upper critical Choquard equation} also plays a crucial role and we summarize the non-degeneracy results as follows.
\smallskip

\noindent\textbf{Theorem B.}\label{thm nondegeneracy} \cite[Theorem 1.4]{li2024nondegeneracypositivebubblesolutions}\emph{
Assume that $N\geq 3$ and $\alpha\in (0,N)$, then the linearized operator of \eqref{upper critical Choquard equation} at $\bar{U}_{\xi,\lambda}$ defined by
\begin{equation*}
       L\phi=-\Delta\phi-(2^{*}_\alpha-1)\bar{U}_{\xi,\lambda}^{2^{*}_\alpha-2}\phi\left(\displaystyle{\int_{\R^N}}\frac{\bar{U}_{\xi,\lambda}^{2^{*}_\alpha}(y)}{|x-y|^{\alpha}}dy\right)-2^{*}_{\alpha} \bar{U}_{\xi,\lambda}^{2^{*}_\alpha-1}\left(\displaystyle{\int_{\R^N}}\frac{\bar{U}_{\xi,\lambda}^{2^{*}_{\alpha}-1}(y)\phi(y)}{|x-y|^{\alpha}}dy\right)
\end{equation*}
 only admits solutions in $\mathcal{D}^{1,2}(\R^N)$ of the form
$$\phi=\bar{a} D_{\lambda}\bar{U}_{\xi,\lambda}+\vec{b}\cdot\nabla_{\xi}\bar{U}_{\xi,\lambda},$$
where $\bar{a}\in\R$ and $\vec{b}\in\R^N$.}

\section{Some useful estimates}\label{section-Someusefulestimates}
Since $u_{\varepsilon}$ satisfies assumption \eqref{decomposition-1}, then using the same argument in \cite[Proposition 2]{Rey1990}, there is a unique way to choose $ \alpha_{\varepsilon}$, $\xi_{\varepsilon}$ and $\lambda_{\varepsilon}$ such that $u_{\varepsilon}$ has the following orthogonal decomposition
 \begin{equation}\label{eq orthogonal decomposition}
     u_{\varepsilon}=\alpha_{\varepsilon}(P\bar{U}_{\xi_{\varepsilon},\lambda_{\varepsilon}}+w_{\varepsilon}),
 \end{equation}
 where $\alpha_{\varepsilon}\in\R^{+}$, $\xi_{\varepsilon}\in\Omega$, $\lambda_{\varepsilon}\in\R^{+}$, $w_{\varepsilon}\in E_{\xi_{\varepsilon},\lambda_{\varepsilon}}$ with
 \begin{equation}\label{eq condition parameter}
     \alpha_{\varepsilon}\to 1,\quad \lambda_{\varepsilon}d_{\varepsilon}:=\lambda_{\varepsilon}\text{dist}(\xi_{\varepsilon},\partial\Omega)\to+\infty,\quad\|w_{\varepsilon}\|_{H^{1}_{0}(\Omega)}\to0 \text{~as~}\varepsilon\to 0.
 \end{equation}
In the following, we always assume that $u_{\varepsilon}$ is written as in \eqref{eq orthogonal decomposition} and \eqref{eq condition parameter}. Now, we will give some estimates on $u_{\varepsilon}$ and $w_{\varepsilon}$.
\begin{Lem}\label{lemma-1}
    It holds that 
    \begin{equation}\label{eq lemma-1}
        \int_{\Omega}|\nabla u_{\varepsilon}|^{2}(x)dx=\int_{\Omega}\int_{\Omega}\frac{u_{\varepsilon}^{2^*_{\alpha}+\varepsilon}(y)u_{\varepsilon}^{2^*_{\alpha}+\varepsilon}(x)}{|x-y|^{\alpha}}dydx\to S_{HL}^{\frac{2^*_{\alpha}}{2^*_{\alpha}-1}}\text{~~as~~}\varepsilon\to0,
    \end{equation}
    where $S_{HL}$ is the sharp constant related to the HLS inequality defined by \eqref{definition of SHL}.
\end{Lem}
\begin{proof}
    First, from \eqref{slightly supercritical choquard equation} and \eqref{eq orthogonal decomposition} we have
    \begin{equation}\label{eq proof lemma-1-1}
        \int_{\Omega}\int_{\Omega}\frac{u_{\varepsilon}^{2^*_{\alpha}+\varepsilon}(y)u_{\varepsilon}^{2^*_{\alpha}+\varepsilon}(x)}{|x-y|^{\alpha}}dydx=\int_{\Omega}|\nabla u_{\varepsilon}|^{2}(x)dx=\alpha_{\varepsilon}^{2}\left(\int_{\Omega}|\nabla P\bar{U}_{\xi_{\varepsilon},\lambda_{\varepsilon}}|^{2}(x)dx+\int_{\Omega}|\nabla w_{\varepsilon}|^{2}dx\right).    \end{equation}
        On the other hand, from Lemma \ref{estimate of U-lambda-a and psi-lambda-a 2}, \eqref{solution of choquard equation} and \eqref{eq SHL with S anf C-N-alpha}
        \begin{equation}\label{eq proof lemma-1-2} 
           \int_{\Omega}|\nabla P\bar{U}_{\xi_{\varepsilon},\lambda_{\varepsilon}}|^{2}(x)dx=S_{HL}^{\frac{2^*_{\alpha}}{2^*_{\alpha}-1}}+O\left(\frac{1}{(\lambda_{\varepsilon}d_\varepsilon)^{N-2}}\right).
        \end{equation}
Thus \eqref{eq lemma-1} follows from \eqref{eq condition parameter}, \eqref{eq proof lemma-1-1} and \eqref{eq proof lemma-1-2}.
\end{proof}

\begin{Lem}\label{lemma-2}
    It holds that
    \begin{equation}\label{eq lemma-2}
        \lambda_{\varepsilon}^{\varepsilon}\to 1\text{~~as~~}\varepsilon\to0.
    \end{equation}
\end{Lem}
\begin{proof}
First, by \eqref{slightly supercritical choquard equation}, \eqref{eq orthogonal decomposition} and Lemma \ref{lem inequality 2}, we have
    \begin{equation}\label{eq lemma-2 proof-1}
    \begin{aligned}
        &\int_{\Omega}\int_{\Omega}\frac{u_{\varepsilon}^{2^*_{\alpha}+\varepsilon}(y)u_{\varepsilon}^{2^*_{\alpha}+\varepsilon}(x)}{|x-y|^{\alpha}}dydx\\
        &=\alpha_{\varepsilon}^{2(2^*_{\alpha}+\varepsilon)}\int_{\Omega}\int_{\Omega}\frac{(P\bar{U}_{\xi_{\varepsilon},\lambda{_{\varepsilon}}}+w_{\varepsilon})^{2^*_{\alpha}+\varepsilon}(y)(P\bar{U}_{\xi_{\varepsilon},\lambda{_{\varepsilon}}}+w_{\varepsilon})^{2^*_{\alpha}+\varepsilon-1}(x)P\bar{U}_{\xi_{\varepsilon},\lambda{_{\varepsilon}}}(x)}{|x-y|^{\alpha}}dydx\\
        &\quad+\alpha_{\varepsilon}\int_{\Omega}\int_{\Omega}\frac{u_{\varepsilon}^{2^*_{\alpha}+\varepsilon}(y)u_{\varepsilon}^{2^*_{\alpha}+\varepsilon-1}(x)w_{\xi_{\varepsilon},\lambda{_{\varepsilon}}}(x)}{|x-y|^{\alpha}}dydx\\
        &=\alpha_{\varepsilon}^{2(2^*_{\alpha}+\varepsilon)}\int_{\Omega}\int_{\Omega}\frac{P\bar{U}_{\xi_{\varepsilon},\lambda{_{\varepsilon}}}^{2^*_{\alpha}+\varepsilon}(y)P\bar{U}_{\xi_{\varepsilon},\lambda{_{\varepsilon}}}^{2^*_{\alpha}+\varepsilon}(x)}{|x-y|^{\alpha}}dydx\\
        &\quad+O\left(\int_{\Omega}\int_{\Omega}\frac{P\bar{U}_{\xi_{\varepsilon},\lambda{_{\varepsilon}}}^{2^*_{\alpha}+\varepsilon}(y)(P\bar{U}_{\xi_{\varepsilon},\lambda{_{\varepsilon}}}^{2^*_{\alpha}+\varepsilon-1}|w_{\varepsilon}|+P\bar{U}_{\xi_{\varepsilon},\lambda{_{\varepsilon}}}|w_{\varepsilon}|^{2^*_{\alpha}+\varepsilon-1})(x)}{|x-y|^{\alpha}}dydx\right)\\
        &\quad+O\left(\int_{\Omega}\int_{\Omega}\frac{(P\bar{U}_{\xi_{\varepsilon},\lambda{_{\varepsilon}}}^{2^*_{\alpha}+\varepsilon-1}|w_{\varepsilon}|+|w_{\varepsilon}|^{2^*_{\alpha}+\varepsilon})(y)P\bar{U}_{\xi_{\varepsilon},\lambda{_{\varepsilon}}}^{2^*_{\alpha}+\varepsilon}(x)}{|x-y|^{\alpha}}dydx\right)\\
        &\quad+O\left(\int_{\Omega}\int_{\Omega}\frac{(P\bar{U}_{\xi_{\varepsilon},\lambda{_{\varepsilon}}}^{2^*_{\alpha}+\varepsilon-1}|w_{\varepsilon}|+|w_{\varepsilon}|^{2^*_{\alpha}+\varepsilon})(y)(P\bar{U}_{\xi_{\varepsilon},\lambda{_{\varepsilon}}}^{2^*_{\alpha}+\varepsilon-1}|w_{\varepsilon}|+P\bar{U}_{\xi_{\varepsilon},\lambda{_{\varepsilon}}}|w_{\varepsilon}|^{2^*_{\alpha}+\varepsilon-1})(x)}{|x-y|^{\alpha}}dydx\right)\\
        &\quad-\alpha_{\varepsilon}\int_{\Omega}\Delta u_{\varepsilon}w_{\varepsilon}dx.
    \end{aligned}
\end{equation}
Moreover, we observe that
\begin{equation}\label{eq lemma-2 proof-2}
    \begin{aligned}
        &\int_{\Omega}\int_{\Omega}\frac{P\bar{U}_{\xi_{\varepsilon},\lambda{_{\varepsilon}}}^{2^*_{\alpha}+\varepsilon}(y)P\bar{U}_{\xi_{\varepsilon},\lambda{_{\varepsilon}}}^{2^*_{\alpha}+\varepsilon}(x)}{|x-y|^{\alpha}}dydx\\
        &=\int_{\Omega}\int_{\Omega}\frac{(\bar{U}_{\xi_{\varepsilon},\lambda{_{\varepsilon}}}-\bar{\varphi}_{\xi_{\varepsilon},\lambda_{\varepsilon}})^{2^*_{\alpha}+\varepsilon}(y)(\bar{U}_{\xi_{\varepsilon},\lambda{_{\varepsilon}}}-\bar{\varphi}_{\xi_{\varepsilon},\lambda_{\varepsilon}})^{2^*_{\alpha}+\varepsilon}(x)}{|x-y|^{\alpha}}dydx\\
        &=\int_{\Omega}\int_{\Omega}\frac{\bar{U}_{\xi_{\varepsilon},\lambda{_{\varepsilon}}}^{2^*_{\alpha}+\varepsilon}(y)\bar{U}_{\xi_{\varepsilon},\lambda{_{\varepsilon}}}^{2^*_{\alpha}+\varepsilon}(x)}{|x-y|^{\alpha}}dydx+O\left(\int_{\Omega}\int_{\Omega}\frac{\bar{U}_{\xi_{\varepsilon},\lambda{_{\varepsilon}}}^{2^*_{\alpha}+\varepsilon}(y)\bar{U}_{\xi_{\varepsilon},\lambda{_{\varepsilon}}}^{2^*_{\alpha}+\varepsilon-1}(x)\bar{\varphi}_{\xi_{\varepsilon},\lambda_{\varepsilon}}(x)}{|x-y|^{\alpha}}dydx\right).\\
    \end{aligned}
\end{equation}
Next, we are going to estimate each term on the right hand side of \eqref{eq lemma-2 proof-1} and \eqref{eq lemma-2 proof-2}. By HLS inequality and Lemma \ref{estimate of U-lambda-a and psi-lambda-a 2}, we get 
\begin{equation}
    \begin{aligned}
        &\int_{\Omega}\int_{\Omega}\frac{\bar{U}_{\xi_{\varepsilon},\lambda{_{\varepsilon}}}^{2^*_{\alpha}+\varepsilon}(y)\bar{U}_{\xi_{\varepsilon},\lambda{_{\varepsilon}}}^{2^*_{\alpha}+\varepsilon}(x)}{|x-y|^{\alpha}}dydx\\
        &=\int_{B(\xi_{\varepsilon},d_{\varepsilon})}\int_{B(\xi_{\varepsilon},d_{\varepsilon})}\frac{\bar{U}_{\xi_{\varepsilon},\lambda{_{\varepsilon}}}^{2^*_{\alpha}+\varepsilon}(y)\bar{U}_{\xi_{\varepsilon},\lambda{_{\varepsilon}}}^{2^*_{\alpha}+\varepsilon}(x)}{|x-y|^{\alpha}}dydx+\int_{\Omega\setminus B(\xi_{\varepsilon},d_{\varepsilon})}\int_{B(\xi_{\varepsilon},d_{\varepsilon})}\frac{\bar{U}_{\xi_{\varepsilon},\lambda{_{\varepsilon}}}^{2^*_{\alpha}+\varepsilon}(y)\bar{U}_{\xi_{\varepsilon},\lambda{_{\varepsilon}}}^{2^*_{\alpha}+\varepsilon}(x)}{|x-y|^{\alpha}}dydx\\
        &\quad+\int_{\Omega}\int_{\Omega\setminus B(\xi_{\varepsilon},d_{\varepsilon})}\frac{\bar{U}_{\xi_{\varepsilon},\lambda{_{\varepsilon}}}^{2^*_{\alpha}+\varepsilon}(y)\bar{U}_{\xi_{\varepsilon},\lambda{_{\varepsilon}}}^{2^*_{\alpha}+\varepsilon}(x)}{|x-y|^{\alpha}}dydx\\
        &=\int_{B(\xi_{\varepsilon},d_{\varepsilon})}\int_{B(\xi_{\varepsilon},d_{\varepsilon})}\frac{\bar{U}_{\xi_{\varepsilon},\lambda{_{\varepsilon}}}^{2^*_{\alpha}+\varepsilon}(y)\bar{U}_{\xi_{\varepsilon},\lambda{_{\varepsilon}}}^{2^*_{\alpha}+\varepsilon}(x)}{|x-y|^{\alpha}}dydx+O\left(\frac{\lambda_{\varepsilon}^{(N-2)\varepsilon}}{(\lambda_{\varepsilon}d_{\varepsilon})^{\frac{2N-\alpha}{2}}}\right)
 \end{aligned}
\end{equation}
and
\begin{equation}
    \begin{aligned}
       &\int_{B(\xi_{\varepsilon},d_{\varepsilon})}\int_{B(\xi_{\varepsilon},d_{\varepsilon})}\frac{\bar{U}_{\xi_{\varepsilon},\lambda{_{\varepsilon}}}^{2^*_{\alpha}+\varepsilon}(y)\bar{U}_{\xi_{\varepsilon},\lambda{_{\varepsilon}}}^{2^*_{\alpha}+\varepsilon}(x)}{|x-y|^{\alpha}}dydx\\
       &=\lambda_{\varepsilon}^{(N-2)\varepsilon}\left\{\int_{\R^{N}}\int_{\R^{N}}\frac{\bar{U}_{0,1}^{2^*_{\alpha}+\varepsilon}(y)\bar{U}_{0,1}^{2^*_{\alpha}+\varepsilon}(x)}{|x-y|^{\alpha}}dydx-\int_{\R^{N}}\int_{\R^{N}\setminus B(0,\lambda_{\varepsilon}d_{\varepsilon})}\frac{\bar{U}_{0,1}^{2^*_{\alpha}+\varepsilon}(y)\bar{U}_{0,1}^{2^*_{\alpha}+\varepsilon}(x)}{|x-y|^{\alpha}}dydx\right.\\
       &\quad-\left.\int_{\R^{N}\setminus B(0,\lambda_{\varepsilon}d_{\varepsilon})}\int_{B(0,\lambda_{\varepsilon}d_{\varepsilon})}\frac{\bar{U}_{0,1}^{2^*_{\alpha}+\varepsilon}(y)\bar{U}_{0,1}^{2^*_{\alpha}+\varepsilon}(x)}{|x-y|^{\alpha}}dydx\right\}\\
       &=\lambda_{\varepsilon}^{(N-2)\varepsilon}\int_{\R^{N}}\int_{\R^{N}}\frac{\bar{U}_{0,1}^{2^*_{\alpha}+\varepsilon}(y)\bar{U}_{0,1}^{2^*_{\alpha}+\varepsilon}(x)}{|x-y|^{\alpha}}dydx+O\left(\frac{\lambda_{\varepsilon}^{(N-2)\varepsilon}}{(\lambda_{\varepsilon}d_{\varepsilon})^{\frac{2N-\alpha}{2}}}\right)\\
        &=\lambda_{\varepsilon}^{(N-2)\varepsilon}\left(\int_{\R^{N}}\int_{\R^{N}}\frac{\bar{U}_{0,1}^{2^*_{\alpha}}(y)\bar{U}_{0,1}^{2^*_{\alpha}}(x)}{|x-y|^{\alpha}}dydx+O(\varepsilon)\right)+O\left(\frac{\lambda_{\varepsilon}^{(N-2)\varepsilon}}{(\lambda_{\varepsilon}d_{\varepsilon})^{\frac{2N-\alpha}{2}}}\right)\\
        &=\lambda_{\varepsilon}^{(N-2)\varepsilon}(S_{HL}^{\frac{2^*_{\alpha}}{2^*_{\alpha}-1}}+o(1)). 
    \end{aligned}
\end{equation}      
On the other hand, by HLS inequality and Lemma \ref{estimate of U-lambda-a and psi-lambda-a 4}, we have
\begin{equation}
    \begin{aligned}
        &\int_{\Omega}\int_{\Omega}\frac{\bar{U}_{\xi_{\varepsilon},\lambda{_{\varepsilon}}}^{2^*_{\alpha}+\varepsilon}(y)\bar{U}_{\xi_{\varepsilon},\lambda{_{\varepsilon}}}^{2^*_{\alpha}+\varepsilon-1}(x)\bar{\varphi}_{\xi_{\varepsilon},\lambda_{\varepsilon}}(x)}{|x-y|^{\alpha}}dydx\\
        &\lesssim\left(\int_{\Omega}\bar{U}_{\xi_{\varepsilon},\lambda{_{\varepsilon}}}^{\frac{2^*(2^*_{\alpha}+\varepsilon)}{2^*_{\alpha}}}dx\right)^{\frac{2^*_{\alpha}}{2^*}}\left(\int_{\Omega}\bar{U}_{\xi_{\varepsilon},\lambda{_{\varepsilon}}}^{\frac{2^*(2^*_{\alpha}+\varepsilon-1)}{2^*_{\alpha}-1}}dx\right)^{\frac{2^*_{\alpha}-1}{2^*}}\left(\int_{\Omega}\bar{\varphi}_{\xi_{\varepsilon},\lambda_{\varepsilon}}^{2^*}dx\right)^{\frac{1}{2^*}}\\
        &=\lambda_{\varepsilon}^{(N-2)\varepsilon}O\left((\lambda_{\varepsilon}d_{\varepsilon})^{-\frac{N-2}{2}}\right),
    \end{aligned}
\end{equation}

\begin{equation}
    \begin{aligned}
        &\int_{\Omega}\int_{\Omega}\frac{P\bar{U}_{\xi_{\varepsilon},\lambda{_{\varepsilon}}}^{2^*_{\alpha}+\varepsilon}(y)(P\bar{U}_{\xi_{\varepsilon},\lambda{_{\varepsilon}}}^{2^*_{\alpha}+\varepsilon-1}|w_{\varepsilon}|+P\bar{U}_{\xi_{\varepsilon},\lambda{_{\varepsilon}}}|w_{\varepsilon}|^{2^*_{\alpha}+\varepsilon-1})(x)}{|x-y|^{\alpha}}dydx\\
        &=\lambda_{\varepsilon}^{(N-2)\varepsilon}O\left(\|w_{\varepsilon}\|_{L^{2^*}}+\|w_{\varepsilon}\|_{L^{2^*}}^{2^*_{\alpha}+\varepsilon-1}\right),
    \end{aligned}
\end{equation}

\begin{equation}
    \begin{aligned}
        &\int_{\Omega}\int_{\Omega}\frac{(P\bar{U}_{\xi_{\varepsilon},\lambda{_{\varepsilon}}}^{2^*_{\alpha}+\varepsilon-1}|w_{\varepsilon}|+|w_{\varepsilon}|^{2^*_{\alpha}+\varepsilon})(y)P\bar{U}_{\xi_{\varepsilon},\lambda{_{\varepsilon}}}^{2^*_{\alpha}+\varepsilon}(x)}{|x-y|^{\alpha}}dydx\\
        &=\lambda_{\varepsilon}^{(N-2)\varepsilon}O\left(\|w_{\varepsilon}\|_{L^{2^*}}+\|w_{\varepsilon}\|_{L^{2^*}}^{2^*_{\alpha}+\varepsilon}\right),
    \end{aligned}
\end{equation}
and
\begin{equation}
    \begin{aligned}
        &\int_{\Omega}\int_{\Omega}\frac{(P\bar{U}_{\xi_{\varepsilon},\lambda{_{\varepsilon}}}^{2^*_{\alpha}+\varepsilon-1}|w_{\varepsilon}|+|w_{\varepsilon}|^{2^*_{\alpha}+\varepsilon})(y)(P\bar{U}_{\xi_{\varepsilon},\lambda{_{\varepsilon}}}^{2^*_{\alpha}+\varepsilon-1}|w_{\varepsilon}|+P\bar{U}_{\xi_{\varepsilon},\lambda{_{\varepsilon}}}|w_{\varepsilon}|^{2^*_{\alpha}+\varepsilon-1})(x)}{|x-y|^{\alpha}}dydx\\
        &=\lambda_{\varepsilon}^{(N-2)\varepsilon}O\left(\|w_{\varepsilon}\|_{L^{2^*}}^{2}+\|w_{\varepsilon}\|_{L^{2^*}}^{2^*_{\alpha}+\varepsilon}\right).
    \end{aligned}
\end{equation}
Furthermore, by Lemma \ref{lemma-1}
\begin{equation}
    -\int_{\Omega}\Delta u_{\varepsilon}w_{\varepsilon}dx=\int_{\Omega}\nabla u_{\varepsilon}\cdot\nabla w_{\varepsilon}=O(\|w_{\varepsilon}\|_{H^{1}_{0}(\Omega)}).
\end{equation}
Now, combining all the estimates above and using \eqref{eq condition parameter}, we get
\begin{equation}
    \int_{\Omega}\int_{\Omega}\frac{u_{\varepsilon}^{2^*_{\alpha}+\varepsilon}(y)u_{\varepsilon}^{2^*_{\alpha}+\varepsilon}(x)}{|x-y|^{\alpha}}dydx=\alpha_{\varepsilon}^{2(2^*_{\alpha}+\varepsilon)}\lambda_{\varepsilon}^{(N-2)\varepsilon}(S_{HL}^{\frac{2^*_{\alpha}}{2^*_{\alpha}-1}}+o(1))+o(1).
\end{equation}
Notice that $\int_{\Omega}\int_{\Omega}\frac{u_{\varepsilon}^{2^*_{\alpha}+\varepsilon}(y)u_{\varepsilon}^{2^*_{\alpha}+\varepsilon}(x)}{|x-y|^{\alpha}}dydx\to S_{HL}^{\frac{2^*_{\alpha}}{2^*_{\alpha}-1}}$ and $\alpha_{\varepsilon}\to 1$ as $\varepsilon\to 0$, thus \eqref{eq lemma-2} holds.
\end{proof}

\begin{Cor}\label{cor-1}
    It holds that
    \begin{equation}
        \bar{U}_{\xi_{\varepsilon},\lambda_{\varepsilon}}^{\varepsilon}=\bar{C}_{N,\alpha}^{\varepsilon}\lambda_{\varepsilon}^{\frac{(N-2)\varepsilon}{2}}+O\left(\varepsilon\log(1+\lambda_{\varepsilon}^{2}|x-\xi_{\varepsilon}|^{2})\right)\text{~~in~~}\Omega.
    \end{equation}
\end{Cor}
\begin{proof}
    The proof is similar to\cite[Lemma 2.3]{Ben2003CCM} (see also \cite{Rey1999ADV}), so we omit the details.
\end{proof}
\begin{Lem}
    It holds that
    \begin{equation}\label{eq lemma-3}
        \int_{\Omega}\int_{\Omega}\frac{|w|_{\varepsilon}^{2^*_{\alpha}+\varepsilon}(y)|w|_{\varepsilon}^{2^*_{\alpha}+\varepsilon}(x)}{|x-y|^{\alpha}}dydx\to 0\text{~~as~~}\varepsilon\to0.
    \end{equation}
\end{Lem}
\begin{proof}
First, we observe that
    \begin{equation}
    \begin{aligned}
       \int_{\Omega}\int_{\Omega}&\frac{u_{\varepsilon}^{2^*_{\alpha}+\varepsilon}(y)u_{\varepsilon}^{2^*_{\alpha}+\varepsilon}(x)}{|x-y|^{\alpha}}dydx\\
       &=\alpha_{\varepsilon}^{2(2^*_{\alpha}+\varepsilon)}\int_{\Omega}\int_{\Omega}\frac{(P\bar{U}_{\xi_{\varepsilon},\lambda_{\varepsilon}}+w_{\varepsilon})^{2^*_{\alpha}+\varepsilon}(y)(P\bar{U}_{\xi_{\varepsilon},\lambda_{\varepsilon}}+w_{\varepsilon})^{2^*_{\alpha}+\varepsilon}(x)}{|x-y|^{\alpha}}dydx\\
       &=\alpha_{\varepsilon}^{2(2^*_{\alpha}+\varepsilon)}\int_{\Omega}\int_{\Omega}\frac{P\bar{U}_{\xi_{\varepsilon},\lambda_{\varepsilon}}^{2^*_{\alpha}+\varepsilon}(y)P\bar{U}_{\xi_{\varepsilon},\lambda_{\varepsilon}}^{2^*_{\alpha}+\varepsilon}(x)}{|x-y|^{\alpha}}dydx\\
       &\quad+\alpha_{\varepsilon}^{2(2^*_{\alpha}+\varepsilon)}\int_{\Omega}\int_{\Omega}\frac{|w_{\varepsilon}|^{2^*_{\alpha}+\varepsilon}(y)|w_{\varepsilon}|^{2^*_{\alpha}+\varepsilon}(x)}{|x-y|^{\alpha}}dydx\\
       &\quad+2\alpha_{\varepsilon}^{2(2^*_{\alpha}+\varepsilon)}\int_{\Omega}\int_{\Omega}\frac{P\bar{U}_{\xi_{\varepsilon},\lambda_{\varepsilon}}^{2^*_{\alpha}+\varepsilon}(y)|w_{\varepsilon}|^{2^*_{\alpha}+\varepsilon}(x)}{|x-y|^{\alpha}}dydx\\
       &\quad+O\left(\int_{\Omega}\int_{\Omega}\frac{P\bar{U}_{\xi_{\varepsilon},\lambda_{\varepsilon}}^{2^*_{\alpha}+\varepsilon}(y)(P\bar{U}_{\xi_{\varepsilon},\lambda_{\varepsilon}}^{2^*_{\alpha}+\varepsilon-1}|w_{\varepsilon}|+P\bar{U}_{\xi_{\varepsilon},\lambda_{\varepsilon}}|w_{\varepsilon}|^{2^*_{\alpha}+\varepsilon-1})(x)}{|x-y|^{\alpha}}dydx\right)\\
       &\quad+O\left(\int_{\Omega}\int_{\Omega}\frac{|w_{\varepsilon}|^{2^*_{\alpha}+\varepsilon}(y)(P\bar{U}_{\xi_{\varepsilon},\lambda_{\varepsilon}}^{2^*_{\alpha}+\varepsilon-1}|w_{\varepsilon}|+P\bar{U}_{\xi_{\varepsilon},\lambda_{\varepsilon}}|w_{\varepsilon}|^{2^*_{\alpha}+\varepsilon-1})(x)}{|x-y|^{\alpha}}dydx\right).
    \end{aligned}
\end{equation}
Moreover, by HLS inequality, \eqref{eq condition parameter} and the estimate in Lemma \ref{lemma-2}, we have

\begin{equation}
    \begin{aligned}
        \int_{\Omega}\int_{\Omega}\frac{u_{\varepsilon}^{2^*_{\alpha}+\varepsilon}(y)u_{\varepsilon}^{2^*_{\alpha}+\varepsilon}(x)}{|x-y|^{\alpha}}dydx&=(1+o(1))(S_{HL}^{\frac{2^*_{\alpha}}{2^*_{\alpha}-1}}+o(1))+o(1)\\
        &\quad+(1+o(1))\int_{\Omega}\int_{\Omega}\frac{|w_{\varepsilon}|^{2^*_{\alpha}+\varepsilon}(y)|w_{\varepsilon}|^{2^*_{\alpha}+\varepsilon}(x)}{|x-y|^{\alpha}}dydx.
    \end{aligned}
\end{equation}
Since $\int_{\Omega}\int_{\Omega}\frac{u_{\varepsilon}^{2^*_{\alpha}+\varepsilon}(y)u_{\varepsilon}^{2^*_{\alpha}+\varepsilon}(x)}{|x-y|^{\alpha}}dydx\to S_{HL}^{\frac{2^*_{\alpha}}{2^*_{\alpha}-1}}$ as $\varepsilon\to0$, thus \eqref{eq lemma-3} holds.
\end{proof}
Let $x_{\varepsilon}\in\Omega$ such that $u_{\varepsilon}(x_{\varepsilon})=\max_{x\in\Omega} u_{\varepsilon}(x):=M_{\varepsilon}^{\frac{N-2}{2}}$ and we define the following normalized functions
\begin{equation}\label{eq definition of v-varepsilon}
    v_{\varepsilon}(x):=M_{\varepsilon}^{-\frac{N-2}{2}}u_{\varepsilon}(M_{\varepsilon}^{-\frac{2^*_{\alpha}-1+\varepsilon}{2^*_{\alpha}-1}}x+x_{\varepsilon})\text{~~where~~}x\in\Omega_{\varepsilon}:=\{x\in\R^{N}:M_{\varepsilon}^{-\frac{2^*_{\alpha}-1+\varepsilon}{2^*_{\alpha}-1}}x+x_{\varepsilon}\in\Omega\}.
\end{equation}
It is easy to verify that $0<v_{\varepsilon}(x)\leq v_{\varepsilon}(0)=\max_{x\in\Omega_{\varepsilon}}v_{\varepsilon}(x)=1$ and
\begin{equation}
    \begin{cases}
        -\Delta v_{\varepsilon}(x)=\left(\int_{\Omega_{\varepsilon}}\frac{v_{\varepsilon}^{2^*_{\alpha}+\varepsilon}(y)}{|x-y|^{\alpha}}dy\right)v_{\varepsilon}^{2^*_{\alpha}+\varepsilon-1}(x),\quad v_{\varepsilon}(x)>0&\text{~~in~~}\Omega_{\varepsilon},\\
       \quad\ \ v_{\varepsilon}(x)=0&\text{~~on~~}\partial\Omega_{\varepsilon}.\\
    \end{cases}
\end{equation}
Moreover, we set $v_{\varepsilon}(x):=0$ for any $x\in\R^{N}\setminus\Omega_{\varepsilon}$.

\begin{Lem}\label{lemma-4}
     As $\varepsilon\to 0$, it holds that 
     \begin{equation}
        \|u_{\varepsilon}\|_{L^{\infty}(\Omega)}^{\varepsilon}=O(1)\text{~~and~~}\|w_{\varepsilon}\|_{L^{\infty}(\Omega)}^{\varepsilon}=O(1).
     \end{equation}
\end{Lem}
\begin{proof}
Since $\|u_{\varepsilon}\|_{H^{1}_{0}(\Omega)}\to S_{HL}^{\frac{2^*_{\alpha}}{2^*_{\alpha}-1}}$ as $\varepsilon\to0$ and $S_{HL}(\Omega)$ is never achieved except $\Omega=\R^{N}$, thus $M_{\varepsilon}$ cannot be uniformly bounded. Next, using the standard elliptic regularity theory, after passing to a subsequence, we have $v_{\varepsilon}\to v_{0}$ in $C^{1}_{loc}(\Omega_{\infty})$ and $v_{0}$ satisfies
\begin{equation}
    \begin{cases}
        -\Delta v_{0}(x)=\left(\int_{\Omega_{\infty}}\frac{v_{0}^{2^*_{\alpha}}(y)}{|x-y|^{\alpha}}dy\right)v_{0}^{2^*_{\alpha}-1}(x),\quad v_{0}(x)>0&\text{~~in~~}\Omega_{\infty},\\
        \quad\ \ v_{0}(0)=\max_{x\in\Omega_{\infty}}v_{0}(x)=1,\\
    \end{cases}
\end{equation}
where $\Omega_{\infty}=\R^{N}$ or $\Omega_{\infty}=\R_{+}^{N}$ (up to a translation and rotation). If $\Omega_{\infty}=\R_{+}^{N}$, then we can deduce that $v_{0}=0$ by Pohoz\v{a}ev identity (see \cite[Lemma 2.7]{Goel2020}), this makes a contradiction with $v_{0}(0)=1$. Thus $M_{\varepsilon}^{\frac{2^*_{\alpha}-1+\varepsilon}{2^*_{\alpha}-1}}dist(x_{\varepsilon},\partial\Omega)\to+\infty$ and $\Omega_{\infty}=\R^{N}$.  Thanks to the results in \cite{Gao2018,Du2019,Guo2019}, we know that $v_{0}(x)=\bar{U}_{0,\lambda_{0}}(x)$, where $\lambda_{0}\in\R^{+}$ satisfying $\bar{C}_{N,\alpha}\lambda_{0}^{\frac{N-2}{2}}=1$. Thus, there exists a constant $c>0$ such that
\begin{equation}
\begin{aligned}
    &M_{\varepsilon}^{\frac{(N-2)\varepsilon}{2^*_{\alpha}-1}}\int_{B(x_{\varepsilon},M_{\varepsilon}^{-\frac{2^*_{\alpha}-1+\varepsilon}{2^*_{\alpha}-1}})}\int_{B(x_{\varepsilon},M_{\varepsilon}^{-\frac{2^*_{\alpha}-1+\varepsilon}{2^*_{\alpha}-1}})}\frac{u_{\varepsilon}^{2^*_{\alpha}+\varepsilon}(y)u_{\varepsilon}^{2^*_{\alpha}+\varepsilon}(x)}{|x-y|^{\alpha}}dydx\\
    &\quad=\int_{B(0,1)}\int_{B(0,1)}\frac{v_{\varepsilon}^{2^*_{\alpha}+\varepsilon}(y)v_{\varepsilon}^{2^*_{\alpha}+\varepsilon}(x)}{|x-y|^{\alpha}}dydx\to c>0\text{~~as~~}\varepsilon\to0
\end{aligned} 
\end{equation}
and
\begin{equation}
\begin{aligned}
    &M_{\varepsilon}^{\frac{(2N-\alpha)\varepsilon}{2^*_{\alpha}-1}}\int_{B(x_{\varepsilon},M_{\varepsilon}^{-\frac{2^*_{\alpha}-1+\varepsilon}{2^*_{\alpha}-1}})}\int_{B(x_{\varepsilon},M_{\varepsilon}^{-\frac{2^*_{\alpha}-1+\varepsilon}{2^*_{\alpha}-1}})}\frac{u_{\varepsilon}^{2^*_{\alpha}}(y)u_{\varepsilon}^{2^*_{\alpha}}(x)}{|x-y|^{\alpha}}dydx\\
    &\quad=\int_{B(0,1)}\int_{B(0,1)}\frac{v_{\varepsilon}^{2^*_{\alpha}}(y)v_{\varepsilon}^{2^*_{\alpha}}(x)}{|x-y|^{\alpha}}dydx\to c>0\text{~~as~~}\varepsilon\to0.
\end{aligned} 
\end{equation}
On the other hand, by decomposition \eqref{eq orthogonal decomposition} and HLS inequality
\begin{equation}
    \begin{aligned}
        &\int_{B(x_{\varepsilon},M_{\varepsilon}^{-\frac{2^*_{\alpha}-1+\varepsilon}{2^*_{\alpha}-1}})}\int_{B(x_{\varepsilon},M_{\varepsilon}^{-\frac{2^*_{\alpha}-1+\varepsilon}{2^*_{\alpha}-1}})}\frac{u_{\varepsilon}^{2^*_{\alpha}+\varepsilon}(y)u_{\varepsilon}^{2^*_{\alpha}+\varepsilon}(x)}{|x-y|^{\alpha}}dydx\\
        &=\alpha_{\varepsilon}^{2(2^*_{\alpha}+\varepsilon)}\int_{B(x_{\varepsilon},M_{\varepsilon}^{-\frac{2^*_{\alpha}-1+\varepsilon}{2^*_{\alpha}-1}})}\int_{B(x_{\varepsilon},M_{\varepsilon}^{-\frac{2^*_{\alpha}-1+\varepsilon}{2^*_{\alpha}-1}})}\frac{(P\bar{U}_{\xi_{\varepsilon},\lambda_{\varepsilon}}+w_{\varepsilon})^{2^*_{\alpha}+\varepsilon}(y)(P\bar{U}_{\xi_{\varepsilon},\lambda_{\varepsilon}}+w_{\varepsilon})^{2^*_{\alpha}+\varepsilon}(x)}{|x-y|^{\alpha}}dydx\\
        &=\alpha_{\varepsilon}^{2(2^*_{\alpha}+\varepsilon)}\int_{B(x_{\varepsilon},M_{\varepsilon}^{-\frac{2^*_{\alpha}-1+\varepsilon}{2^*_{\alpha}-1}})}\int_{B(x_{\varepsilon},M_{\varepsilon}^{-\frac{2^*_{\alpha}-1+\varepsilon}{2^*_{\alpha}-1}})}\frac{\bar{U}_{\xi_{\varepsilon},\lambda_{\varepsilon}}^{2^*_{\alpha}+\varepsilon}(y)\bar{U}_{\xi_{\varepsilon},\lambda_{\varepsilon}}^{2^*_{\alpha}+\varepsilon}(x)}{|x-y|^{\alpha}}dydx+o(1)\\
        &=\alpha_{\varepsilon}^{2(2^*_{\alpha}+\varepsilon)}\lambda_{\varepsilon}^{(N-2)\varepsilon}\int_{B(0,\lambda_{\varepsilon}M_{\varepsilon}^{-\frac{2^*_{\alpha}-1+\varepsilon}{2^*_{\alpha}-1}})}\int_{B(0,\lambda_{\varepsilon}M_{\varepsilon}^{-\frac{2^*_{\alpha}-1+\varepsilon}{2^*_{\alpha}-1}})}\frac{\bar{U}_{\lambda_{\varepsilon}(\xi_{\varepsilon}-x_{\varepsilon}),1}^{2^*_{\alpha}+\varepsilon}(y)\bar{U}^{2^*_{\alpha}+\varepsilon}_{\lambda_{\varepsilon}(\xi_{\varepsilon}-x_{\varepsilon}),1}(x)}{|x-y|^{\alpha}}dydx+o(1)
    \end{aligned}
\end{equation}
and
\begin{equation}
    \begin{aligned}
        &\int_{B(x_{\varepsilon},M_{\varepsilon}^{-\frac{2^*_{\alpha}-1+\varepsilon}{2^*_{\alpha}-1}})}\int_{B(x_{\varepsilon},M_{\varepsilon}^{-\frac{2^*_{\alpha}-1+\varepsilon}{2^*_{\alpha}-1}})}\frac{u_{\varepsilon}^{2^*_{\alpha}}(y)u_{\varepsilon}^{2^*_{\alpha}}(x)}{|x-y|^{\alpha}}dydx\\
        &=\alpha_{\varepsilon}^{22^*_{\alpha}}\int_{B(x_{\varepsilon},M_{\varepsilon}^{-\frac{2^*_{\alpha}-1+\varepsilon}{2^*_{\alpha}-1}})}\int_{B(x_{\varepsilon},M_{\varepsilon}^{-\frac{2^*_{\alpha}-1+\varepsilon}{2^*_{\alpha}-1}})}\frac{(P\bar{U}_{\xi_{\varepsilon},\lambda_{\varepsilon}}+w_{\varepsilon})^{2^*_{\alpha}}(y)(P\bar{U}_{\xi_{\varepsilon},\lambda_{\varepsilon}}+w_{\varepsilon})^{2^*_{\alpha}}(x)}{|x-y|^{\alpha}}dydx\\
        &=\alpha_{\varepsilon}^{22^*_{\alpha}}\int_{B(x_{\varepsilon},M_{\varepsilon}^{-\frac{2^*_{\alpha}-1+\varepsilon}{2^*_{\alpha}-1}})}\int_{B(x_{\varepsilon},M_{\varepsilon}^{-\frac{2^*_{\alpha}-1+\varepsilon}{2^*_{\alpha}-1}})}\frac{\bar{U}_{\xi_{\varepsilon},\lambda_{\varepsilon}}^{2^*_{\alpha}}(y)\bar{U}_{\xi_{\varepsilon},\lambda_{\varepsilon}}^{2^*_{\alpha}}(x)}{|x-y|^{\alpha}}dydx+o(1)\\
        &=\alpha_{\varepsilon}^{22^*_{\alpha}}\int_{B(0,\lambda_{\varepsilon}M_{\varepsilon}^{-\frac{2^*_{\alpha}-1+\varepsilon}{2^*_{\alpha}-1}})}\int_{B(0,\lambda_{\varepsilon}M_{\varepsilon}^{-\frac{2^*_{\alpha}-1+\varepsilon}{2^*_{\alpha}-1}})}\frac{\bar{U}_{\lambda_{\varepsilon}(\xi_{\varepsilon}-x_{\varepsilon}),1}^{2^*_{\alpha}}(y)\bar{U}^{2^*_{\alpha}}_{\lambda_{\varepsilon}(\xi_{\varepsilon}-x_{\varepsilon}),1}(x)}{|x-y|^{\alpha}}dydx+o(1).
    \end{aligned}
\end{equation}
Combining the estimates above together with Lemma \ref{lemma-2}, we obtain that as $\varepsilon\to 0$
\begin{equation}\label{lemma-4-proof-6}
    \begin{aligned}
        M_{\varepsilon}^{\frac{(N-2)\varepsilon}{2^*_{\alpha}-1}}\int_{B(0,\lambda_{\varepsilon}M_{\varepsilon}^{-\frac{2^*_{\alpha}-1+\varepsilon}{2^*_{\alpha}-1}})}\int_{B(0,\lambda_{\varepsilon}M_{\varepsilon}^{-\frac{2^*_{\alpha}-1+\varepsilon}{2^*_{\alpha}-1}})}\frac{\bar{U}_{\lambda_{\varepsilon}(\xi_{\varepsilon}-x_{\varepsilon}),1}^{2^*_{\alpha}+\varepsilon}(y)\bar{U}^{2^*_{\alpha}+\varepsilon}_{\lambda_{\varepsilon}(\xi_{\varepsilon}-x_{\varepsilon}),1}(x)}{|x-y|^{\alpha}}dydx\to c>0
    \end{aligned}
\end{equation}
and
\begin{equation}\label{lemma-4-proof-7}
    \begin{aligned}
        M_{\varepsilon}^{\frac{(2N-\alpha)\varepsilon}{2^*_{\alpha}-1}}\int_{B(0,\lambda_{\varepsilon}M_{\varepsilon}^{-\frac{2^*_{\alpha}-1+\varepsilon}{2^*_{\alpha}-1}})}\int_{B(0,\lambda_{\varepsilon}M_{\varepsilon}^{-\frac{2^*_{\alpha}-1+\varepsilon}{2^*_{\alpha}-1}})}\frac{\bar{U}_{\lambda_{\varepsilon}(\xi_{\varepsilon}-x_{\varepsilon}),1}^{2^*_{\alpha}}(y)\bar{U}^{2^*_{\alpha}}_{\lambda_{\varepsilon}(\xi_{\varepsilon}-x_{\varepsilon}),1}(x)}{|x-y|^{\alpha}}dydx\to c>0.
    \end{aligned}
\end{equation}
Next, we treat the following two cases separately.\\
\smallskip
\textbf{Case 1.} $\lambda_{\varepsilon}M_{\varepsilon}^{-\frac{2^*_{\alpha}-1+\varepsilon}{2^*_{\alpha}-1}}\not\to 0$ as $\varepsilon\to0$. Then there exists a constant $c_{1}$ independent on $\varepsilon$ such that $\lambda_{\varepsilon}M_{\varepsilon}^{-\frac{2^*_{\alpha}-1+\varepsilon}{2^*_{\alpha}-1}}\geq c_{1}>0$. Thus, by Lemma \ref{lemma-2}, we have $M_{\varepsilon}^{\varepsilon}=O(1)$.\\
\smallskip
\textbf{Case 2.} $\lambda_{\varepsilon}M_{\varepsilon}^{-\frac{2^*_{\alpha}-1+\varepsilon}{2^*_{\alpha}-1}}\to 0$ as $\varepsilon\to0$. If $\lambda_{\varepsilon}|\xi_{\varepsilon}-x_{\varepsilon}|\not\to +\infty$ as $\varepsilon\to0$. Using \eqref{lemma-4-proof-7}, we get
\begin{equation}
     M_{\varepsilon}^{\frac{(2N-\alpha)\varepsilon}{2^*_{\alpha}-1}}(\lambda_{\varepsilon}M_{\varepsilon}^{-\frac{2^*_{\alpha}-1+\varepsilon}{2^*_{\alpha}-1}})^{2N-\alpha}\to c_{2}>0\text{~~as~}\varepsilon\to0.
\end{equation}
Thus
\begin{equation}
    \lambda_{\varepsilon}^{2N-\alpha}M_{\varepsilon}^{-(2N-\alpha)}\to c_{2} \text{~~as~}\varepsilon\to0.
\end{equation}
Using Lemma \ref{lemma-2}, we obtain $M_{\varepsilon}^{\varepsilon}=O(1)$. If $\lambda_{\varepsilon}|\xi_{\varepsilon}-x_{\varepsilon}|\to +\infty$ as $\varepsilon\to0$. Using \eqref{lemma-4-proof-6} and \eqref{lemma-4-proof-7}, we get as $\varepsilon\to0$
\begin{equation}
    \frac{\lambda_{\varepsilon}^{2N-\alpha}}{(\lambda_{\varepsilon}|\xi_{\varepsilon}-x_{\varepsilon}|)^{2(N-2)(2^*_{\alpha}+\varepsilon)}M_{\varepsilon}^{(2N-\alpha)+(N-2)\varepsilon}}\to c_{2}>0,\quad  \frac{\lambda_{\varepsilon}^{2N-\alpha}}{(\lambda_{\varepsilon}|\xi_{\varepsilon}-x_{\varepsilon}|)^{2(2N-\alpha)}M_{\varepsilon}^{(2N-\alpha)}}\to c_{2}>0.
\end{equation}
Thus
\begin{equation}
    M_{\varepsilon}^{(N-2)\varepsilon}(\lambda_{\varepsilon}|\xi_{\varepsilon}-x_{\varepsilon}|)^{2(N-2)\varepsilon}\to 1\text{~~as~}\varepsilon\to0.
\end{equation}
Using Lemma \ref{lemma-2}, we obtain $M_{\varepsilon}^{\varepsilon}=1+o(1)$. Finally $\|w_{\varepsilon}\|_{L^{\infty}(\Omega)}^{\varepsilon}=O(1)$ follows from \eqref{eq orthogonal decomposition} and Lemma \ref{lemma-2}. Thus we complete the proof. 
\end{proof}

\begin{Lem}\label{lemma-5}
Assume that $\alpha\in(0,\min\{4,N\})$, then it holds that 
    \begin{equation}
	\|w_\varepsilon\|_{H_{0}^{1}(\Omega)}\lesssim \varepsilon +\begin{cases}
        \frac{1}{(\lambda_{\varepsilon}d_{\varepsilon})^{N-2}}&\text{~~if~~}N<6-\alpha,\\
        \frac{(\log(\lambda_{\varepsilon}d_{\varepsilon}))^{\frac{4-\alpha}{6-\alpha}}}{(\lambda_{\varepsilon}d_{\varepsilon})^{4-\alpha}}&\text{~~if~~}N=6-\alpha,\\
        \frac{1}{(\lambda_{\varepsilon}d_{\varepsilon})^{\frac{N+2-\alpha}{2}}}&\text{~~if~~}N>6-\alpha.\\
    \end{cases}
\end{equation}
\begin{proof}
  Multiplying both sides of \eqref{slightly supercritical choquard equation} by $w_{\varepsilon}$ and integrating on $\Omega$, we get
  \begin{equation}
      \int_{\Omega}\nabla u_{\varepsilon}\cdot\nabla w_{\varepsilon}dx=\int_{\Omega}\int_{\Omega}\frac{u_{\varepsilon}^{2^*_{\alpha}+\varepsilon}(y)u_{\varepsilon}^{2^*_{\alpha}+\varepsilon-1}(x)w_{\varepsilon}(x)}{|x-y|^{\alpha}}dydx.
  \end{equation}
  Then using decomposition \eqref{eq orthogonal decomposition} and Lemma \ref{lem inequality 2}, we have
  \begin{equation}
      \begin{aligned}
          \int_{\Omega}&|\nabla w_{\varepsilon}|^{2}dx\\
          &=\alpha_{\varepsilon}^{2(2^*_{\alpha}+\varepsilon-1)}\int_{\Omega}\int_{\Omega}\frac{(P\bar{U}_{\xi_{\varepsilon},\lambda_{\varepsilon}}+w_{\varepsilon})^{2^*_{\alpha}+\varepsilon}(y)(P\bar{U}_{\xi_{\varepsilon},\lambda_{\varepsilon}}+w_{\varepsilon})^{2^*_{\alpha}+\varepsilon-1}(x)w_{\varepsilon}(x)}{|x-y|^{\alpha}}dydx\\
          &=\alpha_{\varepsilon}^{2(2^*_{\alpha}+\varepsilon-1)}\int_{\Omega}\int_{\Omega}\frac{P\bar{U}_{\xi_{\varepsilon},\lambda_{\varepsilon}}^{2^*_{\alpha}+\varepsilon}(y)P\bar{U}_{\xi_{\varepsilon},\lambda_{\varepsilon}}^{2^*_{\alpha}+\varepsilon-1}(x)w_{\varepsilon}(x)}{|x-y|^{\alpha}}dydx\\
          &\quad+\alpha_{\varepsilon}^{2(2^*_{\alpha}+\varepsilon-1)}(2^*_{\alpha}+\varepsilon-1)\int_{\Omega}\int_{\Omega}\frac{P\bar{U}_{\xi_{\varepsilon},\lambda_{\varepsilon}}^{2^*_{\alpha}+\varepsilon}(y)P\bar{U}_{\xi_{\varepsilon},\lambda_{\varepsilon}}^{2^*_{\alpha}+\varepsilon-2}(x)w_{\varepsilon}^{2}(x)}{|x-y|^{\alpha}}dydx\\
          &\quad+\alpha_{\varepsilon}^{2(2^*_{\alpha}+\varepsilon-1)}(2^*_{\alpha}+\varepsilon)\int_{\Omega}\int_{\Omega}\frac{P\bar{U}_{\xi_{\varepsilon},\lambda_{\varepsilon}}^{2^*_{\alpha}+\varepsilon-1}(y)w_{\varepsilon}(y)P\bar{U}_{\xi_{\varepsilon},\lambda_{\varepsilon}}^{2^*_{\alpha}+\varepsilon-1}(x)w_{\varepsilon}(x)}{|x-y|^{\alpha}}dydx\\         
          &\quad+\alpha_{\varepsilon}^{2(2^*_{\alpha}+\varepsilon-1)}(2^*_{\alpha}+\varepsilon)(2^*_{\alpha}+\varepsilon-1)\int_{\Omega}\int_{\Omega}\frac{P\bar{U}_{\xi_{\varepsilon},\lambda_{\varepsilon}}^{2^*_{\alpha}+\varepsilon-1}(y)w_{\varepsilon}(y)P\bar{U}_{\xi_{\varepsilon},\lambda_{\varepsilon}}^{2^*_{\alpha}+\varepsilon-2}(x)w_{\varepsilon}^{2}(x)}{|x-y|^{\alpha}}dydx\\         &\quad+O\left(\int_{\Omega}\int_{\Omega}\frac{\bar{U}_{\xi_{\varepsilon},\lambda_{\varepsilon}}^{2^*_{\alpha}+\varepsilon}(y)(\bar{U}_{\xi_{\varepsilon},\lambda_{\varepsilon}}^{2^*_{\alpha}+\varepsilon-3}|w_{\varepsilon}|^{3}\chi_{\bar{U}_{\xi_{\varepsilon},\lambda_{\varepsilon}}<|w_{\varepsilon}|}+|w_{\varepsilon}|^{2^*_{\alpha}+\varepsilon})(x)}{|x-y|^{\alpha}}dydx\right)\\
          &\quad+O\left(\int_{\Omega}\int_{\Omega}\frac{\bar{U}_{\xi_{\varepsilon},\lambda_{\varepsilon}}^{2^*_{\alpha}+\varepsilon-1}(y)|w_{\varepsilon}|(y)(\bar{U}_{\xi_{\varepsilon},\lambda_{\varepsilon}}^{2^*_{\alpha}+\varepsilon-3}|w_{\varepsilon}|^{3}\chi_{\bar{U}_{\xi_{\varepsilon},\lambda_{\varepsilon}}<|w_{\varepsilon}|}+|w_{\varepsilon}|^{2^*_{\alpha}+\varepsilon})(x)}{|x-y|^{\alpha}}dydx\right)\\
          &\quad+O\left(\int_{\Omega}\int_{\Omega}\frac{(\bar{U}_{\xi_{\varepsilon},\lambda_{\varepsilon}}^{2^*_{\alpha}+\varepsilon-2}|w_{\varepsilon}|^{2}+|w_{\varepsilon}|^{2^*_{\alpha}+\varepsilon})(y)u_{\varepsilon}^{2^*_{\alpha}+\varepsilon-1}(x)|w_{\varepsilon}(x)|}{|x-y|^{\alpha}}dydx\right).
      \end{aligned}
  \end{equation}
Moreover, using Lemma \ref{lemma-2}, Lemma \ref{lemma-4} and HLS inequality, we get
  \begin{equation}
      Q_{\varepsilon}(w_{\varepsilon},w_{\varepsilon})=f_{\varepsilon}(w_{\varepsilon})+O\left(\|w_{\varepsilon}\|_{H^{1}_{0}}^{\min\{3,2^*_{\alpha}\}}\right),
  \end{equation}
  where
  \begin{equation}
      \begin{aligned}
          Q_{\varepsilon}(w_{\varepsilon},w_{\varepsilon}):&= \int_{\Omega}|\nabla w_{\varepsilon}|^{2}dx\\
          &\quad-\alpha_{\varepsilon}^{2(2^*_{\alpha}+\varepsilon-1)}(2^*_{\alpha}+\varepsilon)\int_{\Omega}\int_{\Omega}\frac{P\bar{U}_{\xi_{\varepsilon},\lambda_{\varepsilon}}^{2^*_{\alpha}+\varepsilon-1}(y)w_{\varepsilon}(y)P\bar{U}_{\xi_{\varepsilon},\lambda_{\varepsilon}}^{2^*_{\alpha}+\varepsilon-1}(x)w_{\varepsilon}(x)}{|x-y|^{\alpha}}dydx\\
         &\quad-\alpha_{\varepsilon}^{2(2^*_{\alpha}+\varepsilon-1)}(2^*_{\alpha}+\varepsilon-1)\int_{\Omega}\int_{\Omega}\frac{P\bar{U}_{\xi_{\varepsilon},\lambda_{\varepsilon}}^{2^*_{\alpha}+\varepsilon}(y)P\bar{U}_{\xi_{\varepsilon},\lambda_{\varepsilon}}^{2^*_{\alpha}+\varepsilon-2}(x)w_{\varepsilon}^{2}(x)}{|x-y|^{\alpha}}dydx
      \end{aligned}
  \end{equation}
  and
  \begin{equation}
      f_{\varepsilon}(w_{\varepsilon}):=\alpha_{\varepsilon}^{2(2^*_{\alpha}+\varepsilon-1)}\int_{\Omega}\int_{\Omega}\frac{P\bar{U}_{\xi_{\varepsilon},\lambda_{\varepsilon}}^{2^*_{\alpha}+\varepsilon}(y)P\bar{U}_{\xi_{\varepsilon},\lambda_{\varepsilon}}^{2^*_{\alpha}+\varepsilon-1}(x)w_{\varepsilon}(x)}{|x-y|^{\alpha}}dydx.
  \end{equation}
Notice that, by HLS inequality, Lemma \ref{estimate of U-lambda-a and psi-lambda-a 4}, Lemma \ref{lem inequality 2} and Corollary \ref{cor-1}
\begin{equation}
    \begin{aligned}
        &\int_{\Omega}\int_{\Omega}\frac{P\bar{U}_{\xi_{\varepsilon},\lambda_{\varepsilon}}^{2^*_{\alpha}+\varepsilon-1}(y)w_{\varepsilon}(y)P\bar{U}_{\xi_{\varepsilon},\lambda_{\varepsilon}}^{2^*_{\alpha}+\varepsilon-1}(x)w_{\varepsilon}(x)}{|x-y|^{\alpha}}dydx\\
        &=\int_{\Omega}\int_{\Omega}\frac{(\bar{U}_{\xi_{\varepsilon},\lambda_{\varepsilon}}-\bar{\varphi}_{\xi_{\varepsilon},\lambda_{\varepsilon}})^{2^*_{\alpha}+\varepsilon-1}(y)w_{\varepsilon}(y)(\bar{U}_{\xi_{\varepsilon},\lambda_{\varepsilon}}-\bar{\varphi}_{\xi_{\varepsilon},\lambda_{\varepsilon}})^{2^*_{\alpha}+\varepsilon-1}(x)w_{\varepsilon}(x)}{|x-y|^{\alpha}}dydx\\
        &=\int_{\Omega}\int_{\Omega}\frac{\bar{U}_{\xi_{\varepsilon},\lambda_{\varepsilon}}^{2^*_{\alpha}+\varepsilon-1}(y)w_{\varepsilon}(y)\bar{U}_{\xi_{\varepsilon},\lambda_{\varepsilon}}^{2^*_{\alpha}+\varepsilon-1}(x)w_{\varepsilon}(x)}{|x-y|^{\alpha}}dydx\\
        &\quad+O\left(\int_{\Omega}\int_{\Omega}\frac{\bar{U}_{\xi_{\varepsilon},\lambda_{\varepsilon}}^{2^*_{\alpha}+\varepsilon-1}(y)|w_{\varepsilon}|(y)\bar{U}_{\xi_{\varepsilon},\lambda_{\varepsilon}}^{2^*_{\alpha}+\varepsilon-2}\bar{\varphi}_{\xi_{\varepsilon},\lambda_{\varepsilon}}(x)|w_{\varepsilon}|(x)}{|x-y|^{\alpha}}dydx\right)\\        &=\int_{\Omega}\int_{\Omega}\frac{\bar{U}_{\xi_{\varepsilon},\lambda_{\varepsilon}}^{2^*_{\alpha}+\varepsilon-1}(y)w_{\varepsilon}(y)\bar{U}_{\xi_{\varepsilon},\lambda_{\varepsilon}}^{2^*_{\alpha}+\varepsilon-1}(x)w_{\varepsilon}(x)}{|x-y|^{\alpha}}dydx+O\left(\frac{\|w_{\varepsilon}\|_{H^{1}_{0}}^{2}}{(\lambda_{\varepsilon}d_{\varepsilon})^{\frac{N-2}{2}}}\right)\\
        &=\int_{\Omega}\int_{\Omega}\frac{\bar{U}_{\xi_{\varepsilon},\lambda_{\varepsilon}}^{2^*_{\alpha}-1}(y)w_{\varepsilon}(y)\bar{U}_{\xi_{\varepsilon},\lambda_{\varepsilon}}^{2^*_{\alpha}-1}(x)w_{\varepsilon}(x)}{|x-y|^{\alpha}}dydx+o(\|w_{\varepsilon}\|_{H^{1}_{0}}^{2})\\
    \end{aligned}
\end{equation}
and similarly
\begin{equation}
    \begin{aligned}
        &\int_{\Omega}\int_{\Omega}\frac{P\bar{U}_{\xi_{\varepsilon},\lambda_{\varepsilon}}^{2^*_{\alpha}+\varepsilon}(y)P\bar{U}_{\xi_{\varepsilon},\lambda_{\varepsilon}}^{2^*_{\alpha}+\varepsilon-2}(x)w_{\varepsilon}^{2}(x)}{|x-y|^{\alpha}}dydx\\
        &=\int_{\Omega}\int_{\Omega}\frac{(\bar{U}_{\xi_{\varepsilon},\lambda_{\varepsilon}}-\bar{\varphi}_{\xi_{\varepsilon},\lambda_{\varepsilon}})^{2^*_{\alpha}+\varepsilon}(y)(\bar{U}_{\xi_{\varepsilon},\lambda_{\varepsilon}}-\bar{\varphi}_{\xi_{\varepsilon},\lambda_{\varepsilon}})^{2^*_{\alpha}+\varepsilon-2}(x)w_{\varepsilon}^{2}(x)}{|x-y|^{\alpha}}dydx\\
        &=\int_{\Omega}\int_{\Omega}\frac{\bar{U}_{\xi_{\varepsilon},\lambda_{\varepsilon}}^{2^*_{\alpha}+\varepsilon}(y)\bar{U}_{\xi_{\varepsilon},\lambda_{\varepsilon}}^{2^*_{\alpha}+\varepsilon-2}(x)w_{\varepsilon}^{2}(x)}{|x-y|^{\alpha}}dydx\\
        &\quad+O\left(\int_{\Omega}\int_{\Omega}\frac{\bar{U}_{\xi_{\varepsilon},\lambda_{\varepsilon}}^{2^*_{\alpha}+\varepsilon}(y)(\bar{U}_{\xi_{\varepsilon},\lambda_{\varepsilon}}^{2^*_{\alpha}+\varepsilon-3}\bar{\varphi}_{\xi_{\varepsilon},\lambda_{\varepsilon}}+\bar{U}_{\xi_{\varepsilon},\lambda_{\varepsilon}}\bar{\varphi}_{\xi_{\varepsilon},\lambda_{\varepsilon}}^{2^*_{\alpha}+\varepsilon-3})(x)w_{\varepsilon}^{2}(x)}{|x-y|^{\alpha}}dydx\right)\\
        &\quad+O\left(\int_{\Omega}\int_{\Omega}\frac{\bar{U}_{\xi_{\varepsilon},\lambda_{\varepsilon}}^{2^*_{\alpha}+\varepsilon-1}(y)\bar{\varphi}_{\xi_{\varepsilon},\lambda_{\varepsilon}}(y)\bar{U}_{\xi_{\varepsilon},\lambda_{\varepsilon}}^{2^*_{\alpha}+\varepsilon-2}(x)w_{\varepsilon}^{2}(x)}{|x-y|^{\alpha}}dydx\right)\\
        &=\int_{\Omega}\int_{\Omega}\frac{\bar{U}_{\xi_{\varepsilon},\lambda_{\varepsilon}}^{2^*_{\alpha}}(y)\bar{U}_{\xi_{\varepsilon},\lambda_{\varepsilon}}^{2^*_{\alpha}-2}(x)w_{\varepsilon}^{2}(x)}{|x-y|^{\alpha}}dydx+o(\|w_{\varepsilon}\|_{H^{1}_{0}}^{2}).
    \end{aligned}
\end{equation}
Thus, we have
\begin{equation}
\begin{aligned}
    Q_{\varepsilon}(w_{\varepsilon},w_{\varepsilon})&= \int_{\Omega}|\nabla w_{\varepsilon}|^{2}dx-2^*_{\alpha}\int_{\Omega}\int_{\Omega}\frac{\bar{U}_{\xi_{\varepsilon},\lambda_{\varepsilon}}^{2^*_{\alpha}-1}(y)w_{\varepsilon}(y)\bar{U}_{\xi_{\varepsilon},\lambda_{\varepsilon}}^{2^*_{\alpha}-1}(x)w_{\varepsilon}(x)}{|x-y|^{\alpha}}dydx\\
         &\quad-(2^*_{\alpha}-1)\int_{\Omega}\int_{\Omega}\frac{\bar{U}_{\xi_{\varepsilon},\lambda_{\varepsilon}}^{2^*_{\alpha}}(y)\bar{U}_{\xi_{\varepsilon},\lambda_{\varepsilon}}^{2^*_{\alpha}-2}(x)w_{\varepsilon}^{2}(x)}{|x-y|^{\alpha}}dydx+o(\|w\|_{H^{1}_{0}}^{2})\\
         &:=Q(w_{\varepsilon},w_{\varepsilon})+o(\|w_{\varepsilon}\|_{H^{1}_{0}}^{2}).
\end{aligned}   
\end{equation}
Notice that the bubble $\bar{U}_{\xi_{\varepsilon},\lambda_{\varepsilon}}$ is nondegenerate as stated in Theorem B, then we can prove that $Q$ is coercive, uniformly with respect to $\varepsilon$, on $E_{\xi_{\varepsilon},\lambda_{\varepsilon}}$ (see \cite[Lemma 3.4]{Yang2023BlowUp}), that is,  there exists some constant $c>0$ independent on $\varepsilon$ such that for $\varepsilon>0$ small enough
\begin{equation}
    Q(w,w)\geq c\|w\|_{H^{1}_{0}}^{2}\text{~~for any~~}w\in E_{\xi_{\varepsilon},\lambda_{\varepsilon}}.
\end{equation}
Therefore
\begin{equation}\label{lemma-5-proof-6}
    \|w\|_{H^{1}_{0}}^{2}\lesssim f_{\varepsilon}(w_{\varepsilon}).
\end{equation}
We observe that
\begin{equation}
    \begin{aligned}
        f_{\varepsilon}(w_{\varepsilon})&=\alpha_{\varepsilon}^{2(2^*_{\alpha}+\varepsilon-1)}\int_{\Omega}\int_{\Omega}\frac{(\bar{U}_{\xi_{\varepsilon},\lambda_{\varepsilon}}-\bar{\varphi}_{\xi_{\varepsilon},\lambda_{\varepsilon}})^{2^*_{\alpha}+\varepsilon}(y)(\bar{U}_{\xi_{\varepsilon},\lambda_{\varepsilon}}-\bar{\varphi}_{\xi_{\varepsilon},\lambda_{\varepsilon}})^{2^*_{\alpha}+\varepsilon-1}(x)w_{\varepsilon}(x)}{|x-y|^{\alpha}}dydx\\
        &=\alpha_{\varepsilon}^{2(2^*_{\alpha}+\varepsilon-1)}\int_{\Omega}\int_{\Omega}\frac{\bar{U}_{\xi_{\varepsilon},\lambda_{\varepsilon}}^{2^*_{\alpha}+\varepsilon}(y)\bar{U}_{\xi_{\varepsilon},\lambda_{\varepsilon}}^{2^*_{\alpha}+\varepsilon-1}(x)w_{\varepsilon}(x)}{|x-y|^{\alpha}}dydx\\
        &\quad+O\left(\int_{\Omega}\int_{\Omega}\frac{\bar{U}_{\xi_{\varepsilon},\lambda_{\varepsilon}}^{2^*_{\alpha}+\varepsilon}(y)\bar{U}_{\xi_{\varepsilon},\lambda_{\varepsilon}}^{2^*_{\alpha}+\varepsilon-2}(x)\bar{\varphi}_{\xi_{\varepsilon},\lambda_{\varepsilon}}(x)|w|_{\varepsilon}(x)}{|x-y|^{\alpha}}dydx\right)\\
        &\quad+O\left(\int_{\Omega}\int_{\Omega}\frac{\bar{U}_{\xi_{\varepsilon},\lambda_{\varepsilon}}^{2^*_{\alpha}+\varepsilon-1}\bar{\varphi}_{\xi_{\varepsilon},\lambda_{\varepsilon}}(y)\bar{U}_{\xi_{\varepsilon},\lambda_{\varepsilon}}^{2^*_{\alpha}+\varepsilon-1}(x)|w|_{\varepsilon}(x)}{|x-y|^{\alpha}}dydx\right).
    \end{aligned}
\end{equation}
Moreover, using \eqref{eq orthogonal decomposition}, Lemma \ref{lemma-2}, Corollary \ref{cor-1} and Lemma \ref{lemma-4} 
\begin{equation}
    \begin{aligned}
&\int_{\Omega}\int_{\Omega}\frac{\bar{U}_{\xi_{\varepsilon},\lambda_{\varepsilon}}^{2^*_{\alpha}+\varepsilon}(y)\bar{U}_{\xi_{\varepsilon},\lambda_{\varepsilon}}^{2^*_{\alpha}+\varepsilon-1}(x)w_{\varepsilon}(x)}{|x-y|^{\alpha}}dydx\\
&=\bar{C}_{N,\alpha}^{2\varepsilon}\lambda_{\varepsilon}^{(N-2)\varepsilon}\int_{\Omega}\int_{\Omega}\frac{\bar{U}_{\xi_{\varepsilon},\lambda_{\varepsilon}}^{2^*_{\alpha}}(y)\bar{U}_{\xi_{\varepsilon},\lambda_{\varepsilon}}^{2^*_{\alpha}-1}(x)w_{\varepsilon}(x)}{|x-y|^{\alpha}}dydx\\
&\quad+O\left(\varepsilon \int_{\R^{N}}\int_{\Omega}\frac{\log(1+\lambda_{\varepsilon}^{2}|y-\xi_{\varepsilon}|^{2})\bar{U}_{\xi_{\varepsilon},\lambda_{\varepsilon}}^{2^*_{\alpha}}(y)\bar{U}_{\xi_{\varepsilon},\lambda_{\varepsilon}}^{2^*_{\alpha}-1}(x)|w|_{\varepsilon}(x)}{|x-y|^{\alpha}}dydx\right)\\
&\quad+O\left(\varepsilon \int_{\R^{N}}\int_{\Omega}\frac{\bar{U}_{\xi_{\varepsilon},\lambda_{\varepsilon}}^{2^*_{\alpha}}(y)\bar{U}_{\xi_{\varepsilon},\lambda_{\varepsilon}}^{2^*_{\alpha}-1}(x)\log(1+\lambda_{\varepsilon}^{2}|x-\xi_{\varepsilon}|^{2})|w|_{\varepsilon}(x)}{|x-y|^{\alpha}}dydx\right)\\
&\quad+O\left(\varepsilon^{2} \int_{\R^{N}}\int_{\Omega}\frac{\log(1+\lambda_{\varepsilon}^{2}|y-\xi_{\varepsilon}|^{2})\bar{U}_{\xi_{\varepsilon},\lambda_{\varepsilon}}^{2^*_{\alpha}}(y)\bar{U}_{\xi_{\varepsilon},\lambda_{\varepsilon}}^{2^*_{\alpha}-1}(x)\log(1+\lambda_{\varepsilon}^{2}|x-\xi_{\varepsilon}|^{2})|w|_{\varepsilon}(x)}{|x-y|^{\alpha}}dydx\right).
    \end{aligned}
\end{equation}
Thus, by HLS inequality, Lemma \ref{estimate of U-lambda-a and psi-lambda-a 4} and condition $\alpha<4$
\begin{equation}
\begin{aligned}
     &f_{\varepsilon}(w)\\
     &\lesssim \varepsilon\|w_{\varepsilon}\|_{H^{1}_{0}}+\frac{\|w_{\varepsilon}\|_{H^{1}_{0}}}{(\lambda_{\varepsilon}d_{\varepsilon}^{2})^{\frac{N-2}{2}}}\left(\int_{\Omega}\bar{U}_{\xi_{\varepsilon},\lambda_{\varepsilon}}^{\frac{2^*(2^*_{\alpha}-2)}{2^*_{\alpha}-1}}dx\right)^{\frac{2^*_{\alpha}-1}{2^*}}+\frac{\|w_{\varepsilon}\|_{H^{1}_{0}}}{(\lambda_{\varepsilon}d_{\varepsilon}^{2})^{\frac{N-2}{2}}}\left(\int_{\Omega}\bar{U}_{\xi_{\varepsilon},\lambda_{\varepsilon}}^{\frac{2^*(2^*_{\alpha}-1)}{2^*_{\alpha}}}dx\right)^{\frac{2^*_{\alpha}}{2^*}}\\
    &\lesssim \varepsilon \|w_{\varepsilon}\|_{H^{1}_{0}}+\frac{1}{(\lambda_{\varepsilon}d_{\varepsilon})^{N-2}}\|w_{\varepsilon}\|_{H^{1}_{0}}+\begin{cases}
        \frac{1}{(\lambda_{\varepsilon}d_{\varepsilon})^{N-2}}\|w_{\varepsilon}\|_{H^{1}_{0}}&\text{~~if~~}N<6-\alpha,\\
        \frac{(\log(\lambda_{\varepsilon}d_{\varepsilon}))^{\frac{4-\alpha}{6-\alpha}}}{(\lambda_{\varepsilon}d_{\varepsilon})^{4-\alpha}}\|w_{\varepsilon}\|_{H^{1}_{0}}&\text{~~if~~}N=6-\alpha,\\
        \frac{1}{(\lambda_{\varepsilon}d_{\varepsilon})^{\frac{N+2-\alpha}{2}}}\|w_{\varepsilon}\|_{H^{1}_{0}}&\text{~~if~~}N>6-\alpha,\\
    \end{cases}\\
    &\lesssim \varepsilon \|w_{\varepsilon}\|_{H^{1}_{0}}+\begin{cases}
        \frac{1}{(\lambda_{\varepsilon}d_{\varepsilon})^{N-2}}\|w_{\varepsilon}\|_{H^{1}_{0}}&\text{~~if~~}N<6-\alpha,\\
        \frac{(\log(\lambda_{\varepsilon}d_{\varepsilon}))^{\frac{4-\alpha}{6-\alpha}}}{(\lambda_{\varepsilon}d_{\varepsilon})^{4-\alpha}}\|w_{\varepsilon}\|_{H^{1}_{0}}&\text{~~if~~}N=6-\alpha,\\
        \frac{1}{(\lambda_{\varepsilon}d_{\varepsilon})^{\frac{N+2-\alpha}{2}}}\|w_{\varepsilon}\|_{H^{1}_{0}}&\text{~~if~~}N>6-\alpha,\\
    \end{cases}
\end{aligned}
\end{equation}
which together with \eqref{lemma-5-proof-6}, we can obtain the desired estimate.
\end{proof}
\end{Lem}

\section{Proof of Theorem \ref{thm-1}}\label{section-prooftheorem}
We first establish the following key proposition for proving Theorem \ref{thm-1}.
\begin{Prop}\label{prop-1}
Assume that $\alpha\in(0,\min\{4,N\})$ and $u_{\varepsilon}$ satisfies the assumption of Theorem \ref{thm-1}, then there exist constants $c_{1}(N,\alpha)>0$ and $c_{2}(N,\alpha)>0$ such that
    \begin{equation}
        \frac{R(\xi_{\varepsilon})}{\lambda_{\varepsilon}^{N-2}}(c_{1}+o(1))+\varepsilon(c_{2}+o(1))=O\left(\frac{1}{(\lambda_{\varepsilon}d_{\varepsilon})^{\frac{2N-\alpha}{2}}}\right),
    \end{equation}
where $\xi_{\varepsilon}$, $\lambda_{\varepsilon}$, $d_{\varepsilon}$ are given in \eqref{eq orthogonal decomposition} and $R(\cdot)$ is the Robin function of domain $\Omega$.
\end{Prop}
\begin{proof}
    Multiplying both sides of \eqref{slightly supercritical choquard equation} by $\lambda_{\varepsilon}\frac{P\bar{U}_{\xi_{\varepsilon},\lambda_{\varepsilon}}}{\partial\lambda}$ and integrating on $\Omega$
    \begin{equation}\label{prop-1-proof-1}
    \begin{aligned}
        -\int_{\Omega}\Delta u_{\varepsilon}\lambda_{\varepsilon}\frac{\partial P\bar{U}_{\xi_{\varepsilon},\lambda_{\varepsilon}}}{\partial\lambda}dx=\int_{\Omega}\int_{\Omega}\frac{u_{\varepsilon}^{2^*_{\alpha}+\varepsilon}(y)u_{\varepsilon}^{2^*_{\alpha}+\varepsilon-1}(x)\lambda_{\varepsilon}\frac{\partial P\bar{U}_{\xi_{\varepsilon},\lambda_{\varepsilon}}}{\partial\lambda}(x)}{|x-y|^{\alpha}}dydx.
    \end{aligned}      
    \end{equation}
Using \eqref{eq orthogonal decomposition} and \eqref{eq condition parameter}, the left-hand side of \eqref{prop-1-proof-1} becomes
\begin{equation}
    \begin{aligned}
        \text{LHS}
        &=-\alpha_{\varepsilon}\int_{\Omega}\Delta(P\bar{U}_{\xi_{\varepsilon},\lambda_{\varepsilon}}+w_{\varepsilon})\lambda_{\varepsilon}\frac{\partial P\bar{U}_{\xi_{\varepsilon},\lambda_{\varepsilon}}}{\partial\lambda}dx\\
        &=-\alpha_{\varepsilon}\int_{\Omega}\Delta P\bar{U}_{\xi_{\varepsilon},\lambda_{\varepsilon}}\lambda_{\varepsilon}\frac{\partial P\bar{U}_{\xi_{\varepsilon},\lambda_{\varepsilon}}}{\partial\lambda}dx+\alpha_{\varepsilon}\lambda_{\varepsilon}\int_{\Omega}\nabla w_{\varepsilon}\cdot\nabla\frac{\partial P\bar{U}_{\xi_{\varepsilon},\lambda_{\varepsilon}}}{\partial\lambda}dx\\
        &=\alpha_{\varepsilon}N(N-2)\bar{C}_{N,\alpha}^{2}\int_{\Omega}{U}_{\xi_{\varepsilon},\lambda_{\varepsilon}}^{2^*-1}\lambda_{\varepsilon}\frac{\partial P{U}_{\xi_{\varepsilon},\lambda_{\varepsilon}}}{\partial\lambda}dx+\alpha_{\varepsilon}\lambda_{\varepsilon}\int_{\Omega}\nabla w_{\varepsilon}\cdot\nabla\frac{\partial P\bar{U}_{\xi_{\varepsilon},\lambda_{\varepsilon}}}{\partial\lambda}dx\\
        &=\alpha_{\varepsilon}N(N-2)\bar{C}_{N,\alpha}^{2}\int_{\Omega}{U}_{\xi_{\varepsilon},\lambda_{\varepsilon}}^{2^*-1}\lambda_{\varepsilon}\frac{\partial{U}_{\xi_{\varepsilon},\lambda_{\varepsilon}}}{\partial\lambda}dx-\alpha_{\varepsilon}N(N-2)\bar{C}_{N,\alpha}^{2}\int_{\Omega}{U}_{\xi_{\varepsilon},\lambda_{\varepsilon}}^{2^*-1}\lambda_{\varepsilon}\frac{\partial{\varphi}_{\xi_{\varepsilon},\lambda_{\varepsilon}}}{\partial\lambda}dx\\
        &=-\alpha_{\varepsilon}N(N-2)\bar{C}_{N,\alpha}^{2}\int_{\R^{N}\setminus\Omega}{U}_{\xi_{\varepsilon},\lambda_{\varepsilon}}^{2^*-1}\lambda_{\varepsilon}\frac{\partial{U}_{\xi_{\varepsilon},\lambda_{\varepsilon}}}{\partial\lambda}d-\alpha_{\varepsilon}N(N-2)\bar{C}_{N,\alpha}^{2}\int_{\Omega}{U}_{\xi_{\varepsilon},\lambda_{\varepsilon}}^{2^*-1}\lambda_{\varepsilon}\frac{\partial{\varphi}_{\xi_{\varepsilon},\lambda_{\varepsilon}}}{\partial\lambda}dx.
    \end{aligned}
\end{equation}
Moreover, using Lemma \ref{estimate of U-lambda-a and psi-lambda-a 1}, Lemma \ref{estimate of U-lambda-a and psi-lambda-a 2} and Lemma \ref{estimate of U-lambda-a and psi-lambda-a 4} we have
\begin{equation}\label{prop-1-proof-3}
    \begin{aligned}
        \text{LHS}&=-\alpha_{\varepsilon}N(N-2)\bar{C}_{N,\alpha}^{2}\lambda_{\varepsilon}\frac{\partial{\varphi}_{\xi_{\varepsilon},\lambda_{\varepsilon}}}{\partial\lambda}(\xi_{\varepsilon})\int_{B(\xi_{\varepsilon},d_{\varepsilon})}{U}_{\xi_{\varepsilon},\lambda_{\varepsilon}}^{2^*-1}dx+O\left(\frac{1}{(\lambda_{\varepsilon}d_{\varepsilon})^{N}}\right)\\
        &\quad+O\left(\frac{1}{\lambda_{\varepsilon}^{\frac{N-2}{2}}d_{\varepsilon}^{N}}\int_{B(\xi_{\varepsilon},d_{\varepsilon})}{U}_{\xi_{\varepsilon},\lambda_{\varepsilon}}^{2^*-1}|x-\xi_{\varepsilon}|^{2}dx\right)\\
        &=\frac{R(\xi_{\varepsilon})}{\lambda_{\varepsilon}^{N-2}}(c_{1}+o(1))+O\left(\frac{\log(\lambda_{\varepsilon}d_{\varepsilon})}{(\lambda_{\varepsilon}d_{\varepsilon})^{N}}\right),
    \end{aligned}
\end{equation}
where 
$c_{1}=c_{1}(N,\alpha)=\frac{N(N-2)^{3}\bar{C}_{N,\alpha}^{2}\omega_{N}}{2}\int_{\R^{N}}\frac{1}{(1+|x|^{2})^{\frac{N+2}{2}}}dx$. Next, using \eqref{eq orthogonal decomposition} and Lemma \ref{lem inequality 2}, the right-hand side of \eqref{prop-1-proof-1} becomes
\begin{equation}\label{prop-1-proof-4}
    \begin{aligned}
        &\text{RHS}\\
        &=\alpha_{\varepsilon}^{2(2^*_{\alpha}+\varepsilon)-1}\int_{\Omega}\int_{\Omega}\frac{(P\bar{U}_{\xi_{\varepsilon},\lambda_{\varepsilon}}+w_{\varepsilon})^{2^*_{\alpha}+\varepsilon}(y)(P\bar{U}_{\xi_{\varepsilon},\lambda_{\varepsilon}}+w_{\varepsilon})^{2^*_{\alpha}+\varepsilon-1}(x)\lambda_{\varepsilon}\frac{\partial P\bar{U}_{\xi_{\varepsilon},\lambda_{\varepsilon}}}{\partial\lambda}(x)}{|x-y|^{\alpha}}dydx\\
        &=\alpha_{\varepsilon}^{2(2^*_{\alpha}+\varepsilon)-1}\int_{\Omega}\int_{\Omega}\frac{P\bar{U}_{\xi_{\varepsilon},\lambda_{\varepsilon}}^{2^*_{\alpha}+\varepsilon}(y)P\bar{U}_{\xi_{\varepsilon},\lambda_{\varepsilon}}^{2^*_{\alpha}+\varepsilon-1}(x)\lambda_{\varepsilon}\frac{\partial P\bar{U}_{\xi_{\varepsilon},\lambda_{\varepsilon}}}{\partial\lambda}(x)}{|x-y|^{\alpha}}dydx\\
          &\quad+\alpha_{\varepsilon}^{2(2^*_{\alpha}+\varepsilon)-1}(2^*_{\alpha}+\varepsilon-1)\int_{\Omega}\int_{\Omega}\frac{P\bar{U}_{\xi_{\varepsilon},\lambda_{\varepsilon}}^{2^*_{\alpha}+\varepsilon}(y)P\bar{U}_{\xi_{\varepsilon},\lambda_{\varepsilon}}^{2^*_{\alpha}+\varepsilon-2}(x)w_{\varepsilon}(x)\lambda_{\varepsilon}\frac{\partial P\bar{U}_{\xi_{\varepsilon},\lambda_{\varepsilon}}}{\partial\lambda}(x)}{|x-y|^{\alpha}}dydx\\
          &\quad+\alpha_{\varepsilon}^{2(2^*_{\alpha}+\varepsilon)-1}(2^*_{\alpha}+\varepsilon)\int_{\Omega}\int_{\Omega}\frac{P\bar{U}_{\xi_{\varepsilon},\lambda_{\varepsilon}}^{2^*_{\alpha}+\varepsilon-1}(y)w_{\varepsilon}(y)P\bar{U}_{\xi_{\varepsilon},\lambda_{\varepsilon}}^{2^*_{\alpha}+\varepsilon-1}(x)\lambda_{\varepsilon}\frac{\partial P\bar{U}_{\xi_{\varepsilon},\lambda_{\varepsilon}}}{\partial\lambda}(x)}{|x-y|^{\alpha}}dydx\\         
          &\quad+O\left(\int_{\Omega}\int_{\Omega}\frac{\bar{U}_{\xi_{\varepsilon},\lambda_{\varepsilon}}^{2^*_{\alpha}+\varepsilon-1}(y)|w_{\varepsilon}|(y)\bar{U}_{\xi_{\varepsilon},\lambda_{\varepsilon}}^{2^*_{\alpha}+\varepsilon-1}(x)|w_{\varepsilon}|(x)}{|x-y|^{\alpha}}dydx\right)\\         &\quad+O\left(\int_{\Omega}\int_{\Omega}\frac{\bar{U}_{\xi_{\varepsilon},\lambda_{\varepsilon}}^{2^*_{\alpha}+\varepsilon}(y)(\bar{U}_{\xi_{\varepsilon},\lambda_{\varepsilon}}^{2^*_{\alpha}+\varepsilon-2}|w_{\varepsilon}|^{2}+\bar{U}_{\xi_{\varepsilon},\lambda_{\varepsilon}}|w_{\varepsilon}|^{2^*_{\alpha}+\varepsilon-1}\chi_{\bar{U}_{\xi_{\varepsilon},\lambda_{\varepsilon}}<|w_{\varepsilon}|})(x)}{|x-y|^{\alpha}}dydx\right)\\
          &\quad+O\left(\int_{\Omega}\int_{\Omega}\frac{\bar{U}_{\xi_{\varepsilon},\lambda_{\varepsilon}}^{2^*_{\alpha}+\varepsilon-1}(y)|w_{\varepsilon}|(y)(\bar{U}_{\xi_{\varepsilon},\lambda_{\varepsilon}}^{2^*_{\alpha}+\varepsilon-2}|w_{\varepsilon}|^{2}+\bar{U}_{\xi_{\varepsilon},\lambda_{\varepsilon}}|w_{\varepsilon}|^{2^*_{\alpha}+\varepsilon-1}\chi_{\bar{U}_{\xi_{\varepsilon},\lambda_{\varepsilon}}<|w_{\varepsilon}|})(x)}{|x-y|^{\alpha}}dydx\right)\\
          &\quad+O\left(\int_{\Omega}\int_{\Omega}\frac{(\bar{U}_{\xi_{\varepsilon},\lambda_{\varepsilon}}^{2^*_{\alpha}+\varepsilon-2}|w_{\varepsilon}|^{2}+|w_{\varepsilon}|^{2^*_{\alpha}+\varepsilon})(y)u_{\varepsilon}^{2^*_{\alpha}+\varepsilon-1}(x)\bar{U}_{\xi_{\varepsilon},\lambda_{\varepsilon}}(x)}{|x-y|^{\alpha}}dydx\right).
    \end{aligned}
\end{equation}
Moreover by HLS inequality, Lemma \ref{lemma-1}, Lemma \ref{lemma-2} and Lemma \ref{lemma-5}, we have
\begin{equation}\label{prop-1-proof-4-1}
    \begin{aligned}
        \text{RHS}   &=\alpha_{\varepsilon}^{2(2^*_{\alpha}+\varepsilon)-1}\int_{\Omega}\int_{\Omega}\frac{P\bar{U}_{\xi_{\varepsilon},\lambda_{\varepsilon}}^{2^*_{\alpha}+\varepsilon}(y)P\bar{U}_{\xi_{\varepsilon},\lambda_{\varepsilon}}^{2^*_{\alpha}+\varepsilon-1}(x)\lambda_{\varepsilon}\frac{\partial P\bar{U}_{\xi_{\varepsilon},\lambda_{\varepsilon}}}{\partial\lambda}(x)}{|x-y|^{\alpha}}dydx\\
          &\quad+\alpha_{\varepsilon}^{2(2^*_{\alpha}+\varepsilon)-1}(2^*_{\alpha}-1)\int_{\Omega}\int_{\Omega}\frac{P\bar{U}_{\xi_{\varepsilon},\lambda_{\varepsilon}}^{2^*_{\alpha}+\varepsilon}(y)P\bar{U}_{\xi_{\varepsilon},\lambda_{\varepsilon}}^{2^*_{\alpha}+\varepsilon-2}(x)w_{\varepsilon}(x)\lambda_{\varepsilon}\frac{\partial P\bar{U}_{\xi_{\varepsilon},\lambda_{\varepsilon}}}{\partial\lambda}(x)}{|x-y|^{\alpha}}dydx\\
          &\quad+\alpha_{\varepsilon}^{2(2^*_{\alpha}+\varepsilon)-1}2^*_{\alpha}\int_{\Omega}\int_{\Omega}\frac{P\bar{U}_{\xi_{\varepsilon},\lambda_{\varepsilon}}^{2^*_{\alpha}+\varepsilon-1}(y)w_{\varepsilon}(y)P\bar{U}_{\xi_{\varepsilon},\lambda_{\varepsilon}}^{2^*_{\alpha}+\varepsilon-1}(x)\lambda_{\varepsilon}\frac{\partial P\bar{U}_{\xi_{\varepsilon},\lambda_{\varepsilon}}}{\partial\lambda}(x)}{|x-y|^{\alpha}}dydx\\    
          &\quad+O(\|w_{\varepsilon}\|_{H^{1}_{0}}^{2})+o(\varepsilon).
    \end{aligned}
\end{equation}
Now, we need to estimate each term on the right-hand side of \eqref{prop-1-proof-4-1}. First, we derive that
\begin{equation}
    \begin{aligned}        &\int_{\Omega}\int_{\Omega}\frac{P\bar{U}_{\xi_{\varepsilon},\lambda_{\varepsilon}}^{2^*_{\alpha}+\varepsilon}(y)P\bar{U}_{\xi_{\varepsilon},\lambda_{\varepsilon}}^{2^*_{\alpha}+\varepsilon-1}(x)\lambda_{\varepsilon}\frac{\partial P\bar{U}_{\xi_{\varepsilon},\lambda_{\varepsilon}}}{\partial\lambda}(x)}{|x-y|^{\alpha}}dydx\\        &=\int_{\Omega}\int_{\Omega}\frac{(\bar{U}_{\xi_{\varepsilon},\lambda_{\varepsilon}}-\bar{\varphi}_{\xi_{\varepsilon},\lambda_{\varepsilon}})^{2^*_{\alpha}+\varepsilon}(y)(\bar{U}_{\xi_{\varepsilon},\lambda_{\varepsilon}}-\bar{\varphi}_{\xi_{\varepsilon},\lambda_{\varepsilon}})^{2^*_{\alpha}+\varepsilon-1}(x)\lambda_{\varepsilon}\frac{\partial P\bar{U}_{\xi_{\varepsilon},\lambda_{\varepsilon}}}{\partial\lambda}(x)}{|x-y|^{\alpha}}dydx\\       &=\int_{\Omega}\int_{\Omega}\frac{\bar{U}_{\xi_{\varepsilon},\lambda_{\varepsilon}}^{2^*_{\alpha}+\varepsilon}(y)\bar{U}_{\xi_{\varepsilon},\lambda_{\varepsilon}}^{2^*_{\alpha}+\varepsilon-1}(x)\lambda_{\varepsilon}\frac{\partial P\bar{U}_{\xi_{\varepsilon},\lambda_{\varepsilon}}}{\partial\lambda}(x)}{|x-y|^{\alpha}}dydx\\
        &\quad-(2^*_{\alpha}+\varepsilon-1)\int_{\Omega}\int_{\Omega}\frac{\bar{U}_{\xi_{\varepsilon},\lambda_{\varepsilon}}^{2^*_{\alpha}+\varepsilon}(y)\bar{U}_{\xi_{\varepsilon},\lambda_{\varepsilon}}^{2^*_{\alpha}+\varepsilon-2}\bar{\varphi}_{\xi_{\varepsilon},\lambda_{\varepsilon}}(x)\lambda_{\varepsilon}\frac{\partial P\bar{U}_{\xi_{\varepsilon},\lambda_{\varepsilon}}}{\partial\lambda}(x)}{|x-y|^{\alpha}}dydx\\
        &\quad-(2^*_{\alpha}+\varepsilon)\int_{\Omega}\int_{\Omega}\frac{\bar{U}_{\xi_{\varepsilon},\lambda_{\varepsilon}}^{2^*_{\alpha}+\varepsilon-1}(y)\bar{\varphi}_{\xi_{\varepsilon},\lambda_{\varepsilon}}(y)\bar{U}_{\xi_{\varepsilon},\lambda_{\varepsilon}}^{2^*_{\alpha}+\varepsilon-1}(x)\lambda_{\varepsilon}\frac{\partial P\bar{U}_{\xi_{\varepsilon},\lambda_{\varepsilon}}}{\partial\lambda}(x)}{|x-y|^{\alpha}}dydx\\                   &\quad+O\left(\int_{\Omega}\int_{\Omega}\frac{\bar{U}_{\xi_{\varepsilon},\lambda_{\varepsilon}}^{2^*_{\alpha}+\varepsilon-1}(y)\bar{\varphi}_{\xi_{\varepsilon},\lambda_{\varepsilon}}(y)\bar{U}_{\xi_{\varepsilon},\lambda_{\varepsilon}}^{2^*_{\alpha}+\varepsilon-1}(x)\bar{\varphi}_{\xi_{\varepsilon},\lambda_{\varepsilon}}(x)}{|x-y|^{\alpha}}dydx\right)\\               
        &\quad+O\left(\int_{\Omega}\int_{\Omega}\frac{(\bar{U}_{\xi_{\varepsilon},\lambda_{\varepsilon}}^{2^*_{\alpha}+\varepsilon-2}\bar{\varphi}_{\xi_{\varepsilon},\lambda_{\varepsilon}}^{2}+\bar{\varphi}_{\xi_{\varepsilon},\lambda_{\varepsilon}}^{2^*_{\alpha}+\varepsilon})(y)\bar{U}_{\xi_{\varepsilon},\lambda_{\varepsilon}}^{2^*_{\alpha}+\varepsilon}(x)}{|x-y|^{\alpha}}dydx\right).
    \end{aligned}
\end{equation}
Notice that, by HLS inequality, Lemma \ref{lemma-4}, Lemma \ref{estimate of U-lambda-a and psi-lambda-a 1}-Lemma \ref{estimate of U-lambda-a and psi-lambda-a 4} and condition $\alpha\in(0,4)$
\begin{equation}
\begin{aligned}
     &\int_{\Omega}\int_{\Omega}\frac{\bar{U}_{\xi_{\varepsilon},\lambda_{\varepsilon}}^{2^*_{\alpha}+\varepsilon}(y)\bar{U}_{\xi_{\varepsilon},\lambda_{\varepsilon}}^{2^*_{\alpha}+\varepsilon-2}\bar{\varphi}_{\xi_{\varepsilon},\lambda_{\varepsilon}}(x)\lambda_{\varepsilon}\frac{\partial P\bar{U}_{\xi_{\varepsilon},\lambda_{\varepsilon}}}{\partial\lambda}(x)}{|x-y|^{\alpha}}dydx\\
     &\quad+\int_{\Omega}\int_{\Omega}\frac{\bar{U}_{\xi_{\varepsilon},\lambda_{\varepsilon}}^{2^*_{\alpha}+\varepsilon-1}(y)\bar{\varphi}_{\xi_{\varepsilon},\lambda_{\varepsilon}}(y)\bar{U}_{\xi_{\varepsilon},\lambda_{\varepsilon}}^{2^*_{\alpha}+\varepsilon-1}(x)\lambda_{\varepsilon}\frac{\partial P\bar{U}_{\xi_{\varepsilon},\lambda_{\varepsilon}}}{\partial\lambda}(x)}{|x-y|^{\alpha}}dydx\\
     &=O\left(\frac{1}{(\lambda_{\varepsilon}d_{\varepsilon})^{N-2}}\right),
\end{aligned} 
\end{equation}
\begin{equation}
    \begin{aligned}
        &\int_{\Omega}\int_{\Omega}\frac{\bar{U}_{\xi_{\varepsilon},\lambda_{\varepsilon}}^{2^*_{\alpha}+\varepsilon-1}(y)\bar{\varphi}_{\xi_{\varepsilon},\lambda_{\varepsilon}}(y)\bar{U}_{\xi_{\varepsilon},\lambda_{\varepsilon}}^{2^*_{\alpha}+\varepsilon-1}(x)\bar{\varphi}_{\xi_{\varepsilon},\lambda_{\varepsilon}}(x)}{|x-y|^{\alpha}}dydx\\
        &=O\left(\frac{1}{(\lambda_{\varepsilon}d_{\varepsilon}^{2})^{N-2}}\left(\int_{\Omega}\bar{U}_{\xi_{\varepsilon},\lambda_{\varepsilon}}^{\frac{2^*(2^*_{\alpha}-1)}{2^*_{\alpha}}}dx\right)^{\frac{22^*_{\alpha}}{2^*}}\right)=O\left(\frac{1}{(\lambda_{\varepsilon}d_{\varepsilon})^{2(N-2)}}\right), 
    \end{aligned}
\end{equation}
and
\begin{equation}
    \begin{aligned}
        &\int_{\Omega}\int_{\Omega}\frac{(\bar{U}_{\xi_{\varepsilon},\lambda_{\varepsilon}}^{2^*_{\alpha}+\varepsilon-2}\bar{\varphi}_{\xi_{\varepsilon},\lambda_{\varepsilon}}^{2}+\bar{\varphi}_{\xi_{\varepsilon},\lambda_{\varepsilon}}^{2^*_{\alpha}+\varepsilon})(y)\bar{U}_{\xi_{\varepsilon},\lambda_{\varepsilon}}^{2^*_{\alpha}+\varepsilon}(x)}{|x-y|^{\alpha}}dydx\\
        &=O\left(\int_{\Omega}\bar{U}_{\xi_{\varepsilon},\lambda_{\varepsilon}}^{2^*-2}(x)\bar{\varphi}_{\xi_{\varepsilon},\lambda_{\varepsilon}}^{2}(x)+\bar{U}_{\xi_{\varepsilon},\lambda_{\varepsilon}}^{2^*-2^*_{\alpha}}(x)\bar{\varphi}_{\xi_{\varepsilon},\lambda_{\varepsilon}}^{2^*_{\alpha}}(x)dx\right)\\
        &=O\left(\frac{1}{(\lambda_{\varepsilon}d_{\varepsilon}^{2})^{N-2}}\int_{\Omega}\bar{U}_{\xi_{\varepsilon},\lambda_{\varepsilon}}^{2^*-2}dx+\frac{1}{(\lambda_{\varepsilon}d_{\varepsilon}^{2})^{\frac{2N-\alpha}{2}}}\int_{\Omega}\bar{U}_{\xi_{\varepsilon},\lambda_{\varepsilon}}^{2^*-2^*_{\alpha}}dx\right)\\
        &=O\left(\frac{1}{(\lambda_{\varepsilon}d_{\varepsilon})^{2(N-2)}}\right)+O\left(\frac{\log(\lambda_{\varepsilon}d_{\varepsilon})}{(\lambda_{\varepsilon}d_{\varepsilon})^{N}}\right).
    \end{aligned}
\end{equation}
Hence
\begin{equation}\label{prop-1-proof-10}
    \begin{aligned}        &\int_{\Omega}\int_{\Omega}\frac{P\bar{U}_{\xi_{\varepsilon},\lambda_{\varepsilon}}^{2^*_{\alpha}+\varepsilon}(y)P\bar{U}_{\xi_{\varepsilon},\lambda_{\varepsilon}}^{2^*_{\alpha}+\varepsilon-1}(x)\lambda_{\varepsilon}\frac{\partial P\bar{U}_{\xi_{\varepsilon},\lambda_{\varepsilon}}}{\partial\lambda}(x)}{|x-y|^{\alpha}}dydx\\            &=\int_{\Omega}\int_{\Omega}\frac{\bar{U}_{\xi_{\varepsilon},\lambda_{\varepsilon}}^{2^*_{\alpha}+\varepsilon}(y)\bar{U}_{\xi_{\varepsilon},\lambda_{\varepsilon}}^{2^*_{\alpha}+\varepsilon-1}(x)\lambda_{\varepsilon}\frac{\partial P\bar{U}_{\xi_{\varepsilon},\lambda_{\varepsilon}}}{\partial\lambda}(x)}{|x-y|^{\alpha}}dydx\\
        &\quad-(2^*_{\alpha}-1)\int_{\Omega}\int_{\Omega}\frac{\bar{U}_{\xi_{\varepsilon},\lambda_{\varepsilon}}^{2^*_{\alpha}+\varepsilon}(y)\bar{U}_{\xi_{\varepsilon},\lambda_{\varepsilon}}^{2^*_{\alpha}+\varepsilon-2}\bar{\varphi}_{\xi_{\varepsilon},\lambda_{\varepsilon}}(x)\lambda_{\varepsilon}\frac{\partial P\bar{U}_{\xi_{\varepsilon},\lambda_{\varepsilon}}}{\partial\lambda}(x)}{|x-y|^{\alpha}}dydx\\
        &\quad-2^*_{\alpha}\int_{\Omega}\int_{\Omega}\frac{\bar{U}_{\xi_{\varepsilon},\lambda_{\varepsilon}}^{2^*_{\alpha}+\varepsilon-1}(y)\bar{\varphi}_{\xi_{\varepsilon},\lambda_{\varepsilon}}(y)\bar{U}_{\xi_{\varepsilon},\lambda_{\varepsilon}}^{2^*_{\alpha}+\varepsilon-1}(x)\lambda_{\varepsilon}\frac{\partial P\bar{U}_{\xi_{\varepsilon},\lambda_{\varepsilon}}}{\partial\lambda}(x)}{|x-y|^{\alpha}}dydx\\                   &\quad+O\left(\frac{1}{(\lambda_{\varepsilon}d_{\varepsilon})^{2(N-2)}}\right)+O\left(\frac{\log(\lambda_{\varepsilon}d_{\varepsilon})}{(\lambda_{\varepsilon}d_{\varepsilon})^{N}}\right)+o(\varepsilon).
    \end{aligned}
\end{equation}
The first term in \eqref{prop-1-proof-10} can be estimated as
\begin{equation}
    \begin{aligned}
        &\int_{\Omega}\int_{\Omega}\frac{\bar{U}_{\xi_{\varepsilon},\lambda_{\varepsilon}}^{2^*_{\alpha}+\varepsilon}(y)\bar{U}_{\xi_{\varepsilon},\lambda_{\varepsilon}}^{2^*_{\alpha}+\varepsilon-1}(x)\lambda_{\varepsilon}\frac{\partial P\bar{U}_{\xi_{\varepsilon},\lambda_{\varepsilon}}}{\partial\lambda}(x)}{|x-y|^{\alpha}}dydx\\
        &=\int_{B(\xi_{\varepsilon},d_{\varepsilon})}\int_{B(\xi_{\varepsilon},d_{\varepsilon})}\frac{\bar{U}_{\xi_{\varepsilon},\lambda_{\varepsilon}}^{2^*_{\alpha}+\varepsilon}(y)\bar{U}_{\xi_{\varepsilon},\lambda_{\varepsilon}}^{2^*_{\alpha}+\varepsilon-1}(x)\lambda_{\varepsilon}\frac{\partial P\bar{U}_{\xi_{\varepsilon},\lambda_{\varepsilon}}}{\partial\lambda}(x)}{|x-y|^{\alpha}}dydx\\ 
        &\quad+\int_{\Omega\setminus B(\xi_{\varepsilon},d_{\varepsilon})}\int_{B(\xi_{\varepsilon},d_{\varepsilon})}\frac{\bar{U}_{\xi_{\varepsilon},\lambda_{\varepsilon}}^{2^*_{\alpha}+\varepsilon}(y)\bar{U}_{\xi_{\varepsilon},\lambda_{\varepsilon}}^{2^*_{\alpha}+\varepsilon-1}(x)\lambda_{\varepsilon}\frac{\partial P\bar{U}_{\xi_{\varepsilon},\lambda_{\varepsilon}}}{\partial\lambda}(x)}{|x-y|^{\alpha}}dydx\\
        &\quad+\int_{\Omega}\int_{\Omega\setminus B(\xi_{\varepsilon},d_{\varepsilon})}\frac{\bar{U}_{\xi_{\varepsilon},\lambda_{\varepsilon}}^{2^*_{\alpha}+\varepsilon}(y)\bar{U}_{\xi_{\varepsilon},\lambda_{\varepsilon}}^{2^*_{\alpha}+\varepsilon-1}(x)\lambda_{\varepsilon}\frac{\partial P\bar{U}_{\xi_{\varepsilon},\lambda_{\varepsilon}}}{\partial\lambda}(x)}{|x-y|^{\alpha}}dydx\\   &=\int_{B(\xi_{\varepsilon},d_{\varepsilon})}\int_{B(\xi_{\varepsilon},d_{\varepsilon})}\frac{\bar{U}_{\xi_{\varepsilon},\lambda_{\varepsilon}}^{2^*_{\alpha}+\varepsilon}(y)\bar{U}_{\xi_{\varepsilon},\lambda_{\varepsilon}}^{2^*_{\alpha}+\varepsilon-1}(x)\lambda_{\varepsilon}\frac{\partial P\bar{U}_{\xi_{\varepsilon},\lambda_{\varepsilon}}}{\partial\lambda}(x)}{|x-y|^{\alpha}}dydx+O\left(\frac{1}{(\lambda_{\varepsilon}d_{\varepsilon})^{\frac{2N-\alpha}{2}}}\right)\\     
        &=\int_{B(\xi_{\varepsilon},d_{\varepsilon})}\int_{B(\xi_{\varepsilon},d_{\varepsilon})}\frac{\bar{U}_{\xi_{\varepsilon},\lambda_{\varepsilon}}^{2^*_{\alpha}+\varepsilon}(y)\bar{U}_{\xi_{\varepsilon},\lambda_{\varepsilon}}^{2^*_{\alpha}+\varepsilon-1}(x)\lambda_{\varepsilon}\frac{\partial \bar{U}_{\xi_{\varepsilon},\lambda_{\varepsilon}}}{\partial\lambda}(x)}{|x-y|^{\alpha}}dydx\\        
        &\quad-\int_{B(\xi_{\varepsilon},d_{\varepsilon})}\int_{B(\xi_{\varepsilon},d_{\varepsilon})}\frac{\bar{U}_{\xi_{\varepsilon},\lambda_{\varepsilon}}^{2^*_{\alpha}+\varepsilon}(y)\bar{U}_{\xi_{\varepsilon},\lambda_{\varepsilon}}^{2^*_{\alpha}+\varepsilon-1}(x)\lambda_{\varepsilon}\frac{\partial \bar{\varphi}_{\xi_{\varepsilon},\lambda_{\varepsilon}}}{\partial\lambda}(x)}{|x-y|^{\alpha}}dydx+O\left(\frac{1}{(\lambda_{\varepsilon}d_{\varepsilon})^{\frac{2N-\alpha}{2}}}\right).  
    \end{aligned}
\end{equation}
Moreover, we obtain that
\begin{equation}
    \begin{aligned}
        &\int_{B(\xi_{\varepsilon},d_{\varepsilon})}\int_{B(\xi_{\varepsilon},d_{\varepsilon})}\frac{\bar{U}_{\xi_{\varepsilon},\lambda_{\varepsilon}}^{2^*_{\alpha}+\varepsilon}(y)\bar{U}_{\xi_{\varepsilon},\lambda_{\varepsilon}}^{2^*_{\alpha}+\varepsilon-1}(x)\lambda_{\varepsilon}\frac{\partial \bar{U}_{\xi_{\varepsilon},\lambda_{\varepsilon}}}{\partial\lambda}(x)}{|x-y|^{\alpha}}dydx\\
        &=\bar{C}_{N,\alpha}^{2(2^*_{\alpha}+\varepsilon)}\lambda_{\varepsilon}^{(N-2)\varepsilon}\frac{N-2}{2}\int_{\R^{N}}\int_{\R^{N}}\frac{1-|x|^{2}}{(1+|y|^{2})^{\frac{(N-2)(2^*_{\alpha}+\varepsilon)}{2}}|x-y|^{\alpha}(1+|x|^{2})^{\frac{(N-2)(2^*_{\alpha}+\varepsilon)}{2}+1}}\\
        &\quad+O\left(\frac{1}{(\lambda_{\varepsilon}d_{\varepsilon})^{\frac{2N-\alpha}{2}}}\right).
    \end{aligned}
\end{equation}
Notice that
\begin{equation}
\begin{aligned}
    \int_{\R^{N}}\int_{\R^{N}}&\frac{1-|x|^{2}}{(1+|y|^{2})^{\frac{(2N-\alpha)}{2}}|x-y|^{\alpha}(1+|x|^{2})^{\frac{(2N+2-\alpha)}{2}}}dydx\\
    &=\frac{N(N-2)}{\bar{C}_{N,\alpha}^{22^*_{\alpha}-2}}\int_{\R^{N}}\frac{1-|x|^{2}}{(1+|x|^{2})^{N+1}}dx=0.
\end{aligned}   
\end{equation}
Thus
\begin{equation}
    \begin{aligned}
&\int_{B(\xi_{\varepsilon},d_{\varepsilon})}\int_{B(\xi_{\varepsilon},d_{\varepsilon})}\frac{\bar{U}_{\xi_{\varepsilon},\lambda_{\varepsilon}}^{2^*_{\alpha}+\varepsilon}(y)\bar{U}_{\xi_{\varepsilon},\lambda_{\varepsilon}}^{2^*_{\alpha}+\varepsilon-1}(x)\lambda_{\varepsilon}\frac{\partial \bar{U}_{\xi_{\varepsilon},\lambda_{\varepsilon}}}{\partial\lambda}(x)}{|x-y|^{\alpha}}dydx\\
&=(c_{2}\varepsilon+o(\varepsilon))+O\left(\frac{1}{(\lambda_{\varepsilon}d_{\varepsilon})^{\frac{2N-\alpha}{2}}}\right),
    \end{aligned}
\end{equation}
where $c_{2}=c_{2}(N,\alpha)$ defined by
\begin{equation}
\begin{aligned}
     c_{2}:&=\bar{C}_{N,\alpha}^{22^*_{\alpha}}\frac{(N-2)^{2}}{4} \int_{\R^{N}}\int_{\R^{N}}\frac{\log(1+|y|^{2})(|x|^{2}-1)}{(1+|y|^{2})^{\frac{(2N-\alpha)}{2}}|x-y|^{\alpha}(1+|x|^{2})^{\frac{(2N+2-\alpha)}{2}}}dydx\\
     &\quad+\bar{C}_{N,\alpha}^{22^*_{\alpha}}\frac{(N-2)^{2}}{4} \int_{\R^{N}}\int_{\R^{N}}\frac{\log(1+|x|^{2})(|x|^{2}-1)}{(1+|y|^{2})^{\frac{(2N-\alpha)}{2}}|x-y|^{\alpha}(1+|x|^{2})^{\frac{(2N+2-\alpha)}{2}}}dydx\\
     &=\bar{C}_{N,\alpha}^{2}\frac{N(N-2)^{3}}{4}\frac{\alpha}{2N-\alpha}\int_{\R^{N}}\frac{\log(1+|x|^{2}}{(1+|x|^{2})^{N}}\frac{|x|^{2}-1}{|x|^{2}+1}dx\\
     &\quad+\bar{C}_{N,\alpha}^{2}\frac{N(N-2)^{3}}{4}\int_{\R^{N}}\frac{\log(1+|x|^{2}}{(1+|x|^{2})^{N}}\frac{|x|^{2}-1}{|x|^{2}+1}dx\\
     &=\frac{N^{2}(N-2)^{2}\bar{C}_{N,\alpha}^{2}}{22^*_{\alpha}}\int_{\R^{N}}\frac{\log(1+|x|^{2}}{(1+|x|^{2})^{N}}\frac{|x|^{2}-1}{|x|^{2}+1}dx.
\end{aligned}
\end{equation}
On the other hand, by Taylor's expansion
\begin{equation}
    \begin{aligned}
        &\int_{B(\xi_{\varepsilon},d_{\varepsilon})}\int_{B(\xi_{\varepsilon},d_{\varepsilon})}\frac{\bar{U}_{\xi_{\varepsilon},\lambda_{\varepsilon}}^{2^*_{\alpha}+\varepsilon}(y)\bar{U}_{\xi_{\varepsilon},\lambda_{\varepsilon}}^{2^*_{\alpha}+\varepsilon-1}(x)\lambda_{\varepsilon}\frac{\partial \bar{\varphi}_{\xi_{\varepsilon},\lambda_{\varepsilon}}}{\partial\lambda}(x)}{|x-y|^{\alpha}}dydx\\
        &=\lambda_{\varepsilon}\frac{\partial \bar{\varphi}_{\xi_{\varepsilon},\lambda_{\varepsilon}}}{\partial\lambda}(\xi_{\varepsilon})\int_{B(\xi_{\varepsilon},d_{\varepsilon})}\int_{B(\xi_{\varepsilon},d_{\varepsilon})}\frac{\bar{U}_{\xi_{\varepsilon},\lambda_{\varepsilon}}^{2^*_{\alpha}+\varepsilon}(y)\bar{U}_{\xi_{\varepsilon},\lambda_{\varepsilon}}^{2^*_{\alpha}+\varepsilon-1}(x)}{|x-y|^{\alpha}}dydx\\
        &\quad+O\left(\frac{1}{\lambda_{\varepsilon}^{\frac{N-2}{2}}d_{\varepsilon}^{N}}\int_{B(\xi_{\varepsilon},d_{\varepsilon})}\int_{B(\xi_{\varepsilon},d_{\varepsilon})}\frac{\bar{U}_{\xi_{\varepsilon},\lambda_{\varepsilon}}^{2^*_{\alpha}+\varepsilon}(y)\bar{U}_{\xi_{\varepsilon},\lambda_{\varepsilon}}^{2^*_{\alpha}+\varepsilon-1}(x)|x-\xi_{\varepsilon}|^{2}}{|x-y|^{\alpha}}dydx\right)\\
        &=-\frac{R(\xi_{\varepsilon})}{\lambda_{\varepsilon}^{N-2}}(c_{1}+o(1))+O\left(\frac{\log(\lambda_{\varepsilon}d_{\varepsilon})}{(\lambda_{\varepsilon}d_{\varepsilon})^{N}}\right)+o(\varepsilon).
    \end{aligned}
\end{equation}
Hence the first term in \eqref{prop-1-proof-10} becomes
\begin{equation}\label{prop-1-prop-61}
    \begin{aligned}
        &\int_{\Omega}\int_{\Omega}\frac{\bar{U}_{\xi_{\varepsilon},\lambda_{\varepsilon}}^{2^*_{\alpha}+\varepsilon}(y)\bar{U}_{\xi_{\varepsilon},\lambda_{\varepsilon}}^{2^*_{\alpha}+\varepsilon-1}(x)\lambda_{\varepsilon}\frac{\partial P\bar{U}_{\xi_{\varepsilon},\lambda_{\varepsilon}}}{\partial\lambda}(x)}{|x-y|^{\alpha}}dydx\\
        &=c_{2}\varepsilon(1+o(1))+\frac{R(\xi_{\varepsilon})}{\lambda_{\varepsilon}^{N-2}}(c_{1}+o(1))+O\left(\frac{1}{(\lambda_{\varepsilon}d_{\varepsilon})^{\frac{2N-\alpha}{2}}}\right).
    \end{aligned}
\end{equation}
The second term in \eqref{prop-1-proof-10} can be estimated by
\begin{equation}
    \begin{aligned}
        &\int_{\Omega}\int_{\Omega}\frac{\bar{U}_{\xi_{\varepsilon},\lambda_{\varepsilon}}^{2^*_{\alpha}+\varepsilon}(y)\bar{U}_{\xi_{\varepsilon},\lambda_{\varepsilon}}^{2^*_{\alpha}+\varepsilon-2}(x)\bar{\varphi}_{\xi_{\varepsilon},\lambda_{\varepsilon}}(x)\lambda_{\varepsilon}\frac{\partial P\bar{U}_{\xi_{\varepsilon},\lambda_{\varepsilon}}}{\partial\lambda}(x)}{|x-y|^{\alpha}}dydx\\
        &=\int_{B(\xi_{\varepsilon},d_{\varepsilon})}\int_{B(\xi_{\varepsilon},d_{\varepsilon})}\frac{\bar{U}_{\xi_{\varepsilon},\lambda_{\varepsilon}}^{2^*_{\alpha}+\varepsilon}(y)\bar{U}_{\xi_{\varepsilon},\lambda_{\varepsilon}}^{2^*_{\alpha}+\varepsilon-2}(x)\bar{\varphi}_{\xi_{\varepsilon},\lambda_{\varepsilon}}(x)\lambda_{\varepsilon}\frac{\partial P\bar{U}_{\xi_{\varepsilon},\lambda_{\varepsilon}}}{\partial\lambda}(x)}{|x-y|^{\alpha}}dydx\\
        &\quad+\int_{\Omega\setminus B(\xi_{\varepsilon},d_{\varepsilon})}\int_{B(\xi_{\varepsilon},d_{\varepsilon})}\frac{\bar{U}_{\xi_{\varepsilon},\lambda_{\varepsilon}}^{2^*_{\alpha}+\varepsilon}(y)\bar{U}_{\xi_{\varepsilon},\lambda_{\varepsilon}}^{2^*_{\alpha}+\varepsilon-2}(x)\bar{\varphi}_{\xi_{\varepsilon},\lambda_{\varepsilon}}(x)\lambda_{\varepsilon}\frac{\partial P\bar{U}_{\xi_{\varepsilon},\lambda_{\varepsilon}}}{\partial\lambda}(x)}{|x-y|^{\alpha}}dydx\\
        &\quad+\int_{\Omega}\int_{\Omega\setminus B(\xi_{\varepsilon},d_{\varepsilon})}\frac{\bar{U}_{\xi_{\varepsilon},\lambda_{\varepsilon}}^{2^*_{\alpha}+\varepsilon}(y)\bar{U}_{\xi_{\varepsilon},\lambda_{\varepsilon}}^{2^*_{\alpha}+\varepsilon-2}(x)\bar{\varphi}_{\xi_{\varepsilon},\lambda_{\varepsilon}}(x)\lambda_{\varepsilon}\frac{\partial P\bar{U}_{\xi_{\varepsilon},\lambda_{\varepsilon}}}{\partial\lambda}(x)}{|x-y|^{\alpha}}dydx\\
        &=\int_{B(\xi_{\varepsilon},d_{\varepsilon})}\int_{B(\xi_{\varepsilon},d_{\varepsilon})}\frac{\bar{U}_{\xi_{\varepsilon},\lambda_{\varepsilon}}^{2^*_{\alpha}+\varepsilon}(y)\bar{U}_{\xi_{\varepsilon},\lambda_{\varepsilon}}^{2^*_{\alpha}+\varepsilon-2}(x)\bar{\varphi}_{\xi_{\varepsilon},\lambda_{\varepsilon}}(x)\lambda_{\varepsilon}\frac{\partial P\bar{U}_{\xi_{\varepsilon},\lambda_{\varepsilon}}}{\partial\lambda}(x)}{|x-y|^{\alpha}}dydx+O\left(\frac{1}{(\lambda_{\varepsilon}d_{\varepsilon})^{\frac{2N-\alpha}{2}}}\right)\\
        &=\int_{B(\xi_{\varepsilon},d_{\varepsilon})}\int_{B(\xi_{\varepsilon},d_{\varepsilon})}\frac{\bar{U}_{\xi_{\varepsilon},\lambda_{\varepsilon}}^{2^*_{\alpha}+\varepsilon}(y)\bar{U}_{\xi_{\varepsilon},\lambda_{\varepsilon}}^{2^*_{\alpha}+\varepsilon-2}(x)\bar{\varphi}_{\xi_{\varepsilon},\lambda_{\varepsilon}}(x)\lambda_{\varepsilon}\frac{\partial \bar{U}_{\xi_{\varepsilon},\lambda_{\varepsilon}}}{\partial\lambda}(x)}{|x-y|^{\alpha}}dydx\\
        &\quad-\int_{B(\xi_{\varepsilon},d_{\varepsilon})}\int_{B(\xi_{\varepsilon},d_{\varepsilon})}\frac{\bar{U}_{\xi_{\varepsilon},\lambda_{\varepsilon}}^{2^*_{\alpha}+\varepsilon}(y)\bar{U}_{\xi_{\varepsilon},\lambda_{\varepsilon}}^{2^*_{\alpha}+\varepsilon-2}(x)\bar{\varphi}_{\xi_{\varepsilon},\lambda_{\varepsilon}}(x)\lambda_{\varepsilon}\frac{\partial \bar{\varphi}_{\xi_{\varepsilon},\lambda_{\varepsilon}}}{\partial\lambda}(x)}{|x-y|^{\alpha}}dydx+O\left(\frac{1}{(\lambda_{\varepsilon}d_{\varepsilon})^{\frac{2N-\alpha}{2}}}\right)\\
        &=\int_{B(\xi_{\varepsilon},d_{\varepsilon})}\int_{B(\xi_{\varepsilon},d_{\varepsilon})}\frac{\bar{U}_{\xi_{\varepsilon},\lambda_{\varepsilon}}^{2^*_{\alpha}+\varepsilon}(y)\bar{U}_{\xi_{\varepsilon},\lambda_{\varepsilon}}^{2^*_{\alpha}+\varepsilon-2}(x)\bar{\varphi}_{\xi_{\varepsilon},\lambda_{\varepsilon}}(x)\lambda_{\varepsilon}\frac{\partial \bar{U}_{\xi_{\varepsilon},\lambda_{\varepsilon}}}{\partial\lambda}(x)}{|x-y|^{\alpha}}dydx\\
        &\quad+O\left(\frac{1}{(\lambda_{\varepsilon}d_{\varepsilon})^{2(N-2)}}\right)+O\left(\frac{\log(\lambda_{\varepsilon}d_{\varepsilon})}{(\lambda_{\varepsilon}d_{\varepsilon})^{N}}\right)+O\left(\frac{1}{(\lambda_{\varepsilon}d_{\varepsilon})^{\frac{2N-\alpha}{2}}}\right).
    \end{aligned}
\end{equation}
Notice that
\begin{equation}
    \begin{aligned}
        &\int_{B(\xi_{\varepsilon},d_{\varepsilon})}\int_{B(\xi_{\varepsilon},d_{\varepsilon})}\frac{\bar{U}_{\xi_{\varepsilon},\lambda_{\varepsilon}}^{2^*_{\alpha}+\varepsilon}(y)\bar{U}_{\xi_{\varepsilon},\lambda_{\varepsilon}}^{2^*_{\alpha}+\varepsilon-2}(x)\bar{\varphi}_{\xi_{\varepsilon},\lambda_{\varepsilon}}(x)\lambda_{\varepsilon}\frac{\partial \bar{U}_{\xi_{\varepsilon},\lambda_{\varepsilon}}}{\partial\lambda}(x)}{|x-y|^{\alpha}}dydx\\
        &=\bar{\varphi}_{\xi_{\varepsilon},\lambda_{\varepsilon}}(\xi_{\varepsilon})\int_{B(\xi_{\varepsilon},d_{\varepsilon})}\int_{B(\xi_{\varepsilon},d_{\varepsilon})}\frac{\bar{U}_{\xi_{\varepsilon},\lambda_{\varepsilon}}^{2^*_{\alpha}+\varepsilon}(y)\bar{U}_{\xi_{\varepsilon},\lambda_{\varepsilon}}^{2^*_{\alpha}+\varepsilon-2}(x)\lambda_{\varepsilon}\frac{\partial \bar{U}_{\xi_{\varepsilon},\lambda_{\varepsilon}}}{\partial\lambda}(x)}{|x-y|^{\alpha}}dydx\\        &\quad+O\left(\frac{1}{\lambda_{\varepsilon}^{\frac{N-2}{2}}d_{\varepsilon}^{N}}\int_{B(\xi_{\varepsilon},d_{\varepsilon})}\int_{B(\xi_{\varepsilon},d_{\varepsilon})}\frac{\bar{U}_{\xi_{\varepsilon},\lambda_{\varepsilon}}^{2^*_{\alpha}+\varepsilon}(y)\bar{U}_{\xi_{\varepsilon},\lambda_{\varepsilon}}^{2^*_{\alpha}+\varepsilon-1}(x)|x-\xi_{\varepsilon}|^{2}}{|x-y|^{\alpha}}dydx\right)\\
        &=\bar{\varphi}_{\xi_{\varepsilon},\lambda_{\varepsilon}}(\xi_{\varepsilon})\int_{B(\xi_{\varepsilon},d_{\varepsilon})}\int_{B(\xi_{\varepsilon},d_{\varepsilon})}\frac{\bar{U}_{\xi_{\varepsilon},\lambda_{\varepsilon}}^{2^*_{\alpha}+\varepsilon}(y)\bar{U}_{\xi_{\varepsilon},\lambda_{\varepsilon}}^{2^*_{\alpha}+\varepsilon-2}(x)\lambda_{\varepsilon}\frac{\partial \bar{U}_{\xi_{\varepsilon},\lambda_{\varepsilon}}}{\partial\lambda}(x)}{|x-y|^{\alpha}}dydx+O\left(\frac{\log(\lambda_{\varepsilon}d_{\varepsilon})}{(\lambda_{\varepsilon}d_{\varepsilon})^{N}}\right).     
    \end{aligned}
\end{equation}
Hence the second term in \eqref{prop-1-proof-10} becomes
\begin{equation}
    \begin{aligned}
        &-(2^*_{\alpha}-1)\int_{\Omega}\int_{\Omega}\frac{\bar{U}_{\xi_{\varepsilon},\lambda_{\varepsilon}}^{2^*_{\alpha}+\varepsilon}(y)\bar{U}_{\xi_{\varepsilon},\lambda_{\varepsilon}}^{2^*_{\alpha}+\varepsilon-2}\bar{\varphi}_{\xi_{\varepsilon},\lambda_{\varepsilon}}(x)\lambda_{\varepsilon}\frac{\partial P\bar{U}_{\xi_{\varepsilon},\lambda_{\varepsilon}}}{\partial\lambda}(x)}{|x-y|^{\alpha}}dydx\\
        &=-(2^*_{\alpha}-1)\bar{\varphi}_{\xi_{\varepsilon},\lambda_{\varepsilon}}(\xi_{\varepsilon})\int_{B(\xi_{\varepsilon},d_{\varepsilon})}\int_{B(\xi_{\varepsilon},d_{\varepsilon})}\frac{\bar{U}_{\xi_{\varepsilon},\lambda_{\varepsilon}}^{2^*_{\alpha}+\varepsilon}(y)\bar{U}_{\xi_{\varepsilon},\lambda_{\varepsilon}}^{2^*_{\alpha}+\varepsilon-2}(x)\lambda_{\varepsilon}\frac{\partial \bar{U}_{\xi_{\varepsilon},\lambda_{\varepsilon}}}{\partial\lambda}(x)}{|x-y|^{\alpha}}dydx\\
        &\quad+O\left(\frac{1}{(\lambda_{\varepsilon}d_{\varepsilon})^{2(N-2)}}\right)+O\left(\frac{\log(\lambda_{\varepsilon}d_{\varepsilon})}{(\lambda_{\varepsilon}d_{\varepsilon})^{N}}\right)+O\left(\frac{1}{(\lambda_{\varepsilon}d_{\varepsilon})^{\frac{2N-\alpha}{2}}}\right).
    \end{aligned}
\end{equation}
Similarly, the third term in \eqref{prop-1-proof-10} can be estimated by 
\begin{equation}
    \begin{aligned}
        &-2^*_{\alpha}\int_{\Omega}\int_{\Omega}\frac{\bar{U}_{\xi_{\varepsilon},\lambda_{\varepsilon}}^{2^*_{\alpha}+\varepsilon-1}(y)\bar{\varphi}_{\xi_{\varepsilon},\lambda_{\varepsilon}}(y)\bar{U}_{\xi_{\varepsilon},\lambda_{\varepsilon}}^{2^*_{\alpha}+\varepsilon-1}(x)\lambda_{\varepsilon}\frac{\partial P\bar{U}_{\xi_{\varepsilon},\lambda_{\varepsilon}}}{\partial\lambda}(x)}{|x-y|^{\alpha}}dydx\\
        &=-2^*_{\alpha}\int_{ B(\xi_{\varepsilon},\lambda_{\varepsilon})}\int_{B(\xi_{\varepsilon},\lambda_{\varepsilon})}\frac{\bar{U}_{\xi_{\varepsilon},\lambda_{\varepsilon}}^{2^*_{\alpha}+\varepsilon-1}(y)\bar{\varphi}_{\xi_{\varepsilon},\lambda_{\varepsilon}}(y)\bar{U}_{\xi_{\varepsilon},\lambda_{\varepsilon}}^{2^*_{\alpha}+\varepsilon-1}(x)\lambda_{\varepsilon}\frac{\partial P\bar{U}_{\xi_{\varepsilon},\lambda_{\varepsilon}}}{\partial\lambda}(x)}{|x-y|^{\alpha}}dydx\\
        &\quad-2^*_{\alpha}\int_{\Omega\setminus B(\xi_{\varepsilon},\lambda_{\varepsilon})}\int_{B(\xi_{\varepsilon},\lambda_{\varepsilon})}\frac{\bar{U}_{\xi_{\varepsilon},\lambda_{\varepsilon}}^{2^*_{\alpha}+\varepsilon-1}(y)\bar{\varphi}_{\xi_{\varepsilon},\lambda_{\varepsilon}}(y)\bar{U}_{\xi_{\varepsilon},\lambda_{\varepsilon}}^{2^*_{\alpha}+\varepsilon-1}(x)\lambda_{\varepsilon}\frac{\partial P\bar{U}_{\xi_{\varepsilon},\lambda_{\varepsilon}}}{\partial\lambda}(x)}{|x-y|^{\alpha}}dydx\\
        &\quad-2^*_{\alpha}\int_{\Omega}\int_{\Omega\setminus B(\xi_{\varepsilon},\lambda_{\varepsilon})}\frac{\bar{U}_{\xi_{\varepsilon},\lambda_{\varepsilon}}^{2^*_{\alpha}+\varepsilon-1}(y)\bar{\varphi}_{\xi_{\varepsilon},\lambda_{\varepsilon}}(y)\bar{U}_{\xi_{\varepsilon},\lambda_{\varepsilon}}^{2^*_{\alpha}+\varepsilon-1}(x)\lambda_{\varepsilon}\frac{\partial P\bar{U}_{\xi_{\varepsilon},\lambda_{\varepsilon}}}{\partial\lambda}(x)}{|x-y|^{\alpha}}dydx\\
        &=-2^*_{\alpha}\int_{ B(\xi_{\varepsilon},\lambda_{\varepsilon})}\int_{B(\xi_{\varepsilon},\lambda_{\varepsilon})}\frac{\bar{U}_{\xi_{\varepsilon},\lambda_{\varepsilon}}^{2^*_{\alpha}+\varepsilon-1}(y)\bar{\varphi}_{\xi_{\varepsilon},\lambda_{\varepsilon}}(y)\bar{U}_{\xi_{\varepsilon},\lambda_{\varepsilon}}^{2^*_{\alpha}+\varepsilon-1}(x)\lambda_{\varepsilon}\frac{\partial \bar{U}_{\xi_{\varepsilon},\lambda_{\varepsilon}}}{\partial\lambda}(x)}{|x-y|^{\alpha}}dydx\\
        &\quad+2^*_{\alpha}\int_{ B(\xi_{\varepsilon},\lambda_{\varepsilon})}\int_{B(\xi_{\varepsilon},\lambda_{\varepsilon})}\frac{\bar{U}_{\xi_{\varepsilon},\lambda_{\varepsilon}}^{2^*_{\alpha}+\varepsilon-1}(y)\bar{\varphi}_{\xi_{\varepsilon},\lambda_{\varepsilon}}(y)\bar{U}_{\xi_{\varepsilon},\lambda_{\varepsilon}}^{2^*_{\alpha}+\varepsilon-1}(x)\lambda_{\varepsilon}\frac{\partial \bar{\varphi}_{\xi_{\varepsilon},\lambda_{\varepsilon}}}{\partial\lambda}(x)}{|x-y|^{\alpha}}dydx\\
        &\quad-2^*_{\alpha}\int_{\Omega\setminus B(\xi_{\varepsilon},\lambda_{\varepsilon})}\int_{B(\xi_{\varepsilon},\lambda_{\varepsilon})}\frac{\bar{U}_{\xi_{\varepsilon},\lambda_{\varepsilon}}^{2^*_{\alpha}+\varepsilon-1}(y)\bar{\varphi}_{\xi_{\varepsilon},\lambda_{\varepsilon}}(y)\bar{U}_{\xi_{\varepsilon},\lambda_{\varepsilon}}^{2^*_{\alpha}+\varepsilon-1}(x)\lambda_{\varepsilon}\frac{\partial P\bar{U}_{\xi_{\varepsilon},\lambda_{\varepsilon}}}{\partial\lambda}(x)}{|x-y|^{\alpha}}dydx\\
        &\quad-2^*_{\alpha}\int_{\Omega}\int_{\Omega\setminus B(\xi_{\varepsilon},\lambda_{\varepsilon})}\frac{\bar{U}_{\xi_{\varepsilon},\lambda_{\varepsilon}}^{2^*_{\alpha}+\varepsilon-1}(y)\bar{\varphi}_{\xi_{\varepsilon},\lambda_{\varepsilon}}(y)\bar{U}_{\xi_{\varepsilon},\lambda_{\varepsilon}}^{2^*_{\alpha}+\varepsilon-1}(x)\lambda_{\varepsilon}\frac{\partial P\bar{U}_{\xi_{\varepsilon},\lambda_{\varepsilon}}}{\partial\lambda}(x)}{|x-y|^{\alpha}}dydx\\
        &=-2^*_{\alpha}\bar{\varphi}_{\xi_{\varepsilon},\lambda_{\varepsilon}}(\xi_{\varepsilon})\int_{B(\xi_{\varepsilon},d_{\varepsilon})}\int_{B(\xi_{\varepsilon},d_{\varepsilon})}\frac{\bar{U}_{\xi_{\varepsilon},\lambda_{\varepsilon}}^{2^*_{\alpha}+\varepsilon-1}(y)\bar{U}_{\xi_{\varepsilon},\lambda_{\varepsilon}}^{2^*_{\alpha}+\varepsilon-1}(x)\lambda_{\varepsilon}\frac{\partial \bar{U}_{\xi_{\varepsilon},\lambda_{\varepsilon}}}{\partial\lambda}(x)}{|x-y|^{\alpha}}dydx\\
        &\quad+O\left(\frac{1}{(\lambda_{\varepsilon}d_{\varepsilon})^{2(N-2)}}\right)+O\left(\frac{\log(\lambda_{\varepsilon}d_{\varepsilon})}{(\lambda_{\varepsilon}d_{\varepsilon})^{N}}\right)+O\left(\frac{1}{(\lambda_{\varepsilon}d_{\varepsilon})^{\frac{2N-\alpha}{2}}}\right).
    \end{aligned}
\end{equation}
Thus the second term plus the third term in \eqref{prop-1-proof-10} becomes
\begin{equation}\label{prop-1-prop-66}
    \begin{aligned}
&-(2^*_{\alpha}-1)\int_{\Omega}\int_{\Omega}\frac{\bar{U}_{\xi_{\varepsilon},\lambda_{\varepsilon}}^{2^*_{\alpha}+\varepsilon}(y)\bar{U}_{\xi_{\varepsilon},\lambda_{\varepsilon}}^{2^*_{\alpha}+\varepsilon-2}\bar{\varphi}_{\xi_{\varepsilon},\lambda_{\varepsilon}}(x)\lambda_{\varepsilon}\frac{\partial P\bar{U}_{\xi_{\varepsilon},\lambda_{\varepsilon}}}{\partial\lambda}(x)}{|x-y|^{\alpha}}dydx\\
&\quad-2^*_{\alpha}\int_{\Omega}\int_{\Omega}\frac{\bar{U}_{\xi_{\varepsilon},\lambda_{\varepsilon}}^{2^*_{\alpha}+\varepsilon-1}(y)\bar{\varphi}_{\xi_{\varepsilon},\lambda_{\varepsilon}}(y)\bar{U}_{\xi_{\varepsilon},\lambda_{\varepsilon}}^{2^*_{\alpha}+\varepsilon-1}(x)\lambda_{\varepsilon}\frac{\partial P\bar{U}_{\xi_{\varepsilon},\lambda_{\varepsilon}}}{\partial\lambda}(x)}{|x-y|^{\alpha}}dydx\\
&=-(2^*_{\alpha}-1)\bar{\varphi}_{\xi_{\varepsilon},\lambda_{\varepsilon}}(\xi_{\varepsilon})\int_{B(\xi_{\varepsilon},d_{\varepsilon})}\int_{B(\xi_{\varepsilon},d_{\varepsilon})}\frac{\bar{U}_{\xi_{\varepsilon},\lambda_{\varepsilon}}^{2^*_{\alpha}+\varepsilon}(y)\bar{U}_{\xi_{\varepsilon},\lambda_{\varepsilon}}^{2^*_{\alpha}+\varepsilon-2}(x)\lambda_{\varepsilon}\frac{\partial \bar{U}_{\xi_{\varepsilon},\lambda_{\varepsilon}}}{\partial\lambda}(x)}{|x-y|^{\alpha}}dydx\\
&\quad-2^*_{\alpha}\bar{\varphi}_{\xi_{\varepsilon},\lambda_{\varepsilon}}(\xi_{\varepsilon})\int_{B(\xi_{\varepsilon},d_{\varepsilon})}\int_{B(\xi_{\varepsilon},d_{\varepsilon})}\frac{\bar{U}_{\xi_{\varepsilon},\lambda_{\varepsilon}}^{2^*_{\alpha}+\varepsilon-1}(y)\bar{U}_{\xi_{\varepsilon},\lambda_{\varepsilon}}^{2^*_{\alpha}+\varepsilon-1}(x)\lambda_{\varepsilon}\frac{\partial \bar{U}_{\xi_{\varepsilon},\lambda_{\varepsilon}}}{\partial\lambda}(x)}{|x-y|^{\alpha}}dydx\\
        &\quad+O\left(\frac{1}{(\lambda_{\varepsilon}d_{\varepsilon})^{2(N-2)}}\right)+O\left(\frac{\log(\lambda_{\varepsilon}d_{\varepsilon})}{(\lambda_{\varepsilon}d_{\varepsilon})^{N}}\right)+O\left(\frac{1}{(\lambda_{\varepsilon}d_{\varepsilon})^{\frac{2N-\alpha}{2}}}\right)\\
&=\frac{R(\xi_{\varepsilon})}{\lambda_{\varepsilon}^{N-2}}(c_{1}+o(1))+o(\varepsilon)+O\left(\frac{1}{(\lambda_{\varepsilon}d_{\varepsilon})^{2(N-2)}}\right)+O\left(\frac{\log(\lambda_{\varepsilon}d_{\varepsilon})}{(\lambda_{\varepsilon}d_{\varepsilon})^{N}}\right)+O\left(\frac{1}{(\lambda_{\varepsilon}d_{\varepsilon})^{\frac{2N-\alpha}{2}}}\right).       
    \end{aligned}
\end{equation}
Combining \eqref{prop-1-prop-61} and \eqref{prop-1-prop-66} together, the first term of \eqref{prop-1-proof-4-1} can be estimated by
\begin{equation}\label{prop-1-proof-67}
    \begin{aligned}
        &\alpha_{\varepsilon}^{2(2^*_{\alpha}+\varepsilon)-1}\int_{\Omega}\int_{\Omega}\frac{P\bar{U}_{\xi_{\varepsilon},\lambda_{\varepsilon}}^{2^*_{\alpha}+\varepsilon}(y)P\bar{U}_{\xi_{\varepsilon},\lambda_{\varepsilon}}^{2^*_{\alpha}+\varepsilon-1}(x)\lambda_{\varepsilon}\frac{\partial P\bar{U}_{\xi_{\varepsilon},\lambda_{\varepsilon}}}{\partial\lambda}(x)}{|x-y|^{\alpha}}dydx\\
        &=c_{2}\varepsilon(1+o(1))+2\frac{R(\xi_{\varepsilon})}{\lambda_{\varepsilon}^{N-2}}(c_{1}+o(1))\\
        &\quad+O\left(\frac{1}{(\lambda_{\varepsilon}d_{\varepsilon})^{2(N-2)}}\right)+O\left(\frac{1}{(\lambda_{\varepsilon}d_{\varepsilon})^{\frac{2N-\alpha}{2}}}\right)+O\left(\frac{\log(\lambda_{\varepsilon}d_{\varepsilon})}{(\lambda_{\varepsilon}d_{\varepsilon})^{N}}\right).
    \end{aligned}
\end{equation}
Now, we consider the second term and the third term in \eqref{prop-1-proof-4-1}. First, by Lemma \ref{lem inequality 2} we have
\begin{equation}
    \begin{aligned}
        &\int_{\Omega}\int_{\Omega}\frac{P\bar{U}_{\xi_{\varepsilon},\lambda_{\varepsilon}}^{2^*_{\alpha}+\varepsilon}(y)P\bar{U}_{\xi_{\varepsilon},\lambda_{\varepsilon}}^{2^*_{\alpha}+\varepsilon-2}(x)w_{\varepsilon}(x)\lambda_{\varepsilon}\frac{\partial P\bar{U}_{\xi_{\varepsilon},\lambda_{\varepsilon}}}{\partial\lambda}(x)}{|x-y|^{\alpha}}dydx\\
        &=\int_{\Omega}\int_{\Omega}\frac{(\bar{U}_{\xi_{\varepsilon},\lambda_{\varepsilon}}-\bar{\varphi}_{\xi_{\varepsilon},\lambda_{\varepsilon}})^{2^*_{\alpha}+\varepsilon}(y)(\bar{U}_{\xi_{\varepsilon},\lambda_{\varepsilon}}-\bar{\varphi}_{\xi_{\varepsilon},\lambda_{\varepsilon}})^{2^*_{\alpha}+\varepsilon-2}(x)w_{\varepsilon}(x)\lambda_{\varepsilon}\frac{\partial P\bar{U}_{\xi_{\varepsilon},\lambda_{\varepsilon}}}{\partial\lambda}(x)}{|x-y|^{\alpha}}dydx\\
        &=\int_{\Omega}\int_{\Omega}\frac{\bar{U}_{\xi_{\varepsilon},\lambda_{\varepsilon}}^{2^*_{\alpha}+\varepsilon}(y)\bar{U}_{\xi_{\varepsilon},\lambda_{\varepsilon}}^{2^*_{\alpha}+\varepsilon-2}(x)w_{\varepsilon}(x)\lambda_{\varepsilon}\frac{\partial \bar{U}_{\xi_{\varepsilon},\lambda_{\varepsilon}}}{\partial\lambda}(x)}{|x-y|^{\alpha}}dydx\\
        &\quad-\int_{\Omega}\int_{\Omega}\frac{\bar{U}_{\xi_{\varepsilon},\lambda_{\varepsilon}}^{2^*_{\alpha}+\varepsilon}(y)\bar{U}_{\xi_{\varepsilon},\lambda_{\varepsilon}}^{2^*_{\alpha}+\varepsilon-2}(x)w_{\varepsilon}(x)\lambda_{\varepsilon}\frac{\partial \bar{\varphi}_{\xi_{\varepsilon},\lambda_{\varepsilon}}}{\partial\lambda}(x)}{|x-y|^{\alpha}}dydx\\
        &\quad+O\left(\int_{\Omega}\int_{\Omega}\frac{\bar{U}_{\xi_{\varepsilon},\lambda_{\varepsilon}}^{2^*_{\alpha}+\varepsilon}(y)\bar{U}_{\xi_{\varepsilon},\lambda_{\varepsilon}}^{2^*_{\alpha}+\varepsilon-2}(x)\bar{\varphi}_{\xi_{\varepsilon},\lambda_{\varepsilon}}(x)|w_{\varepsilon}|(x)}{|x-y|^{\alpha}}dydx\right)\\
        &\quad+O\left(\int_{\Omega}\int_{\Omega}\frac{\bar{U}_{\xi_{\varepsilon},\lambda_{\varepsilon}}^{2^*_{\alpha}+\varepsilon-1}(x)\bar{\varphi}_{\xi_{\varepsilon},\lambda_{\varepsilon}}(y)\bar{U}_{\xi_{\varepsilon},\lambda_{\varepsilon}}^{2^*_{\alpha}+\varepsilon-1}(x)|w_{\varepsilon}|(x)}{|x-y|^{\alpha}}dydx\right)
    \end{aligned}
\end{equation}
and
\begin{equation}
    \begin{aligned}
        &\int_{\Omega}\int_{\Omega}\frac{P\bar{U}_{\xi_{\varepsilon},\lambda_{\varepsilon}}^{2^*_{\alpha}+\varepsilon-1}(y)w_{\varepsilon}(y)P\bar{U}_{\xi_{\varepsilon},\lambda_{\varepsilon}}^{2^*_{\alpha}+\varepsilon-1}(x)\lambda_{\varepsilon}\frac{\partial P\bar{U}_{\xi_{\varepsilon},\lambda_{\varepsilon}}}{\partial\lambda}(x)}{|x-y|^{\alpha}}dydx\\
        &=\int_{\Omega}\int_{\Omega}\frac{(\bar{U}_{\xi_{\varepsilon},\lambda_{\varepsilon}}-\bar{\varphi}_{\xi_{\varepsilon},\lambda_{\varepsilon}})^{2^*_{\alpha}+\varepsilon-1}(y)w_{\varepsilon}(y)(\bar{U}_{\xi_{\varepsilon},\lambda_{\varepsilon}}-\bar{\varphi}_{\xi_{\varepsilon},\lambda_{\varepsilon}})^{2^*_{\alpha}+\varepsilon-1}(x)\lambda_{\varepsilon}\frac{\partial P\bar{U}_{\xi_{\varepsilon},\lambda_{\varepsilon}}}{\partial\lambda}(x)}{|x-y|^{\alpha}}dydx\\
        &=\int_{\Omega}\int_{\Omega}\frac{\bar{U}_{\xi_{\varepsilon},\lambda_{\varepsilon}}^{2^*_{\alpha}+\varepsilon-1}(y)w_{\varepsilon}(y)\bar{U}_{\xi_{\varepsilon},\lambda_{\varepsilon}}^{2^*_{\alpha}+\varepsilon-1}(x)\lambda_{\varepsilon}\frac{\partial \bar{U}_{\xi_{\varepsilon},\lambda_{\varepsilon}}}{\partial\lambda}(x)}{|x-y|^{\alpha}}dydx\\
        &\quad-\int_{\Omega}\int_{\Omega}\frac{\bar{U}_{\xi_{\varepsilon},\lambda_{\varepsilon}}^{2^*_{\alpha}+\varepsilon-1}(y)w_{\varepsilon}(y)\bar{U}_{\xi_{\varepsilon},\lambda_{\varepsilon}}^{2^*_{\alpha}+\varepsilon-1}(x)\lambda_{\varepsilon}\frac{\partial \bar{\varphi}_{\xi_{\varepsilon},\lambda_{\varepsilon}}}{\partial\lambda}(x)}{|x-y|^{\alpha}}dydx\\
        &\quad+O\left(\int_{\Omega}\int_{\Omega}\frac{\bar{U}_{\xi_{\varepsilon},\lambda_{\varepsilon}}^{2^*_{\alpha}+\varepsilon-1}(y)|w_{\varepsilon}|(y)\bar{U}_{\xi_{\varepsilon},\lambda_{\varepsilon}}^{2^*_{\alpha}+\varepsilon-1}(x)\bar{\varphi}_{\xi_{\varepsilon},\lambda_{\varepsilon}}(x)}{|x-y|^{\alpha}}dydx\right)\\
        &\quad+O\left(\int_{\Omega}\int_{\Omega}\frac{\bar{U}_{\xi_{\varepsilon},\lambda_{\varepsilon}}^{2^*_{\alpha}+\varepsilon-2}(y)\bar{\varphi}_{\xi_{\varepsilon},\lambda_{\varepsilon}}(y)|w_{\varepsilon}|(y)\bar{U}_{\xi_{\varepsilon},\lambda_{\varepsilon}}^{2^*_{\alpha}+\varepsilon}(x)}{|x-y|^{\alpha}}dydx\right).
    \end{aligned}
\end{equation}
Thus
\begin{equation}
    \begin{aligned}
        &(2^*_{\alpha}-1)\int_{\Omega}\int_{\Omega}\frac{P\bar{U}_{\xi_{\varepsilon},\lambda_{\varepsilon}}^{2^*_{\alpha}+\varepsilon}(y)P\bar{U}_{\xi_{\varepsilon},\lambda_{\varepsilon}}^{2^*_{\alpha}+\varepsilon-2}(x)w_{\varepsilon}(x)\lambda_{\varepsilon}\frac{\partial P\bar{U}_{\xi_{\varepsilon},\lambda_{\varepsilon}}}{\partial\lambda}(x)}{|x-y|^{\alpha}}dydx\\
        &\quad+2^*_{\alpha}\int_{\Omega}\int_{\Omega}\frac{P\bar{U}_{\xi_{\varepsilon},\lambda_{\varepsilon}}^{2^*_{\alpha}+\varepsilon-1}(y)w_{\varepsilon}(y)P\bar{U}_{\xi_{\varepsilon},\lambda_{\varepsilon}}^{2^*_{\alpha}+\varepsilon-1}(x)\lambda_{\varepsilon}\frac{\partial P\bar{U}_{\xi_{\varepsilon},\lambda_{\varepsilon}}}{\partial\lambda}(x)}{|x-y|^{\alpha}}dydx\\
        &=(2^*_{\alpha}-1)\int_{\Omega}\int_{\Omega}\frac{\bar{U}_{\xi_{\varepsilon},\lambda_{\varepsilon}}^{2^*_{\alpha}+\varepsilon}(y)\bar{U}_{\xi_{\varepsilon},\lambda_{\varepsilon}}^{2^*_{\alpha}+\varepsilon-2}(x)w_{\varepsilon}(x)\lambda_{\varepsilon}\frac{\partial \bar{U}_{\xi_{\varepsilon},\lambda_{\varepsilon}}}{\partial\lambda}(x)}{|x-y|^{\alpha}}dydx\\
        &\quad+2^*_{\alpha}\int_{\Omega}\int_{\Omega}\frac{\bar{U}_{\xi_{\varepsilon},\lambda_{\varepsilon}}^{2^*_{\alpha}+\varepsilon-1}(y)w_{\varepsilon}(y)\bar{U}_{\xi_{\varepsilon},\lambda_{\varepsilon}}^{2^*_{\alpha}+\varepsilon-1}(x)\lambda_{\varepsilon}\frac{\partial \bar{U}_{\xi_{\varepsilon},\lambda_{\varepsilon}}}{\partial\lambda}(x)}{|x-y|^{\alpha}}dydx\\
        &\quad-(2^*_{\alpha}-1)\int_{\Omega}\int_{\Omega}\frac{\bar{U}_{\xi_{\varepsilon},\lambda_{\varepsilon}}^{2^*_{\alpha}+\varepsilon}(y)\bar{U}_{\xi_{\varepsilon},\lambda_{\varepsilon}}^{2^*_{\alpha}+\varepsilon-2}(x)w_{\varepsilon}(x)\lambda_{\varepsilon}\frac{\partial \bar{\varphi}_{\xi_{\varepsilon},\lambda_{\varepsilon}}}{\partial\lambda}(x)}{|x-y|^{\alpha}}dydx\\
        &\quad-2^*_{\alpha}\int_{\Omega}\int_{\Omega}\frac{\bar{U}_{\xi_{\varepsilon},\lambda_{\varepsilon}}^{2^*_{\alpha}+\varepsilon-1}(y)w_{\varepsilon}(y)\bar{U}_{\xi_{\varepsilon},\lambda_{\varepsilon}}^{2^*_{\alpha}+\varepsilon-1}(x)\lambda_{\varepsilon}\frac{\partial \bar{\varphi}_{\xi_{\varepsilon},\lambda_{\varepsilon}}}{\partial\lambda}(x)}{|x-y|^{\alpha}}dydx\\&\quad+O\left(\int_{\Omega}\int_{\Omega}\frac{\bar{U}_{\xi_{\varepsilon},\lambda_{\varepsilon}}^{2^*_{\alpha}+\varepsilon}(y)\bar{U}_{\xi_{\varepsilon},\lambda_{\varepsilon}}^{2^*_{\alpha}+\varepsilon-2}(x)\bar{\varphi}_{\xi_{\varepsilon},\lambda_{\varepsilon}}(x)|w_{\varepsilon}|(x)}{|x-y|^{\alpha}}dydx\right)\\
        &\quad+O\left(\int_{\Omega}\int_{\Omega}\frac{\bar{U}_{\xi_{\varepsilon},\lambda_{\varepsilon}}^{2^*_{\alpha}+\varepsilon-1}(x)\bar{\varphi}_{\xi_{\varepsilon},\lambda_{\varepsilon}}(y)\bar{U}_{\xi_{\varepsilon},\lambda_{\varepsilon}}^{2^*_{\alpha}+\varepsilon-1}(x)|w_{\varepsilon}|(x)}{|x-y|^{\alpha}}dydx\right).
    \end{aligned}
\end{equation}
Notice that, by \eqref{eq orthogonal decomposition}, HLS inequality and Lemma \ref{lemma-4}
\begin{equation}
    \begin{aligned}
        &(2^*_{\alpha}-1)\int_{\Omega}\int_{\Omega}\frac{\bar{U}_{\xi_{\varepsilon},\lambda_{\varepsilon}}^{2^*_{\alpha}+\varepsilon}(y)\bar{U}_{\xi_{\varepsilon},\lambda_{\varepsilon}}^{2^*_{\alpha}+\varepsilon-2}(x)w_{\varepsilon}(x)\lambda_{\varepsilon}\frac{\partial \bar{U}_{\xi_{\varepsilon},\lambda_{\varepsilon}}}{\partial\lambda}(x)}{|x-y|^{\alpha}}dydx\\
        &\quad+2^*_{\alpha}\int_{\Omega}\int_{\Omega}\frac{\bar{U}_{\xi_{\varepsilon},\lambda_{\varepsilon}}^{2^*_{\alpha}+\varepsilon-1}(y)w_{\varepsilon}(y)\bar{U}_{\xi_{\varepsilon},\lambda_{\varepsilon}}^{2^*_{\alpha}+\varepsilon-1}(x)\lambda_{\varepsilon}\frac{\partial \bar{U}_{\xi_{\varepsilon},\lambda_{\varepsilon}}}{\partial\lambda}(x)}{|x-y|^{\alpha}}dydx\\
        &=\bar{C}_{N,\alpha}^{2\varepsilon}\lambda_{\varepsilon}^{(N-2)\varepsilon}(2^*_{\alpha}-1)\int_{\Omega}\int_{\Omega}\frac{\bar{U}_{\xi_{\varepsilon},\lambda_{\varepsilon}}^{2^*_{\alpha}}(y)\bar{U}_{\xi_{\varepsilon},\lambda_{\varepsilon}}^{2^*_{\alpha}-2}(x)w_{\varepsilon}(x)\lambda_{\varepsilon}\frac{\partial \bar{U}_{\xi_{\varepsilon},\lambda_{\varepsilon}}}{\partial\lambda}(x)}{|x-y|^{\alpha}}dydx\\
        &\quad+\bar{C}_{N,\alpha}^{2\varepsilon}\lambda_{\varepsilon}^{(N-2)\varepsilon}2^*_{\alpha}\int_{\Omega}\int_{\Omega}\frac{\bar{U}_{\xi_{\varepsilon},\lambda_{\varepsilon}}^{2^*_{\alpha}-1}(y)\lambda_{\varepsilon}\frac{\partial \bar{U}_{\xi_{\varepsilon},\lambda_{\varepsilon}}}{\partial\lambda}(y)\bar{U}_{\xi_{\varepsilon},\lambda_{\varepsilon}}^{2^*_{\alpha}-1}(y)w_{\varepsilon}(x)}{|x-y|^{\alpha}}dydx+O(\varepsilon\|w_{\varepsilon}\|_{H^{1}_{0}})\\
        &=\bar{C}_{N,\alpha}^{2\varepsilon}\lambda_{\varepsilon}^{(N-2)\varepsilon}(2^*_{\alpha}-1)\int_{\R^{N}}\int_{\Omega}\frac{\bar{U}_{\xi_{\varepsilon},\lambda_{\varepsilon}}^{2^*_{\alpha}}(y)\bar{U}_{\xi_{\varepsilon},\lambda_{\varepsilon}}^{2^*_{\alpha}-2}(x)w_{\varepsilon}(x)\lambda_{\varepsilon}\frac{\partial \bar{U}_{\xi_{\varepsilon},\lambda_{\varepsilon}}}{\partial\lambda}(x)}{|x-y|^{\alpha}}dydx\\
        &\quad+\bar{C}_{N,\alpha}^{2\varepsilon}\lambda_{\varepsilon}^{(N-2)\varepsilon}2^*_{\alpha}\int_{\R^{N}}\int_{\Omega}\frac{\bar{U}_{\xi_{\varepsilon},\lambda_{\varepsilon}}^{2^*_{\alpha}-1}(y)\lambda_{\varepsilon}\frac{\partial \bar{U}_{\xi_{\varepsilon},\lambda_{\varepsilon}}}{\partial\lambda}(y)\bar{U}_{\xi_{\varepsilon},\lambda_{\varepsilon}}^{2^*_{\alpha}-1}(y)w_{\varepsilon}(x)}{|x-y|^{\alpha}}dydx\\
        &\quad+O\left(\int_{\R^{N}\setminus\Omega}\int_{\Omega}\frac{\bar{U}_{\xi_{\varepsilon},\lambda_{\varepsilon}}^{2^*_{\alpha}}(y)\bar{U}_{\xi_{\varepsilon},\lambda_{\varepsilon}}^{2^*_{\alpha}-1}(x)|w_{\varepsilon}|(x)}{|x-y|^{\alpha}}dydx\right)+O(\varepsilon\|w_{\varepsilon}\|_{H^{1}_{0}})\\
        &=O\left(\int_{\R^{N}\setminus\Omega}\int_{\Omega}\frac{\bar{U}_{\xi_{\varepsilon},\lambda_{\varepsilon}}^{2^*_{\alpha}}(y)\bar{U}_{\xi_{\varepsilon},\lambda_{\varepsilon}}^{2^*_{\alpha}-1}(x)|w_{\varepsilon}|(x)}{|x-y|^{\alpha}}dydx\right)+O(\varepsilon\|w_{\varepsilon}\|_{H^{1}_{0}})\\
        &=O(\varepsilon\|w_{\varepsilon}\|_{H^{1}_{0}})+O\left(\frac{1}{(\lambda_{\varepsilon}d_{\varepsilon})^{\frac{2N-\alpha}{2}}}\|w_{\varepsilon}\|_{H^{1}_{0}}\right),
    \end{aligned}
\end{equation}

\begin{equation}
    \begin{aligned}
        &\int_{\Omega}\int_{\Omega}\frac{\bar{U}_{\xi_{\varepsilon},\lambda_{\varepsilon}}^{2^*_{\alpha}+\varepsilon}(y)\bar{U}_{\xi_{\varepsilon},\lambda_{\varepsilon}}^{2^*_{\alpha}+\varepsilon-2}(x)|w_{\varepsilon}|(x)\lambda_{\varepsilon}\frac{\partial \bar{\varphi}_{\xi_{\varepsilon},\lambda_{\varepsilon}}}{\partial\lambda}(x)}{|x-y|^{\alpha}}dydx\\
        &\quad+\int_{\Omega}\int_{\Omega}\frac{\bar{U}_{\xi_{\varepsilon},\lambda_{\varepsilon}}^{2^*_{\alpha}+\varepsilon}(y)\bar{U}_{\xi_{\varepsilon},\lambda_{\varepsilon}}^{2^*_{\alpha}+\varepsilon-2}(x)\bar{\varphi}_{\xi_{\varepsilon},\lambda_{\varepsilon}}(x)|w_{\varepsilon}|(x)}{|x-y|^{\alpha}}dydx\\
        &=O\left(\frac{\|w_{\varepsilon}\|_{H^{1}_{0}}}{(\lambda_{\varepsilon}d_{\varepsilon}^{2})^{\frac{N-2}{2}}}\left(\int_{\Omega}\bar{U}_{\xi_{\varepsilon},\lambda_{\varepsilon}}^{\frac{2^*(2^*_{\alpha}-2)}{2^*_{\alpha}-1}}dx\right)^{\frac{2^*_{\alpha}-1}{2^*}}\right)\\
        &=\begin{cases}
            O\left(\frac{\|w_{\varepsilon}\|_{H^{1}_{0}}}{(\lambda_{\varepsilon}d_{\varepsilon})^{N-2}}\right)&\text{~~if~~}N<6-\alpha,\\
             O\left(\frac{\log(\lambda_{\varepsilon},d_{\varepsilon})^{\frac{4-\alpha}{6-\alpha}}\|w_{\varepsilon}\|_{H^{1}_{0}}}{(\lambda_{\varepsilon}d_{\varepsilon})^{4-\alpha}}\right)&\text{~~if~~}N=6-\alpha,\\
            O\left(\frac{\|w_{\varepsilon}\|_{H^{1}_{0}}}{(\lambda_{\varepsilon}d_{\varepsilon})^{\frac{N+2-\alpha}{2}}}\right)&\text{~~if~~}N>6-\alpha,\\
        \end{cases}        
    \end{aligned}
\end{equation}
and
\begin{equation}
    \begin{aligned}
        &\int_{\Omega}\int_{\Omega}\frac{\bar{U}_{\xi_{\varepsilon},\lambda_{\varepsilon}}^{2^*_{\alpha}+\varepsilon-1}(y)|w_{\varepsilon}|(y)\bar{U}_{\xi_{\varepsilon},\lambda_{\varepsilon}}^{2^*_{\alpha}+\varepsilon-1}(x)\lambda_{\varepsilon}\frac{\partial \bar{\varphi}_{\xi_{\varepsilon},\lambda_{\varepsilon}}}{\partial\lambda}(x)}{|x-y|^{\alpha}}dydx\\
        &\quad+\int_{\Omega}\int_{\Omega}\frac{\bar{U}_{\xi_{\varepsilon},\lambda_{\varepsilon}}^{2^*_{\alpha}+\varepsilon-1}(x)\bar{\varphi}_{\xi_{\varepsilon},\lambda_{\varepsilon}}(y)\bar{U}_{\xi_{\varepsilon},\lambda_{\varepsilon}}^{2^*_{\alpha}+\varepsilon-1}(x)|w_{\varepsilon}|(x)}{|x-y|^{\alpha}}dydx\\
        &=O\left(\frac{\|w_{\varepsilon}\|_{H^{1}_{0}}}{(\lambda_{\varepsilon}d_{\varepsilon}^{2})^{\frac{N-2}{2}}}\left(\int_{\Omega}\bar{U}_{\xi_{\varepsilon},\lambda_{\varepsilon}}^{\frac{2^*(2^*_{\alpha}-1)}{2^*_{\alpha}}}dx\right)^{\frac{2^*_{\alpha}}{2^*}}\right)=O\left(\frac{\|w_{\varepsilon}\|_{H^{1}_{0}}}{(\lambda_{\varepsilon}d_{\varepsilon})^{N-2}}\right). 
    \end{aligned}
\end{equation}
Hence, the second term plus the third term in \eqref{prop-1-proof-4-1} becomes
\begin{equation}
    \begin{aligned}
         &\alpha_{\varepsilon}^{2(2^*_{\alpha}+\varepsilon)-1}(2^*_{\alpha}-1)\int_{\Omega}\int_{\Omega}\frac{P\bar{U}_{\xi_{\varepsilon},\lambda_{\varepsilon}}^{2^*_{\alpha}+\varepsilon}(y)P\bar{U}_{\xi_{\varepsilon},\lambda_{\varepsilon}}^{2^*_{\alpha}+\varepsilon-2}(x)w_{\varepsilon}(x)\lambda_{\varepsilon}\frac{\partial P\bar{U}_{\xi_{\varepsilon},\lambda_{\varepsilon}}}{\partial\lambda}(x)}{|x-y|^{\alpha}}dydx\\
          &\quad+\alpha_{\varepsilon}^{2(2^*_{\alpha}+\varepsilon)-1}2^*_{\alpha}\int_{\Omega}\int_{\Omega}\frac{P\bar{U}_{\xi_{\varepsilon},\lambda_{\varepsilon}}^{2^*_{\alpha}+\varepsilon-1}(y)w_{\varepsilon}(y)P\bar{U}_{\xi_{\varepsilon},\lambda_{\varepsilon}}^{2^*_{\alpha}+\varepsilon-1}(x)\lambda_{\varepsilon}\frac{\partial P\bar{U}_{\xi_{\varepsilon},\lambda_{\varepsilon}}}{\partial\lambda}(x)}{|x-y|^{\alpha}}dydx\\   
           &=O(\varepsilon\|w_{\varepsilon}\|_{H^{1}_{0}})+O\left(\frac{1}{(\lambda_{\varepsilon}d_{\varepsilon})^{\frac{2N-\alpha}{2}}}\|w_{\varepsilon}\|_{H^{1}_{0}}\right)+O\left(\frac{\|w_{\varepsilon}\|_{H^{1}_{0}}}{(\lambda_{\varepsilon}d_{\varepsilon})^{N-2}}\right)\\
           &\quad+\begin{cases}
            O\left(\frac{\|w_{\varepsilon}\|_{H^{1}_{0}}}{(\lambda_{\varepsilon}d_{\varepsilon})^{N-2}}\right)&\text{~~if~~}N<6-\alpha,\\
             O\left(\frac{\log(\lambda_{\varepsilon},d_{\varepsilon})^{\frac{4-\alpha}{6-\alpha}}\|w_{\varepsilon}\|_{H^{1}_{0}}}{(\lambda_{\varepsilon}d_{\varepsilon})^{4-\alpha}}\right)&\text{~~if~~}N=6-\alpha,\\
            O\left(\frac{\|w_{\varepsilon}\|_{H^{1}_{0}}}{(\lambda_{\varepsilon}d_{\varepsilon})^{\frac{N+2-\alpha}{2}}}\right)&\text{~~if~~}N>6-\alpha,\\
            \end{cases}
    \end{aligned}
\end{equation}
which together with \eqref{prop-1-proof-4-1}, \eqref{prop-1-proof-67} and Lemma \ref{lemma-5}
\begin{equation}
    \begin{aligned}
        \text{RHS} &=c_{2}\varepsilon(1+o(1))+2\frac{R(\xi_{\varepsilon})}{\lambda_{\varepsilon}^{N-2}}(c_{1}+o(1))\\
        &\quad+O\left(\frac{1}{(\lambda_{\varepsilon}d_{\varepsilon})^{2(N-2)}}\right)+O\left(\frac{1}{(\lambda_{\varepsilon}d_{\varepsilon})^{\frac{2N-\alpha}{2}}}\right)+O\left(\frac{\log(\lambda_{\varepsilon}d_{\varepsilon})}{(\lambda_{\varepsilon}d_{\varepsilon})^{N}}\right)\\
        &\quad+O\left(\|w_{\varepsilon}\|_{H^{1}_{0}}^{2}\right)+O\left(\frac{\|w_{\varepsilon}\|_{H^{1}_{0}}}{(\lambda_{\varepsilon}d_{\varepsilon})^{N-2}}\right)+\begin{cases}
            O\left(\frac{\|w_{\varepsilon}\|_{H^{1}_{0}}}{(\lambda_{\varepsilon}d_{\varepsilon})^{N-2}}\right)&\text{~~if~~}N<6-\alpha,\\
             O\left(\frac{\log(\lambda_{\varepsilon},d_{\varepsilon})^{\frac{4-\alpha}{6-\alpha}}\|w_{\varepsilon}\|_{H^{1}_{0}}}{(\lambda_{\varepsilon}d_{\varepsilon})^{4-\alpha}}\right)&\text{~~if~~}N=6-\alpha,\\
            O\left(\frac{\|w_{\varepsilon}\|_{H^{1}_{0}}}{(\lambda_{\varepsilon}d_{\varepsilon})^{\frac{N+2-\alpha}{2}}}\right)&\text{~~if~~}N>6-\alpha,\\
            \end{cases}\\
            &=c_{2}\varepsilon(1+o(1))+2\frac{R(\xi_{\varepsilon})}{\lambda_{\varepsilon}^{N-2}}(c_{1}+o(1))+O\left(\frac{1}{(\lambda_{\varepsilon}d_{\varepsilon})^{\frac{2N-\alpha}{2}}}+\frac{1}{(\lambda_{\varepsilon}d_{\varepsilon})^{2}}(\text{~if~}N=3)\right).
    \end{aligned}
\end{equation}
Moreover, using \eqref{prop-1-proof-3}, we have
\begin{equation}
    \begin{aligned}
        \frac{R(\xi_{\varepsilon})}{\lambda_{\varepsilon}^{N-2}}(c_{1}+o(1))+\varepsilon(c_{2}+o(1))=O\left(\frac{1}{(\lambda_{\varepsilon}d_{\varepsilon})^{\frac{2N-\alpha}{2}}}+\frac{1}{(\lambda_{\varepsilon}d_{\varepsilon})^{2}}(\text{~if~}N=3)\right).
    \end{aligned}
\end{equation}
This completes the proof.
\end{proof}

Now, we are able to prove Theorem \ref{thm-1}.
\begin{proof}{\bf{Proof of Theorem \ref{thm-1}.}}
From Proposition \ref{prop-1}, we know 
\begin{equation}\label{thm-1-proof}
    \frac{R(\xi_{\varepsilon})}{\lambda_{\varepsilon}^{N-2}}(c_{1}+o(1))+ \varepsilon(c_{2}+o(1))=O\left(\frac{1}{(\lambda_{\varepsilon}d_{\varepsilon})^{\frac{2N-\alpha}{2}}}+\frac{1}{(\lambda_{\varepsilon}d_{\varepsilon})^{2}}(\text{~if~}N=3)\right).
\end{equation}
If $d_{\varepsilon}\to0$  as $\varepsilon\to0$, then by \cite[Proposition 6.7.1]{Cao_Peng_Yan_2021}, we have $R(\xi_{\varepsilon})\sim d_{\varepsilon}^{2-N}$, this contradicts \eqref{thm-1-proof}. If $d_{\varepsilon}\not\to0$ as $\varepsilon\to0$, then $R(\xi_{\varepsilon})\sim d_{\varepsilon}\sim 1$, this also contradicts \eqref{thm-1-proof}.
\end{proof}
\begin{ackn}
The research has been partly supported by the Open Research Fund of Key Laboratory of Nonlinear Analysis $\&$ Applications (Central China Normal University), Ministry of Education, P.R. China (NAA2024ORG002).
\end{ackn}

\end{document}